\documentclass[12pt,reqno]{amsart}
\usepackage{mathrsfs,amssymb,graphicx,verbatim,amsmath,amsfonts}
\usepackage{paralist}
\usepackage[breaklinks,pdfstartview=FitH]{hyperref}
\usepackage{upgreek}

\renewcommand{\le}{\leqslant}
\renewcommand{\ge}{\geqslant}
\renewcommand{\leq}{\leqslant}

\renewcommand{\setminus}{\smallsetminus}
\renewcommand{\gamma}{\upgamma}

\newcommand{\B}{\mathcal{B}}
\newcommand{\n}{\{1,\ldots,n\}}
\renewcommand{\L}{\mathscr{L}}
\newcommand{\X}{\mathscr X}
\newcommand{\f}{\varphi}
\newcommand{\p}{\psi}

\newcommand{\T}{\mathfrak{T}}
\renewcommand{\d}{\delta}

\newcommand{\e}{\varepsilon}
\newcommand{\R}{\mathbb R}
\newcommand{\1}{\mathbf 1}
\renewcommand{\a}{\mathfrak{a}}
\renewcommand{\b}{\mathfrak{b}}
\newcommand{\G}{\mathbb G}
\newtheorem{theorem}{Theorem}[section]
\newtheorem{proposition}[theorem]{Proposition}
\newtheorem{lemma}[theorem]{Lemma}
\newtheorem{corollary}[theorem]{Corollary}
\newtheorem{claim}[theorem]{Claim}
\newtheorem{fact}[theorem]{Fact}
\theoremstyle{remark}

\newtheorem{remark}[theorem]{Remark}
\newtheorem{question}[theorem]{Question}
\theoremstyle{definition}
\newtheorem{definition}[theorem]{Definition}
\newcommand{\cone}{\mathrm{\mathbf{Cone}}}
\newcommand{\zigzag}{{\textcircled z}}

\newcommand{\circr}{{\textcircled r}}

\renewcommand{\subset}{\subseteq}

\newcommand{\A}{\mathscr A}
\newcommand{\C}{\mathfrak C}

\newcommand{\F}{\mathcal F}
\newcommand{\N}{\mathbb N}

\newcommand{\eqdef}{\stackrel{\mathrm{def}}{=}}

\newcommand{\Lip}{\mathrm{Lip}}
\newcommand{\oz}{\zigzag}

\DeclareMathOperator{\diam}{diam}
\DeclareMathOperator{\girth}{girth}

\begin{document}

\title[Expanders for Hadamard spaces and random graphs]{Expanders with respect to {H}adamard spaces and random graphs}

\author{Manor Mendel}
\address {Mathematics and Computer Science Department\\
The Open University of Israel\\
1 University
Road, P.O. Box 808 Raanana 43107, Israel}
\email{mendelma@gmail.com}.
\author{Assaf Naor}
\address{Courant Institute\\ New York University\\ 251 Mercer Street, New York NY 10012, USA}
\email{naor@cims.nyu.edu}

\date{}

\vspace{-0.27in}
\begin{abstract}
It is shown that there exists a sequence of $3$-regular graphs
$\{G_n\}_{n=1}^\infty$ and a Hadamard space $X$ such that
$\{G_n\}_{n=1}^\infty$ forms an expander sequence with respect to
$X$, yet random regular graphs are not expanders with respect to
$X$. This answers a question of~\cite{NS11}. $\{G_n\}_{n=1}^\infty$
are also shown to be expanders with respect to random regular
graphs, yielding a deterministic sublinear time constant factor
approximation algorithm for computing the average squared distance
in subsets of a random graph. The proof uses the Euclidean cone over
a random graph, an auxiliary continuous geometric object that allows
for the implementation of martingale methods.
\end{abstract}


\maketitle

\vspace{-0.5in}

{\footnotesize \setcounter{tocdepth}{4} \tableofcontents}

\section{Introduction}

Throughout this paper all graphs are unweighted, non-oriented, and
finite, and they are allowed to have parallel edges and self-loops,
unless stated otherwise. Given a graph $G$, we denote its vertices
by $V_G$ and its edges by $E_G$. If $G$ is connected then we denote
the shortest-path metric that it induces on $V_G$ by $d_G$.


Fix $d\in \N$ and let $\{G_n\}_{n=1}^\infty$ be a sequence
of $d$-regular graphs such that $\lim_{n\to \infty} |V_{G_n}|=\infty$.
Then $\{G_n\}_{n=1}^\infty$ is an expander sequence (see
e.g.~\cite{HLW}) if and only if for every sequence of Hilbert space
valued functions $\{f_n:V_n\to \ell_2\}_{n=1}^\infty$ we have
\begin{multline}\label{eq:compute average}
\frac{1}{|V_{G_n}|^2}\sum_{(x,y)\in V_{G_n}\times V_{G_n}} \|f_n(x)-f_n(y)\|_2^2\\\asymp \frac{1}{|E_{G_n}|}
\sum_{\{u,v\}\in E_{G_n}} \|f_n(u)-f_n(v)\|_2^2.
\end{multline}
Here and in what follows, when we write $A\asymp B$ we mean that
there exist two universal constants $c,C\in (0,\infty)$ such that
$cA\le B\le CB$. Also, in what follows the notations $A\lesssim B$
and $B\gtrsim A$ mean that $A\le KB$ for some universal constant
$K\in (0,\infty)$. If we need to allow $K$ to depend on parameters,
we indicate this by subscripts, thus e.g. $A \lesssim_{\alpha,\beta}
B$ means that $A \leq K(\alpha,\beta) B$ for some
$K(\alpha,\beta)\in (0,\infty)$ which is allowed to depend only on
the parameters $\alpha$ and $\beta$.

The asymptotic identity~\eqref{eq:compute average} says that
whenever one assigns a vector to each vertex of $G_n$, the average
squared distance between these vectors can be estimated up
to universal constant factors by averaging those squared distances
that correspond to edges of $G_n$, an average of only $d|V_{G_n}|$
numbers rather than the full $|V_{G_n}|^2$ pairwise distances.
This
geometric characterization of expanders as  ``universal average
Euclidean distance approximators" (following terminology
of~\cite{BGS}) is easy to prove (its proof
will also be explained in the ensuing discussion).

It is natural to investigate the possible validity
of~\eqref{eq:compute average} when the Hilbertian metric is replaced
other metrics. \eqref{eq:compute average}
comprises of two asymptotic inequalities, one of which holds true in
any metric space: if $G$ is a $d$-regular graph and $(X,d_X)$ is a
metric space then for every $f:V_G\to X$ we have
\begin{equation}\label{eq:upper average bound trivial}
\frac{1}{|E_G|}
\sum_{\{u,v\}\in E_G} d_X\big(f(u),f(v)\big)^2\le \frac{4}{|V_G|^2}\sum_{(x,y)\in V_G\times V_G} d_X\big(f(x),f(y)\big)^2.
\end{equation}
Indeed, by the triangle inequality and the convexity of $t\mapsto
t^2$,

\begin{equation}\label{eq:uvw}
 d_X(f(u),f(v))^2\le
2d_X(f(u),f(w))^2+2d_X(f(w),f(v))^2
\end{equation}
for every $u,v,w\in V_G$. Since $G$ is $d$-regular, the bound~\eqref{eq:upper average bound
trivial} follows by averaging~\eqref{eq:uvw} over the set $\{(u,v,w)\in V_G\times V_G\times V_G:\ \{u,v\}\in
E_G\}$.

The nontrivial content of~\eqref{eq:compute average} is therefore
the fact that the left hand side of~\eqref{eq:compute average} can
be bounded by a multiple (independent of $n$ and $f$) of the right
hand side of~\eqref{eq:compute average}. Thus, given a regular graph
$G$ and a metric space $(X,d_X)$, let $\gamma(G,d_X^2)$ be the
infimum over those $\gamma\in (0,\infty]$ such that for every
$f:V_G\to X$,
\begin{equation}\label{eq:def gamma}
\frac{1}{|V_G|^2}\sum_{(u,v)\in V_G\times V_G}d_X\big(f(u),f(v)\big)^2\le
\frac{\gamma}{|E_G|}\sum_{\{u,v\}\in E_G} d_X\big(f(u),f(v)\big)^2.
\end{equation}

 If $X=\R$ with $d_\R(s,t)\eqdef |s-t|$ then by expanding the squares in~\eqref{eq:def gamma} one checks that
$$
\gamma(G,d_\R^2)=\frac{1}{1-\lambda_2(G)},
$$
where $\lambda_2(G)$ denotes the second largest eigenvalue of the
normalized adjacency matrix of the graph $G$.  Despite the fact that
there is no actual spectrum present in this geometric context, we
think of $\gamma(G,d_X^2)$ as the reciprocal of the spectral gap of
$G$ with respect to $(X,d_X)$. The value of $\gamma(G,d_X^2)$ is
sensitive to the choice of metric space $(X,d_X)$ and it can be very
different from the reciprocal of the spectral gap of $G$. We refer
to~\cite{MN-towards} for more information on nonlinear spectral
gaps.

 Given $d\in \N$, a sequence of $d$-regular graphs $\{G_n\}_{n=1}^\infty$ is said to be an expander sequence with respect to a metric space $(X,d_X)$ if $\lim_{n\to \infty} |V_{G_n}|=\infty$ and $\sup_{n\in \N} \gamma(G_n,d_X^2)<\infty$. Note that by Cheeger's inequality~\cite{Che70,AM85} for every graph $G$ we have
$$
|X|\ge 2\implies \gamma(G,d_X^2)\gtrsim \frac{1}{\sqrt{1-\lambda_2(G)}}.
$$
Hence, unless $X$ is a singleton, if $\{G_n\}_{n=1}^\infty$ is an expander sequence with respect to $(X,d_X)$ then $\sup_{n\in \N} \lambda_2(G_n)<1$. This means that for every nontrivial metric space $(X,d_X)$, being an expander sequence with respect to $(X,d_X)$ is a stronger requirement than being an expander sequence in the classical sense.

Nonlinear spectral gaps first arose in the context of bi-Lipschitz
embeddings; notable examples include the works of
Enflo~\cite{Enf76}, Gromov~\cite{Gromov-filling},  Bourgain, Milman
and Wolfson~\cite{BMW}, Pisier~\cite{Pisier-type}, Linial, London
and Rabinovich~\cite{LLR}, and Matou\v{s}ek~\cite{Mat97} (examples
of more recent applications of nonlinear spectral gaps to
bi-Lipschitz embeddings appear in~\cite{BLMN05,KN06}).
Gromov~\cite{Gromov-random-group} studied nonlinear spectral gaps in
the context of coarse embeddings and the Novikov conjecture, a
direction that has been pursued in the works of Ozawa~\cite{Ozawa},
Kasparov and Yu~\cite{KY06}, V.~Lafforgue~\cite{Laff08,Laf09,Laf10},
Pisier~\cite{pisier-2008}, and ourselves~\cite{MN-towards}.
Nonlinear spectral gaps also arise in fixed point theory for group
actions; see the works of Wang~\cite{Wang98,Wang00},
Gromov~\cite{Gromov-random-group}, Izeki and Nayatani~\cite{IN05},
Pansu~\cite{Pan09}, Naor and Silberman~\cite{NS11}, and Izeki, Kondo
and Nayatani~\cite{IKN12}. We refer to the works of
Gromov~\cite{Gro01} and Pichot~\cite{Pic08} for additional geometric
applications of nonlinear spectral gaps. In
Section~\ref{sec:sublinear} we discuss the relevance of nonlinear
spectral gaps to approximation algorithms, based on ideas of Barhum,
Goldreich and Shraibman~\cite{BGS}.

Answering a question of Kasparov and Yu~\cite{KY06}, V. Lafforgue
proved~\cite{Laff08} that there exists a sequence of bounded degree
graphs that are expanders with respect to {\em every} uniformly
convex normed space; such graph sequences are called
super-expanders.
In~\cite{MN-towards} we found a different construction of
super-expanders; here we show that our method can be applied to
situations in which it seems difficult to use Lafforgue's
(algebraic) approach.

It is a challenging question to characterize those metric spaces with respect to which there exist expander sequences. Random regular graphs are with high probability expanders in the classical sense (i.e., with respect to $\R$), but we are far from understanding those metric spaces with respect to which random regular graphs are expanders. One of the consequences of the results obtained here is that there exists a metric space $(X,d_X)$ with respect to which there exists an expander sequence, yet almost surely random regular graphs are not expanders with respect to $(X,d_X)$. No such example was previously known.

Let $\G_n$ be the set of all graphs on the vertex set
$\{1,\ldots,n\}$. The subset of $\G_n$ consisting of all the
connected graphs is denoted $\G_n^{\mathrm{con}}$. Given an integer
$d\ge 3$, let $\mathcal{G}_{n,d}$ be the probability measure on
$\G_n$ which is uniform over all those graphs in $\G_n$ that are
$d$-regular and have no self-loops and no parallel edges.
By~\cite{Wor81}, $\lim_{n\to
\infty}\mathcal{G}_{n,d}(\G_n^{\mathrm{con}})=1$.

A Hadamard space (also known as a complete $CAT(0)$ space) is a complete metric space $(X,d_X)$ with the property that for every $x,y\in X$ there exists a point $w\in X$ such that for every $z\in X$ we have
\begin{equation}\label{eq:def CAT(0)}
d_X(z,w)^2+\frac14 d_X(x,y)^2\le \frac12 d_X(z,x)^2+\frac12 d_X(z,y)^2.
\end{equation}
See the books~\cite{Jos97,BH-book} and the survey~\cite{Sturm03} for
an extensive account of Hadamard spaces.
One of the main questions left open in~\cite{NS11} (specifically,
see page 1547 of~\cite{NS11}) is whether or not it is true that if
$(X,d_X)$ is a Hadamard space that admits at least one expander
sequence then every classical expander sequence is also an expander
sequence with respect to $(X,d_X)$. Theorem~\ref{thm:main1} below,
which is the first of our two main theorems, answers this question.

\begin{theorem}\label{thm:main1}
There exist a Hadamard space $(X,d_X)$ and a sequence of $3$-regular graphs $\{G_n\}_{n=1}^\infty$  with $\lim_{n\to\infty} |V_{G_n}|=\infty$
such that
\begin{equation}\label{eq:G_n are indeed expanders}
\sup_{n\in \N} \gamma(G_n,d_X^2)<\infty,
\end{equation}
yet there exists $c\in (0,\infty)$ such that for every $d\in \N$,
\begin{equation}\label{eq:random graphs are not expanders}
\lim_{n\to \infty}\mathcal{G}_{n,d}\left(\left\{H\in \G_n:\ \gamma(H,d_X^2)\ge c(\log_d n)^2\right\}\right)=1.
\end{equation}
\end{theorem}

Theorem~\ref{thm:main1} answers the above mentioned question from~\cite{NS11}. Indeed, by~\eqref{eq:G_n are indeed expanders} the Hadamard space $(X,d_X)$ admits some expander sequence. Random regular graphs are asymptotically almost surely classical expanders~\cite{Bol88}, i.e., there exists $C\in (0,\infty)$ such that for every integer $d\ge 3$,
\begin{equation}\label{eq:random graphs are classical expanders}
\lim_{n\to \infty}\mathcal{G}_{n,d}\left(\left\{H\in \G_n:\ \gamma(H,d_\R^2)\le C\right\}\right)=1.
\end{equation}
Consequently, it follows from~\eqref{eq:random graphs are not expanders} and~\eqref{eq:random graphs are classical expanders} that not all classical expander sequences are expanders with respect to $(X,d_X)$.

Theorem~\ref{thm:main1} yields the first known example of a metric
space $(X,d_X)$ with respect to which random regular graphs are
asymptotically almost surely not expanders, yet there does exist a
special graph sequence $\{G_n\}_{n=1}^\infty$ that is an expander
sequence with respect to $(X,d_X)$. Observe that
$\{G_n\}_{n=1}^\infty$ is a fortiori a classical expander sequence.


The graphs $\{G_n\}_{n=1}^\infty$ of Theorem~\ref{thm:main1} have desirable properties which no other expander sequence is known to satisfy. Specifically, $\{G_n\}_{n=1}^\infty$ are expanders with respect to random regular graphs. This is made precise in the following theorem.

\begin{theorem}\label{thm:main2} There exists a universal constant $\Gamma\in (0,\infty)$ and a
sequence of $3$-regular graphs $\{G_n\}_{n=1}^\infty$  with
\begin{equation}\label{eq:bounded ratios}
\lim_{n\to\infty} |V_{G_n}|=\infty\qquad \mathrm{and} \qquad \sup_{n\in \N} \frac{|V_{G_{n+1}}|}{|V_{G_n}|}<\infty,
\end{equation}
such that for every integer $d\ge 3$ we have
\begin{equation}\label{eq:main thm random graph}
\lim_{m\to \infty}\mathcal{G}_{m,d}\left(\left\{H\in \G_m^{\mathrm{con}}:\ \sup_{n\in \N}
\gamma\!\left(G_n,d_H^2\right)< \Gamma\right\}\right)=1.
\end{equation}
\end{theorem}
In~\eqref{eq:main thm random graph} we restrict to $H\in
\G_m^{\mathrm{con}}$ because the metric $d_H$ is defined only when
$H$ is connected. Explicit bounds on the rates of convergence
in~\eqref{eq:random graphs are not expanders} and~\eqref{eq:main thm
random graph} are given Section~\ref{sec:slightly stronger}.

 Take $G\in \{G_n\}_{n=1}^\infty$ and write  $V_{G}=\{1,\ldots,k\}$ for some $k\in \N$. Theorem~\ref{thm:main2} asserts that almost surely as $m\to \infty$, if $H$ is a uniformly random $m$-vertex $d$-regular graph then for every $v_1,\ldots,v_k\in V_H$ the average of $d_H(v_i,v_j)^2$ over all $i,j\in \{1,\ldots,k\}$ is at most a constant multiple of the average of $d_H(v_i,v_j)^2$ over those $i,j\in \{1,\ldots,k\}$ that are joined by an edge of $G$. Note that the latter average is over $3k/2\asymp k$ numbers while the former average is over $\binom{k}{2}\asymp k^2$ numbers.

 Thus $G$ is an especially constructed fixed graph that serves as a ``sparse template" for
 computing the average squared distance between any $k$ vertices in a random regular graph.
 This yields {\em sublinear  (deterministic) time} approximate computation of average distances
 in random graphs: our input size is $\Omega(k^2)$, namely all the pairwise distances,
 while using $G$ we estimate $\frac{1}{k^2}\sum_{i=1}^k\sum_{j=1}^k d_H(x_i,x_j)^2$ by making
 only $O(k)$ distance queries. See Section~\ref{sec:sublinear} for more on this topic.

 Once the graph $G$ is given to us, the above statement about the average squared
 shortest-path distance of any $k$ vertices of the random graph $H$ involves elementary combinatorics and probability. Nevertheless, our proof of this statement uses methods from analysis and geometry that are interesting in their own right.

 Specifically, in~\cite{MN-towards} we introduced an
 iterative approach to the construction of super-expanders, building
 on the zigzag iteration of Reingold, Vadhan and
 Wigderson~\cite{RVW}. This approach uses estimates on martingales in uniformly convex Banach spaces.
 In~\cite{MN13-ext} we extended the estimates that were needed for the construction of super-expanders
 (namely nonlinear spectral calculus inequalities; see Remark~\ref{rem:calculus} below) to Hadamard spaces
 using an appropriate notion of nonlinear martingale. In order to apply these methods in the
 present setting, we consider the one-dimensional simplicial complex obtained by including all
 the edges of $H$ as unit intervals. We then work with the {\em Euclidean cone} over this one-dimensional
 simplicial complex, which is an auxiliary two-dimensional (random) continuous object on which martingale
 methods can be applied. The definition of the Euclidean cone over a metric space will be recalled in
 Section~\ref{sec:euclidean cone} below.
 We prove that if $H$ is sampled from $\mathcal{G}_{n,d}$ then with high probability its Euclidean cone is
 a (non-disjoint) union of two sets $A_1,A_2$ such that $A_1$ admits a bi-Lipschitz embedding into $L_1$
 and $A_2$ admits a bi-Lipschitz embedding into a Hadamard space. We then treat $A_1$ directly using Matou\v{s}ek's extrapolation lemma for Poincar\'e inequalities~\cite{Mat97}, and we treat $A_2$ using the methods of~\cite{MN-towards,MN13-ext}.

 The implementation of the above strategy  is
not straightforward, relying on a variety of geometric and analytic
tools;
 a more detailed overview of our proof of Theorem~\ref{thm:main2} appears in Section~\ref{sec:union sketch} and Section~\ref{sec:main3}.
 At this juncture we only wish to stress that our proof of Theorem~\ref{thm:main2}
 introduces a way to reason about random graphs that is potentially useful in other contexts: we consider such a  graph as being embedded
   in a larger auxiliary continuous geometric object that allows for the use of analytic methods, even though the statement being proved involves only the vertices of
the original graph.

\medskip

\noindent{\bf Previous work.} Nonlinear spectral gaps have been
studied in the literature from several points of view, leading to
some notational inconsistencies. Here we use the notation that
arises naturally from bi-Lipschitz embedding theory.
Pichot~\cite{Pic08} denotes the reciprocal of $\gamma(G,d_X^2)$ by
$\lambda_1(G,X)$.  In reference to Gromov's original
definition~\cite{Gro01,Gromov-random-group}, Pansu~\cite{Pan09} and
Kondo~\cite{Kondo} denote the same quantity by
$\lambda^{\mathrm{Gro}}(G,X)$ and $\lambda_1^{\mathrm{Gro}}(G,X)$,
respectively. When $X$ is a Hadamard space, a closely related
quantity, known today as Wang's invariant, was introduced by
Wang~\cite{Wang98,Wang00}; Wang's invariant is always within a
factor of $2$ of the reciprocal of $\gamma(G,d_X^2)$. For Hadamard
spaces, Izeki and Nayatani~\cite{IN05} introduced an invariant
that can be used to control the  ratio between Wang's invariant for
a graph  $G$ and the classical spectral gap of $G$.

We were motivated to revisit the question of~\cite{NS11} that
Theorem~\ref{thm:main1} answers by the recent work of
Kondo~\cite{Kondo}. Kondo's goal in~\cite{Kondo} was to construct a
Hadamard space for which the Izeki-Nayatani invariant is trivial.
He succeeded to do so by using the Euclidean cone over certain
expander graphs. In particular he obtained a Hadamard space that
contains bi-Lipschitzly copies of some (classical) expanders. Gromov
considered the same construction in~\cite{Gro01,Gromov-random-group}
for a different but related purpose. The main point of
Theorem~\ref{thm:main1} is to prove that the space $X$ admits a
sequence of expanders; we achieve this by modifying the Gromov-Kondo
construction so that it will be compatible with the method to
construct nonlinear expanders of~\cite{MN-towards}. The fact that
our graphs are expanders with respect to random graphs requires
additional work, relying on a structural result for Euclidean cones
over random graphs that is presented in Section~\ref{sec:structure
cone}.

\medskip

\noindent{\bf Roadmap.}  In Section~\ref{sec:sublinear} we describe
an algorithmic implication of Theorem~\ref{thm:main2}. Because the
proofs of Theorem~\ref{thm:main1} and Theorem~\ref{thm:main2} use
ingredients from several fields, Section~\ref{sec: prem} is devoted
to a detailed explanation of the background and main tools that will
be used in the proof of Theorem~\ref{thm:main1} and
Theorem~\ref{thm:main2}. The proofs themselves are given in
Section~\ref{sec:proofs and main statement}, a section that is
self-contained modulo some (quite substantial) ingredients that are
presented in Section~\ref{sec: prem} and whose proof appears in
later sections. The high-level structure of the argument is best
discerned from reading Section~\ref{sec:proofs and main statement},
since it uses as a ``black box" some conceptual ingredients whose
proof is quite lengthy. Section~\ref{sec:euclidean cone}
investigates the Lipschitz structure of Euclidean cones, proving in
particular that the Euclidean cone over $L_1$ admits a bi-Lipschitz
embedding into $L_1$. Section~\ref{sec:embed sparse} is devoted to
showing that sufficiently sparse graphs admit a bi-Lipschitz
embedding into $L_1$. Section~\ref{sec:random} deals with random
regular graphs, proving in particular a crucial structure theorem
(that holds true with high probability) for the Euclidean cone over
a random graph. In Section~\ref{sec:KF} we present a partial result
towards an open question that was posed by J. Kleinberg. For the
formulation of Kleinberg's question itself see
Section~\ref{sec:kleinberg}.

\section{Sublinear average distance approximation algorithms}\label{sec:sublinear}

Suppose that $(X,d_X)$ is a metric space and we have oracle
access to pairwise distances in $X$. Given $x_1,\ldots,x_n\in X$
write
\begin{equation}\label{eq:defA-universal}
A\eqdef
\frac{1}{n^2}\sum_{i=1}^n\sum_{j=1}^n d_X(x_i,x_j)^2.
\end{equation}
One can compute $A$ exactly with $n^2$ distance queries. But, we
wish to estimate $A$ in {\em sublinear time}, i.e., with only
$o(n^2)$ distance queries.

Indyk~\cite{Indyk99} proved that it is possible to approximate $A$
up to a factor of $1+\e$ by querying the distances between
$O(n/\e^{7/2})$ pairs of points chosen uniformly at random. Barhum,
Goldreich and Shraibman~\cite{BGS} improved the required number of
uniformly random pairs of points to $O(n/\e^2)$, which is
asymptotically tight~\cite{BGS}.

The above simple randomized sampling algorithm shows that for every
$x_1,\ldots,x_n\in X$ one can find $O(n)$ pairs of points from
$\{x_1,\ldots,x_n\}$ whose average distance is within $O(1)$ of $A$,
but these pairs depend on the initial point set
$\{x_1,\ldots,x_n\}\subset X$. Following~\cite{BGS}, for $D\in
[1,\infty)$ we say that a graph $G=(\{1,\ldots,n\},E)$ is a $D$-{\em
universal approximator} with respect to $(X,d_X^2)$ if there exists
(a scaling factor) $s\in (0,\infty)$ such that for every
$x_1,\ldots,x_n\in X$ we have

$$
\frac{1}{n^2}\sum_{i=1}^n\sum_{j=1}^n d_X(x_i,x_j)^2\le
\frac{s}{|E|}\sum_{\{i,j\}\in E} d_X(x_i,x_j)^2\le \frac{D}{n^2}\sum_{i=1}^n\sum_{j=1}^n d_X(x_i,x_j)^2.
$$

In~\cite{BGS} it is shown that there exists $c\in (0,\infty)$ such
that if $G=(\{1,\ldots,n\},E)$ is a $D$-universal approximator with
respect to $(X,d_X^2)$ for {\em every} metric space $(X,d_X)$ then
$$
|E|\gtrsim \frac{n^{1+c/\sqrt{D}}}{\sqrt{D}}.
$$
It was also shown in~\cite{BGS} that there exists $C\in (0,\infty)$
such that for every $D\in [1,\infty)$ and $n\in \N$ there exists a
graph $G=(\{1,\ldots,n\},E)$ that is a $D$-universal approximator
with respect to $(X,d_X^2)$ for {\em every} metric space $(X,d_X)$,
and such that
$$
|E|\le \sqrt{D}\cdot n^{1+C/\sqrt{D}}.
$$
We note that~\cite{BGS} deals with the analogous question for
universal approximators when the quantity $A$
in~\eqref{eq:defA-universal} is defined with the distances raised to
power $1$ rather than being squared. However, the arguments
of~\cite{BGS} easily extend mutatis mutandis to yield the above
stated results (and, in fact, to analogous statements when $A$ is
defined in terms of distances raised to power $p$ for any $p\ge 1$).

We thus have a satisfactory understanding of the size of universal
approximators with respect to {\em all} metric spaces. But, for
special metric spaces it is possible to obtain better tradeoffs.
Indeed, in~\cite{BGS} it is observed that an expander graph is a
linear size $O(1)$-universal approximator with respect to Hilbert
space; this is nothing more than an interpretation
of~\eqref{eq:compute average}, though by being more careful, and
using Ramanujan graphs~\cite{LPS,Mar88} of appropriate degree, it is
shown in~\cite{BGS} how to obtain a $1+\e$ approximation for every
$\e\in (0,1)$.

Due to~\eqref{eq:upper average bound trivial} and~\eqref{eq:def
gamma}, for every graph $G$ and every metric space $(X,d_X)$, if we
set $D=4\gamma(G,d_X^2)$ then $G$ is $D$-universal approximator with
respect to $(X,d_X^2)$. Hence, the super-expanders of~\cite{Laf09}
and~\cite{MN-towards} are $O_X(1)$-universal approximators with
respect to $(X,\|\cdot\|_X^2)$ for every uniformly convex Banach
space $(X,\|\cdot\|_X)$ (by $O_X(1)$ we mean that the approximation
factor depends only on $X$, in fact, it depends only on the modulus
of uniform convexity of $X$).

Theorem~\ref{thm:main2} yields the only known construction of
bounded degree $O(1)$-universal approximators with respect to random
regular graphs; this is the content of Theorem~\ref{thm:universal
approx} below. Graphs sampled from $\mathcal{G}_{n,d}$ occur in
various application areas, e.g. in networking, where they serve as
models for peer-to-peer networks~\cite{peer-2-peer-spaa}. Therefore,
a data structure that can compute quickly the average squared
distance of a given subset of a random regular graph is of
theoretical interest. However, the potential practicality of our
data structure is questionable because the approximation guarantee
is a large universal constant; we made no attempt to improve this
aspect of the construction.

\begin{theorem}\label{thm:universal approx}
There exists $D\in [1,\infty)$ and for every $n\in \N$ there exists
a graph $U_n=(\{1,\ldots,n\},E_n)$ with $|E_n|=O(n)$ such that for
every two integers $m,d\ge 3$, if $H$ is sampled from the
restriction of $\mathcal{G}_{m,d}$ to $\G_m^{\mathrm{con}}$ then
with probability that tends to $1$ as $m\to \infty$ the graphs
$\{U_n\}_{n=1}^\infty$ are $D$-universal approximators with respect
to $(V_H,d_H^2)$.
\end{theorem}

\begin{proof}
Let $\{G_k\}_{k=1}^\infty$ be the graphs from
Theorem~\ref{thm:main2}. Fixing $n\in \N$, it follows
from~\eqref{eq:bounded ratios} that there exists $k\in \N$ such that
$n\le |V_{G_k}|\le Mn$, where $M\in \N$ is a universal constant.
Partition $V_{G_k}$ into $n$ disjoint sets $A_1,\ldots,A_n\subset
V_{G_k}$ satisfying

\begin{equation}\label{eq:A_i size}
\forall\, i\in \{1,\ldots,n\},\qquad |A_i|\in \left\{ \left\lfloor
\frac{|V_{G_k}|}{n}\right\rfloor,\left\lfloor
\frac{|V_{G_k}|}{n}\right\rfloor+1\right\}.
\end{equation}
 Define the edge
multi-set $E_n$ of $U_n$ by letting the number of edges joining
$i\in \{1,\ldots,n\}$  and $j\in \{1,\ldots,n\}$ be equal to the
total number of edges joining $A_i$ and $A_j$ in $G_k$.
 Then $|E_n|\le |E_{G_k}|=\frac32|V_{G_k}|\le \frac32Mn$.

Suppose that $H\in \G_m$ is connected and
\begin{equation}\label{eq:sup gamma assumption}
\sup_{k\in \N} \gamma\left(G_k,d_H^2\right)< \Gamma,
\end{equation}
where $\Gamma$ is the constant from Theorem~\ref{thm:main2}. It
follows from~\eqref{eq:main thm random graph} that if $H\in \G_m$ is
sampled from $\mathcal{G}_{m,d}$ then the above asumptions hold true
with probability that tends to $1$ as $m\to \infty$.

Fix $x_1,\ldots,x_n\in V_H$ and define $f:V_{G_k}\to V_H$ by setting
$f(u)=x_i$ for $u\in A_i$. By the definition of $\gamma(G_k,d_H^2)$
combined with~\eqref{eq:sup gamma assumption} and~\eqref{eq:upper
average bound trivial},

\begin{multline}\label{use gamma le Gamma}
\frac{1}{|V_{G_k}|^2} \sum_{(u,v)\in V_{G_k}\times
V_{G_k}}d_H(f(u),f(v))^2\le \frac{\Gamma}{|E_{G_k}|}\sum_{\{u,v\}\in
E_{G_k}}d_H(f(u),f(v))^2\\\le \frac{4\Gamma}{|V_{G_k}|^2} \sum_{(u,v)\in V_{G_k}\times
V_{G_k}}d_H(f(u),f(v))^2.
\end{multline}
Since
\begin{equation*}
 \sum_{(u,v)\in V_{G_k}\times
V_{G_k}}d_H(f(u),f(v))^2 = \sum_{i=1}^n\sum_{j=1}^n
|A_i|\cdot|A_j|\cdot d_H(x_i,x_j)^2,
\end{equation*}
it follows from~\eqref{eq:A_i size} that
\begin{align}\label{eq:1/4 4}
\frac{n^2\lfloor
|V_{G_k}|/n\rfloor^2}{|V_{G_k}|^2}
&\le\frac{\frac{1}{|V_{G_k}|^2}
\sum_{(u,v)\in V_{G_k}\times V_{G_k}}d_H(f(u),f(v))^2}{\frac{1}{n^2}\sum_{i=1}^n\sum_{j=1}^n d_H(x_i,x_j)^2}
\nonumber\\&\le \frac{n^2\left(\lfloor
|V_{G_k}|/n\rfloor+1\right)^2}{|V_{G_k}|^2}.
\end{align}
Also, by the definition of $E_n$ and $f$ we have
\begin{equation}\label{eq:use edge definition}
\sum_{\{u,v\}\in
E_{G_k}}d_H(f(u),f(v))^2=\sum_{\{i,j\}\in
E_n}d_H(x_i,x_j)^2.
\end{equation}
By substituting~\eqref{eq:1/4 4} and~\eqref{eq:use edge definition}
into~\eqref{use gamma le Gamma} we conclude that

\begin{align*}
\frac{1}{n^2}\sum_{i=1}^n\sum_{j=1}^n d_H(x_i,x_j)^2&\le
\frac{\Gamma|V_{G_k}|^2|E_n|}{n^2\lfloor
|V_{G_k}|/n\rfloor^2|E_{G_k}|}\cdot \frac{1}{|E_n|}\sum_{\{i,j\}\in
E_n}d_H(x_i,x_j)^2\\&\le \frac{4\Gamma\left(1+\frac{1}{\lfloor
|V_{G_k}|/n\rfloor}\right)^2}{n^2}\sum_{i=1}^n\sum_{j=1}^n
d_H(x_i,x_j)^2\\&\le \frac{16\Gamma}{n^2}\sum_{i=1}^n\sum_{j=1}^n
d_H(x_i,x_j)^2.
\end{align*}
This means that $U_n$ is a $D$-universal approximator with respect
to $(V_H,d_H^2)$, where $D=16\Gamma$.
\end{proof}

\begin{remark}\label{rem:linear time}
A straightforward inspection of our proof of Theorem~\ref{thm:main2}
reveals that for every $n\in \N$ the universal approximator $U_n$ of
Theorem~\ref{thm:universal approx} can be constructed in
deterministic $O(n)$ time.
\end{remark}

\begin{remark}
Let $(X,d_X)$ be a metric space and $p\in [1,\infty)$. As noted
above, one can study universal approximators with respect to
$(X,d_X^p)$, i.e., given $x_1,\ldots,x_n\in X$ the goal is
approximate computation of the quantity
$\frac{1}{n^2}\sum_{i=1}^n\sum_{j=1}^n d_X(x_i,x_j)^p$ with $o(n^2)$
distance queries. The references~\cite{Indyk99,BGS} deal with $p=1$,
but as we mentioned earlier, they extend painlessly to general $p\ge
1$. Here we have only dealt with the case $p=2$, but we speculate
that our argument can be modified to yield bounded degree universal
approximators with respect to $(H,d_H^p)$ for any $p\in (1,\infty)$,
where $H$ is a random $d$-regular graph. We did not, however,
attempt to carry out our proof when $p\neq 2$. We also speculate
that the case $p=1$ requires more substantial new ideas, as the type
of martingale arguments that we use typically fail at the endpoint
$p=1$.
\end{remark}

\subsection{A question of J. Kleinberg}\label{sec:kleinberg}

In connection with Theorem~\ref{thm:main2} and the algorithmic
context described in Section~\ref{sec:sublinear}, Jon Kleinberg asked us
whether or not two stochastically independent random $3$-regular graphs are
asymptotically almost surely expanders with respect to each other.
Formally, we have the following open question.

\begin{question}[Jon Kleinberg]\label{Q:kleinberg}
Does there exist a universal constant $K\in (0,\infty)$ such that
$$
\lim_{\min\{m,n\}\to\infty} \mathcal{G}_{m,3}\times \mathcal{G}_{n,3}
\left[\left\{(G,H)\in \G_m\times \G_n^{\mathrm{con}}:\ \gamma\left(G,d_H^2\right)\right\}\le K\right]=1.
$$
\end{question}

While a positive solution of Question~\ref{Q:kleinberg} does not
formally imply that there exist graphs $\{G_n\}_{n=1}^\infty$ as in
Theorem~\ref{thm:main2}, such a statement would be a step towards a
probabilistic construction of such graphs. At present we do not have
methods to argue about the nonlinear spectral gap of random graphs.
In particular, it is a major open question whether or not random
$3$-regular graphs are super-expanders with positive probability.
Thus, all the known constructions in the context of nonlinear
spectral gaps are deterministic, with the only two methods that are
currently available being Lafforgue's algebraic
approach~\cite{Laff08} and our iterative approach~\cite{MN-towards}.
It seems, however, that the problem of obtaining a probabilistic
proof of the existence of expanders with respect to random graphs is
more tractable, and for this reason we believe that
Question~\ref{Q:kleinberg} is a promising research direction. In
Section~\ref{sec:KF} we describe a partial result towards a positive
solution of Question~\ref{Q:kleinberg} that we obtained through
discussions with Uriel Feige; we thank him for allowing us to
include his insights here.

\section{Preliminaries}\label{sec: prem}

In this section we set notation and terminology that will be used throughout
the ensuing discussion. We also present the tools and main steps
that will be used in the proofs of Theorem~\ref{thm:main1} and
Theorem~\ref{thm:main2}.

\subsection{Graph theoretical definitions}\label{sec:simple graphs} We start by defining some basic concepts in graph theory. Most of what we describe here is standard or self-evident terminology, and in fact some of it was used in the introduction without being defined explicitly.

Recall that graphs can have parallel edges and self-loops
unless stated otherwise. Thus,  when discussing a graph $G$ it will always be
understood that  $E_G$ is a {\em multi-subset} of unordered pairs of vertices, i.e., for every $u,v\in V_G$ the unordered pair $\{u,v\}$ is allowed
to appear in $E_G$ multiple times. For $(u,v)\in V_G\times V_G$, denote by
$E_G(u,v)$ the number of times that $\{u,v\}$ appears in $E_G$. The graph $G$ is therefore determined by the integer matrix
$(E_G(u,v))_{(u,v)\in V_G\times V_G}$. A graph is called {\em simple} if it contains no self-loops and no parallel edges, which is equivalent to requiring that $E_G(u,u)=0$ for every $u\in V_G$ and $E_G(u,v)\in \{0,1\}$ for every $u,v\in V_G$. Given $S,T\subset V_G$, we denote by $E_G(S,T)$ the multi-set of all edges in $E_G$ that join a vertex in $S$ and a vertex in $T$ (hence for $(u,v)\in V_G\times V_G$ we have $E_G(u,v)=|E_G(\{u\},\{v\})|$). When $S=T$ we use the simpler notation $E_G(S)=E_G(S,S)$.

The degree of a vertex $u\in V_G$ is
$\deg_G(u)= \sum_{v\in V_G}E_G(u,v).$
Under this convention each self-loop contributes $1$ to the degree of a vertex.
For $d\in \N$, a
graph $G$ is $d$-regular if $\deg_G(u)=d$ for every $u\in V_G$.

For a graph $G$ and a symmetric function $K:V_G\times V_G\to \R$
satisfying $K(u,u)=0$ for every $u\in V_G$, when we write the sum
$\sum_{\{u,v\}\in E_G}K(u,v)$ we mean that each edge is counted
once, i.e,
\begin{align*}
\nonumber\sum_{\{u,v\}\in
E_G}K(u,v)&\eqdef\frac12\sum_{\substack{(u,v)\in V_G\times
V_G\\\{u,v\}\in E_G}}K(u,v)\\&=\frac12 \sum_{(u,v)\in V_G\times V_G}
E_G(u,v)K(u,v).
\end{align*}

If $G$ is a connected graph then the diameter of the
metric space $(V_G,d_G)$ is denoted $\diam(G)$. The one-dimensional
simplicial complex induced by a graph $G$ is denoted $\Sigma(G)$.
Thus $\Sigma(G)$ is obtained from $G$ by including the edges of $G$
as unit intervals. The geodesic distance on $\Sigma(G)$ is denoted
$d_{\Sigma(G)}$. Then
\begin{equation}\label{eq:diam complex}
\diam(G)\le \diam(\Sigma(G))\le \diam(G)+1.
\end{equation}

Given a graph $G$, a subset $C\subseteq V_G$ is called a cycle of $G$ if it is possible to write $C=\{x_1,\ldots,x_k\}$, where the vertices $x_1,\ldots,x_k\in V$ are distinct and
\begin{equation}\label{eq:cycle edges}
\{x_1,x_2\},\{x_2,x_3\},\ldots,\{x_{k-1},x_k\},\{x_k,x_1\}\in E_G.
\end{equation}
Thus a self-loop is a cycle of length $1$ and a parallel edge
induces a cycle of length $2$. By definition, $|E_G(C)|\ge |C|$ for
every cycle $C\subseteq V_G$. A cycle $C\subseteq V_G$ is said to be
an induced cycle if $|E_G(C)|=|C|$, i.e., $E_G(C)$ consists only of
the edges listed in~\eqref{eq:cycle edges}. Note that the smallest
cycle in $G$ is necessarily induced.  The girth of a graph $G$,
denoted $\girth(G)$, is the size of the smallest cycle in $G$. If a connected graph
$G$ does not contain a cycle, i.e., $G$ is a tree, then we shall
use the convention $\girth(G)=2\diam(G)$.

The normalized adjacency matrix of a $d$-regular graph $G$, denoted $A_G$, is the $|V_G|$ by $|V_G|$ symmetric stochastic matrix whose entry at $(u,v)\in V_G\times V_G$ is equal to $E_G(u,v)/d$. The decreasing rearrangement of the eigenvalues of $A_G$ is denoted $$1=\lambda_1(G)\ge \lambda_2(G)\ge \ldots\ge \lambda_{|V_G|}(G).$$



\subsection{Bi-Lipschitz embeddings}\label{sec:prem bi-lip}

A metric space $(U,d_U)$ is said to admit a bi-Lipschitz embedding with distortion at most $D\in [1,\infty)$ into a metric space $(V,d_V)$ if there exists (a scaling factor) $s\in (0,\infty)$ and $f:U\to V$ that satisfies
\begin{equation}\label{eq:def distortion}
\forall\, x,y\in U,\qquad sd_U(x,y)\le d_V\big(f(x),f(y)\big)\le Dsd_U(x,y).
\end{equation}
The infimum over those $D\in (0,\infty)$ for which $(U,d_U)$ admits a bi-Lipschitz embedding with distortion at most $D$ into $(V,d_V)$ is denoted $c_{(V,d_V)}(U,d_U)$, or simply  $c_V(U)$ if the respective metrics are clear from the context (if no such $D\in (0,\infty)$ exists then we set $c_V(U)=\infty$). For $p\in [1,\infty]$ we use the simpler notation $c_p(U)=c_{L_p}(U)$. The quantity $c_2(U)$ is known as the Euclidean distortion of $U$ and the quantity $c_1(U)$ is known as the $L_1$ distortion of $U$.

If $G$ is a connected simple graph with $|V_G|=n$ then
$c_p(\Sigma(G))\lesssim n$ for every $p\in [1,\infty]$. Indeed,
write $V_G=\{1,\ldots,n\}$ and let $e_1,\ldots,e_n\in \R^n$ be the
coordinate basis of $\R^n$. If $\{i,j\}\in E_G$ and $x\in \Sigma(G)$
is a point on the unit interval that corresponds to the edge joining
$i$ and $j$, then map $x$ to $(1-t)e_i+te_j$, where $t$ is the
distance between $x$ and $i$ in $\Sigma(G)$. This trivial embedding
can be improved: one has $c_p(\Sigma(G))\lesssim \log(n+1)$. The
estimate $c_p(V_G,d_G)\lesssim \log(n+1)$ is the classical Bourgain
embedding theorem~\cite{Bourgain-embed}, but it is simple to argue
that the same bound holds true for embeddings of the one-dimensional
simplicial complex of $G$ rather than just its vertices. Using the better
bound $c_p(\Sigma(G))\lesssim \log(n+1)$ in what follows can improve
the implicit constants in Theorem~\ref{thm:main1} and
Theorem~\ref{thm:main2}, but since we ignore constant factors here
it will suffice to use the trivial bound $c_p(\Sigma(G))\lesssim n$.

For future use we also recall that an argument from~\cite{LLR} (see
also~\cite[Sec.~1.1]{MN-towards}) shows that for every connected
$d$-regular graph $H$ and every metric space $(Z,d_Z)$,
 \begin{equation}\label{eq:distortion lower gamma}
 c_{(Z,d_Z)}\left(V_H,d_H\right)\gtrsim \frac{\log_d n}{\sqrt{\gamma(H,d_Z^2)}}.
 \end{equation}

\subsection{Nonlinear absolute spectral gaps}\label{sec:asolute gap definition}
Fix $d,n\in \N$ and suppose that $G$ is an $n$-vertex $d$-regular graph and that $(X,d_X)$ is a metric space. Define $\gamma_+(G,d_X^2)$ to be the infimum over those $\gamma_+\in (0,\infty]$ such that for every two mappings
$f,g:V_G\to X$ we have
\begin{multline}\label{eq:def gamma+}
\frac{1}{n^2}\sum_{(u,v)\in V_G\times V_G}d_X\big(f(u),g(v)\big)^2\\ \le
\frac{\gamma_+}{nd}\sum_{(u,v)\in V_G\times V_G} E_G(u,v)\cdot d_X\big(f(u),g(v)\big)^2.
\end{multline}
Comparison of~\eqref{eq:def gamma+} to~\eqref{eq:def gamma} reveals that $\gamma(G,d_X^2)\le \gamma_+(G,d_X^2)$.

 With this notation one has
$$
\gamma_+\!\left(G,d_\R^2\right)=\frac{1}{1-\max\left\{\lambda_2(G),-\lambda_{|V_G|}(G\right\}}.
$$
For this reason we think of $\gamma_+(G,d_X^2)$ as measuring the reciprocal of the nonlinear absolute spectral gap of $G$ with respect to $(X,d_X)$.

Despite the fact that Theorem~\ref{thm:main1} and Theorem~\ref{thm:main2} deal with the quantity $\gamma(\cdot,\cdot)$, for their proofs it will be very convenient to work with the quantity $\gamma_+(\cdot,\cdot)$. See~\cite{MN-towards} for more on nonlinear absolute spectral gaps, specifically Section~2.2 of~\cite{MN-towards} for information on the relation between $\gamma(\cdot,\cdot)$ and $\gamma_+(\cdot,\cdot)$.

As in~\cite{MN-towards}, it is convenient to also work with nonlinear spectral gaps of symmetric stochastic matrices. Thus, letting $M=(m_{ij})$ be an $n$ by $n$ symmetric stochastic matrix and  $(X,d_X)$ be a metric space, define $\gamma(M,d_X^2)$ to be the infimum over those $\gamma\in (0,\infty]$ such that for every $x_1,\ldots,x_n\in X$ we have
\begin{equation*}\label{eq:def gamma matrix}
\frac{1}{n^2}\sum_{i=1}^n\sum_{j=1}^n d_X(x_i,x_j)^2\le \frac{\gamma}{n} \sum_{i=1}^n\sum_{j=1}^n m_{ij} d_X(x_i,x_j)^2.
\end{equation*}
Also,
define $\gamma_+(M,d_X^2)$ to be the infimum over those $\gamma_+\in (0,\infty]$ such that for every $x_1,\ldots,x_n,y_1,\ldots,y_n\in X$ we have
\begin{equation*}
\frac{1}{n^2}\sum_{i=1}^n\sum_{j=1}^n d_X(x_i,y_j)^2\le \frac{\gamma}{n} \sum_{i=1}^n\sum_{j=1}^n m_{ij} d_X(x_i,y_j)^2.
\end{equation*}
Under these definitions, one checks that for every regular graph $G$
we have $\gamma(G,d_X^2)=\gamma(A_G,d_X^2)$ and
$\gamma_+(G,d_X^2)=\gamma_+(A_G,d_X^2)$.

\subsection{Graph products, edge completion, and Ces\`aro averages}\label{sec:products}

Fix $d_1,d_2\in \N$. Let $G_1$ be a $d_1$-regular graph and let
$G_2$ be a $d_2$-regular graph. Suppose that $|V_{G_2}|=d_1$. Then
one can construct a new graph, called the {\em zigzag product} of
$G_1$ and $G_2$ and denoted $G_1\oz G_2$. This construction is due
to Reingold, Vadhan and Wigderson~\cite{RVW}. We will not need to
recall the definition of $G_1\oz G_2$ here: all that we will use
below is that $G_1\oz G_2$ is $d_2^2$-regular,  $|V_{G_1\oz
G_2}|=|V_{G_1}|\cdot|V_{G_2}|$, and for every metric space $(X,d_X)$
we have
\begin{equation}\label{eq:zigzag sub multiplicativity}
\gamma_+\!\left(G_1\oz G_2,d_X^2\right)\le \gamma_+\!\left(G_1,d_X^2\right)\cdot \gamma_+\!\left(G_2,d_X^2\right)^2.
\end{equation}
The inequality~\eqref{eq:zigzag sub multiplicativity} is due
to~\cite{MN-towards}.

Under the same assumptions on the graphs $G_1,G_2$, i.e.,  that $G_1$ is $d_1$-regular,  $G_2$ is $d_2$-regular, and $|V_{G_2}|=d_1$, one can also construct a new graph, called the {\em replacement product} of $G_1$ and $G_2$ and denoted $G_1 \circr G_2$. This construction is due to Gromov~\cite{Gromov-filling} (see also~\cite{RVW} and~\cite[Sec~8.3]{MN-towards}). Again, we will not need to recall the definition of $G_1\circr G_2$ here: all that we will use below is that the graph $G_1\circr G_2$ is $(d_2+1)$-regular, $|V_{G_1\circr G_2}|=|V_{G_1}|\cdot|V_{G_2}|$, and for every metric space $(X,d_X)$ we have
\begin{equation}\label{eq:replacement sub multiplicativity}
\gamma_+\!\left(G_1\circr G_2,d_X^2\right)\le 3(d_2+1)\cdot\gamma_+\!\left(G_1,d_X^2\right)
\cdot \gamma_+\!\left(G_2,d_X^2\right)^2.
\end{equation}
The inequality~\eqref{eq:replacement sub multiplicativity} is due to~\cite{MN-towards}.

Fix $d\in \N$ and suppose that $G$ is a $d$-regular graph. For every
integer $D\ge d$ one can define a new graph called the $D$-{\em edge
completion} of $G$, and denoted $\mathscr{C}_D(G)$.
See~\cite[Def.~2.8]{MN-towards} for the definition of
$\mathscr{C}_D(G)$. All that we will use below is that $\mathscr{C}_D(G)$ is
$D$-regular, $V_{\mathscr{C}_D(G)}=V_G$, and for every metric space $(X,d_X)$ we have
\begin{equation}\label{eq:completion gamma+}
\gamma_+\!\left(\mathscr{C}_D(G),d_X^2\right)\le 2\gamma_+\!\left(G,d_X^2\right).
\end{equation}
The proof of~\eqref{eq:completion gamma+} is contained in~\cite[Lem.~2.9]{MN-towards}.

We will also work below with {\em Ces\`aro averages} of graphs. Given $d,m\in \N$ and a $d$-regular graph $G$, its $m$th Ces\`aro average $\A_m(G)$ is a new graph defined by $V_{\A_m(G)}=V_G$ and
$$
\forall\, (u,v)\in V_G\times V_G,\qquad E_{\A_m(G)}(u,v)\eqdef \sum_{t=0}^{m-1} d^{m-1-t} \left(A_G^t\right)_{u,v},
$$
where we recall that $A_G$ is a the normalized adjacency matrix of $G$. One checks that $\A_m(G)$ is $md^{m-1}$-regular and that the adjacency matrix of $\A_m(G)$ is the corresponding Ces\`aro average of $A_G$, i.e.,
$$
A_{\A_m(G)}=\frac{1}{m}\sum_{t=0}^{m-1} A_G^t.
$$

The following lemma will be used (twice) in Section~\ref{sec:main3}.
\begin{lemma}\label{lem:from MN plus replacement}
Fix two integers $d,n\ge 3$ and let $G$ be a $d$-regular graph with
$|V_G|=2n$. Then there exists a $3$-regular graph $G^*$ satisfying
$|V_{G^*}|=36dn=18d|V_G|$ such that for every metric space
$(X,d_X)$,
\begin{equation}\label{eq:d^4 goal}
\gamma_+\!\left(G^*,d_X^2\right)\lesssim d^4\gamma\!\left(G,d_X^2\right).
\end{equation}
\end{lemma}

\begin{proof}
By Lemma~2.6 in~\cite{MN-towards}, there exists a $4d$-regular graph $G'$ with $|V_{G'}|=n$ such that $\gamma_+(G',d_X^2)\le 8\gamma(G,d_X^2)$ for every metric space $(X,d_X)$. Let $C_{4d}^\circ$ be the cycle of length $4d$ in which each vertex has exactly one self-loop (thus $C_{4d}^\circ$ is a $3$-regular graph). Since $G'$ is  a $4d$-regular graph, we may form the zigzag product $G''\oz C_{4d}^\circ$, which is a $9$-regular graph with $4dn$ vertices. By~\cite[Lem.~2.1]{MN-towards} we have $\gamma_+(C_{4d}^\circ,d_X^2)\le 192 d^2$. It therefore follows from~\eqref{eq:zigzag sub multiplicativity} that
\begin{equation}\label{d^4 easy}
\gamma_+\!\left(G'',d_X^2\right)\le 8\gamma(G,d_X^2)\cdot (192 d^2)^2\lesssim d^4 \gamma(G,d_X^2).
\end{equation}
Now, let $C_9$ be the cycle of length $9$ (thus $C_9$ is a
$2$-regular graph with $9$ vertices). Define $G^*$ to be the replacement product $G^*=G''\circr C_9$, which is a $3$-regular graph with $36dn$ vertices. By~\cite[Lem.~2.1]{MN-towards} we have $\gamma_+(C_9,d_X^2)\le 648$, so~\eqref{eq:d^4 goal} follows from~\eqref{d^4 easy} and~\eqref{eq:replacement sub multiplicativity}.
\end{proof}

\subsubsection{A zigzag iteration}\label{sec:iteration lemma} Theorem~\ref{thm:zigzag tool} below is a (much simpler)
variant of the iterative procedure by which we constructed
super-expanders in~\cite{MN-towards}, building on the zigzag
iteration of Reingold, Vadhan and Wigderson~\cite{RVW}.

\begin{theorem}\label{thm:zigzag tool} Fix $K\in [1,\infty)$ and two integers $n,d\ge 3$ with $n\ge d^3$. Suppose that $(X,d_X)$ is a metric space with the property that for every regular graph $G$ we have
\begin{equation}\label{eq:spectral calculus condition}
\gamma_+\!\left(\A_m(G),d_X^2\right)\le K\max\left\{1,\frac{\gamma_+(G,d_X^2)}{m}\right\},
\end{equation}
where
\begin{equation}\label{eq:def m}
m\eqdef \left\lfloor\frac{\log n}{3\log d}\right\rfloor.
\end{equation}
Suppose further that there exists a $d$-regular graph $H$ with $|V_H|=n$ such that
\begin{equation}\label{eq:base graph assumption}
\gamma_+\!\left(H,d_X^2\right)\le \sqrt{\frac{m}{2K}}.
\end{equation}
Then there exists a sequence of $3$-regular graphs
$\{G_j\}_{j=1}^\infty$ satisfying
\begin{equation}\label{eq:Gj cardinality}
\forall\, j\in \N,\qquad \left|V_{G_j}\right|=9d^2n^j,
\end{equation}
and
\begin{equation}\label{eq:Gj gamma+ vanilla}
\forall\, j\in \N,\qquad \gamma_+\!\left(G_j,d_X^2\right)\lesssim d^8K \gamma_+\!\left(H,d_X^2\right)^2.
\end{equation}
\end{theorem}

\begin{proof}
We first claim that for every $j\in \N$ there exists a
$d^2$-regular graph $W_j$ such that $|V_{W_j}|=n^j$ and
\begin{equation}\label{eq:W_j gamma+}
\gamma_+\!\left(W_j,d_X^2\right)\le 2K\gamma_+\!\left(H,d_X^2\right)^2.
\end{equation}
The proof of this statement is by induction on $j$. Define $W_1$ to
be the $d^2$-edge completion of $H$, i.e,
$$
W_1\eqdef
\mathscr{C}_{d^2}(H).
$$
Supposing that $W_j$ has been defined, the Ces\`aro average
$\A_m(W_j)$ is an $md^{2(m-1)}$-regular graph with
$|V_{\A_m(W_j)}|=|V_{W_j}|=n^j$. Recalling~\eqref{eq:def m}, we have
$md^{2(m-1)}\le n$. We can therefore form the edge completion
$\mathscr{C}_n\left(\A_m(W_j)\right)$, which is an $n$-regular
graph, and consequently the following zigzag product is well
defined.
\begin{equation}\label{eq:Wj+1}
W_{j+1}\eqdef \mathscr{C}_n\left(\A_m(W_j)\right)\oz H.
\end{equation}
The graph $W_{j+1}$ is $d^2$-regular and
$|V_{W_{j+1}}|=n|V_{\mathscr{C}_n\left(\A_m(W_j)\right)}|=n^{j+1}$.
Moreover,

\begin{eqnarray*}
\gamma_+\!\left(W_{j+1},d_X^2\right)&\stackrel{\eqref{eq:zigzag sub multiplicativity}\wedge
\eqref{eq:Wj+1}}{\le}& \gamma_+\!\left(\mathscr{C}_n\left(\A_m(W_j)\right),d_X^2\right)\cdot
\gamma_+\!\left(H,d_X^2\right)^2\\
&\stackrel{\eqref{eq:completion gamma+}}{\le}& 2\gamma_+\!\left(\A_m(W_j),d_X^2\right)\cdot
\gamma_+\!\left(H,d_X^2\right)^2\\
&\stackrel{\eqref{eq:spectral calculus condition}}{\le}& 2K\max\left\{1,\frac{\gamma_+(W_j,d_X^2)}{m}\right\}\cdot \gamma_+\!\left(H,d_X^2\right)^2\\
&\stackrel{\eqref{eq:W_j gamma+}}{\le} & 2K\max\left\{1,\frac{2K\gamma_+\!\left(H,d_X^2\right)^2}{m}\right\}
\cdot \gamma_+\!\left(H,d_X^2\right)^2\\
&\stackrel{\eqref{eq:base graph assumption}}{=}& 2K\gamma_+\!\left(H,d_X^2\right)^2.
\end{eqnarray*}
This completes the inductive construction of $\{W_j\}_{j=1}^\infty$.

Next, let $C_{d^2}^\circ$ be the cycle of length $d^2$ in which each
vertex has exactly one self-loop. We can form the zigzag product $W_j\oz
C_{d^2}^\circ$, which is a $9$-regular graph with $d^2n^j$ vertices.
By~\cite[Lem.~2.1]{MN-towards} we have
$\gamma_+\!\left(C_{d^2}^\circ,d_X^2\right)\le 12d^4$, so we deduce
from~\eqref{eq:zigzag sub multiplicativity} and~\eqref{eq:W_j
gamma+} that
\begin{equation}\label{eq:d8}
\gamma_+\!\left(W_j\oz
C_{d^2}^\circ,d_X^2\right)\le 2K\gamma_+\!\left(H,d_X^2\right)^2\cdot (12 d^4)^2\lesssim d^8K\gamma_+\!\left(H,d_X^2\right)^2.
\end{equation}

Recalling that $C_9$ denotes the simple cycle of length $9$, define $G_j$ to be the
replacement product
$$
G_j\eqdef \left(W_j\oz
C_{d^2}^\circ\right)\circr C_9.
$$
Then $G_j$ is a $3$-regular graph and $|V_{G_j}|=9d^2n^j$.
By~\cite[Lem.~2.1]{MN-towards} we have $\gamma_+(C_9,d_X^2)\le 648$.
It therefore follows from~\eqref{eq:replacement sub
multiplicativity} and~\eqref{eq:d8} that the desired
estimate~\eqref{eq:Gj gamma+ vanilla} holds true.
\end{proof}

\begin{remark}\label{rem:calculus}
The condition~\eqref{eq:spectral calculus condition} is called a
{\em nonlinear spectral calculus inequality}; see~\cite{MN-towards}
for an explanation of this terminology. Looking ahead, we will rely
here on the fact that such a nonlinear spectral calculus inequality
is available for Hadamard spaces, as proved in~\cite{MN13-ext}.
\end{remark}

\subsection{$CAT(0)$ and $CAT(1)$ spaces}\label{sec:CAT} We need to briefly recall basic definitions related to curvature upper bounds in the sense of Aleksandrov; see~\cite{BH-book} for much more on this topic.


A  metric space $(X,d_X)$ is a $CAT(1)$ space if it satisfies the following conditions. First, for every $x,y\in X$ with $d_X(x,y)<\pi$ there exists a geodesic joining $x$ to $y$, i.e., a curve $\phi:[0,1]\to X$ that satisfies $d_X(\phi(t),x)=td_X(x,y)$ and $d_X(\phi(t),y)=(1-t)d_X(x,y)$
 for every $t\in [0,1]$. Suppose that $x,y,z\in X$ satisfy
 $$
 d_X(x,y)+d_X(y,z)+d_X(z,x)<2\pi,
 $$
and that $\phi_{x,y},\phi_{y,z},\phi_{z,x}:[0,1]\to X$ are geodesics
joining $x$ to $y$, $y$ to $z$, and $z$ to $x$, respectively. Let
$S^2$ be the unit Euclidean sphere in $\R^3$, and let $d_{S^2}$
denote the geodesic metric on $S^2$ (thus the diameter of $S^2$
equals $\pi$ under this metric). As explained in~\cite{BH-book},
there exist $a,b,c\in S^2$ such that $d_{S^2}(a,b)=d_X(x,y)$,
$d_{S^2}(b,c)=d_X(y,z)$ and $d_{S^2}(c,a)=d_X(z,x)$. Let
$\f_{a,b},\f_{b,c},\f_{c,a}:[0,1]\to S^2$ be geodesics joining $a$
to $b$, $b$ to $c$, and $c$ to $a$, respectively. Then the remaining
requirement in the definition of a $CAT(1)$ space is that for every
$s,t\in [0,1]$ we have $d_X(\phi_{x,y}(s),\phi_{y,z}(t))\le
d_{S^2}(\f_{a,b}(s),\f_{b,c}(t))$.

Following Gromov~\cite{Gro01,Gromov-random-group} and
Kondo~\cite{Kondo}, we shall now describe a simple way to obtain
$CAT(1)$ spaces from graphs. Suppose that $\F$ is a family of
connected graphs that satisfies
\begin{equation}\label{eq:def R_F}
R_\F\eqdef \sup_{G\in \F} \frac{\diam(G)}{\girth(G)}<\infty.
\end{equation}
Define a metric space $(U_\F,d_\F)$ as follows. $U_\F$ is the formal
disjoint union of the one-dimensional simplicial complexes
$\{\Sigma(G)\}_{G\in \F}$, i.e.,
\begin{equation}\label{eq:def U_F}
U_\F\eqdef \bigsqcup_{G\in \F}\Sigma(G).
\end{equation}
The metric $d_\F:U_\F\times U_\F\to [0,\infty)$ is defined by
\begin{equation}\label{eq:def metric on U_F}
d_\F(x,y)\eqdef \left\{
\begin{array}{ll} \frac{2\pi d_{\Sigma(G)}(x,y)}{\girth(G)}&\mathrm{if}\ \exists G\in \F\  \mathrm{such\ that}\  x,y\in \Sigma(G),\\
2\pi\left(R_\F+1\right)&\mathrm{otherwise}. \end{array}\right.
\end{equation}
$d_\F$ is indeed a metric because for every $G\in \F$ and every $x,y\in \Sigma(G)$,
\begin{multline}\label{eq:rescaled diameter bound}
\frac{2\pi d_{\Sigma(G)}(x,y)}{\girth(G)}\le
\frac{2\pi\diam(\Sigma(G))}{\girth(G)}\stackrel{\eqref{eq:diam
complex}}{\le}
\frac{2\pi\left(\diam(G)+1\right)}{\girth(G)}\\\stackrel{\eqref{eq:def
R_F}}{\le}
\frac{2\pi\left(R_\F\girth(G))+1\right)}{\girth(G)}\le
2\pi\left(R_\F+1\right).
\end{multline}

We claim that $(U_\F,d_\F)$ is a $CAT(1)$ space. Indeed, since for
distinct $G,H\in \F$ we have $d_\F(\Sigma(G),\Sigma(H))>2\pi$ and
for every $G\in \F$ the metric space $(\Sigma(G),2\pi
d_{\Sigma(G)}/\girth(G))$ is geodesic, every $x,y\in U_\F$ with
$d_\F(x,y)<\pi$ can be joined by a geodesic. Moreover, if $x,y,z\in
U_\F$ satisfy $d_\F(x,y)+d_\F(y,z)+d_\F(z,x)<2\pi$ then necessarily
$x,y,z\in \Sigma(G)$ for some $G\in \F$, in which case we have
\begin{equation}\label{eq:<girth}
d_{\Sigma(G)}(x,y)+d_{\Sigma(G)}(y,z)+d_{\Sigma(G)}(z,x)<\girth(G).
\end{equation}
It follows that
\begin{equation}\label{eq:girth/2}
\max\left\{d_{\Sigma(G)}(x,y),d_{\Sigma(G)}(y,z),d_{\Sigma(G)}(z,x)\right\}<\frac{\girth(G)}{2}.
\end{equation}
Indeed, if~\eqref{eq:girth/2} fails then without loss of generality
we may assume that $d_{\Sigma(G)}(z,x)\ge \girth(G)/2$. By the
triangle inequality we would then have
$d_{\Sigma(G)}(x,y)+d_{\Sigma(G)}(y,z)\ge d_{\Sigma(G)}(z,x)\ge
\girth(G)/2$, so that~\eqref{eq:<girth} would be violated.
By~\eqref{eq:girth/2} the geodesic triangle whose vertices are
$x,y,z$ is contained (isometrically) in a metric tree, and since
metric trees are easily seen to be $CAT(1)$ spaces
(see~\cite[Sec.~II.1.15]{BH-book}), this completes the verification
that $(U_\F,d_\F)$ is a $CAT(1)$ space.

A metric space $(X,d_X)$ is said to be a $CAT(0)$ space if every two
points $x,y\in X$ can be joined by a geodesic, and for every
$x,y,z\in X$, every geodesic $\phi:[0,1]\to X$ with $\phi(0)=y$ and
$\phi(1)=z$, and every $t\in [0,1]$ we have
\begin{equation}\label{eq:t-CAT0}
d_X(x,\phi(t))^2\le (1-t)d_X(x,y)^2+td_X(x,z)^2-t(1-t)d_X(y,z)^2.
\end{equation}
A Hadamard space is a complete $CAT(0)$ space, in which case it
suffices to require the validity of~\eqref{eq:t-CAT0} for
$t=\frac12$.

\subsubsection{Nonlinear spectral calculus for $CAT(0)$
spaces}\label{sec:calculus CAT(0)} In~\cite{MN13-ext} we proved that
$CAT(0)$ metric spaces satisfy a nonlinear spectral calculus
inequality such as~\eqref{eq:spectral calculus condition}, a tool
that will be used crucially in what follows. This is a special case
of a more general result that is proved in~\cite{MN13-ext}, but we
will not formulate it here in order to avoid introducing terminology
that is not needed for the present context (we will use the zigzag
iteration of Theorem~\ref{thm:zigzag tool} {\em only} for $CAT(0)$
spaces).

\begin{theorem}[\cite{MN13-ext}]\label{thm:MNext} Let $(X,d_X)$
be a $CAT(0)$ space. Then for every $m,n\in \N$ and every $n$ by $n$
symmetric and stochastic matrix $M$,
\begin{equation}\label{eq:from MN13}
\gamma_+\!\left(\frac{1}{m}\sum_{t=0}^{m-1}M^t,d_X^2\right)\lesssim \max\left\{1,\frac{\gamma_+(M,d_X^2)}{m}\right\},
\end{equation}
\end{theorem}

\subsection{The Euclidean cone}\label{sec:euclidean cone} Let $(X,d_X)$ be a metric space. The Euclidean cone over $X$, denoted $\cone(X,d_X)$ or simply
$\cone(X)$ when the metric on $X$ is clear from the context, is
defined to be the completion of $(0,\infty)\times X$ under the
metric
\begin{equation}\label{eq:def cone}
d_{\cone(X)}\big((s,x),(t,y)\big)\eqdef \sqrt{s^2+t^2-2st\cos\left(\min\{\pi,d_X(x,y)\}\right)}.
\end{equation}
This definition is due to~\cite{Ber-Cone}; see also the exposition
in~\cite{ABN86} and~\cite{BH-book} (in particular, the fact
that~\eqref{eq:def cone} satisfies the triangle inequality is proved
in Proposition~5.9 of~\cite[Chapter~I.5]{BH-book}). When $X$ itself
is complete, the completion of $(0,\infty)\times X$ under the metric
defined in~\eqref{eq:def cone} amounts to extending the
definition~\eqref{eq:def cone} to $[0,\infty)\times X$ and
identifying all of the points in $\{0\}\times X$ (i.e., adding one
additional point to $(0,\infty)\times X$ as the ``cusp" of the
cone). The results that are discussed below pass immediately from a
space to its completion, so we will mostly ignore this one-point
completion in the ensuing arguments.

Observe that~\eqref{eq:def cone} can be rewritten as follows
\begin{multline}\label{eq:1-cos identity}
d_{\cone(X)}\big((s,x),(t,y)\big)\\=\sqrt{(s-t)^2+2st\left(1-\cos\left(\min\{\pi,d_X(x,y)\}\right)\right)}.
\end{multline}
For the sake of later comparisons between cones, we record here the following very simple fact.
\begin{fact}\label{fact:easy}
Fix $D\in [1,\infty)$ and two metric space $(X,d_X)$ and $(Y,d_Y)$. Suppose that $f:X\to Y$ satisfies
\begin{equation}\label{eq:D-bilip}
\forall\, x,y\in X,\qquad d_X(x,y)\le d_Y(f(x),f(y))\le Dd_X(x,y).
\end{equation}
 Then for every distinct $(s,x),(t,y)\in (0,\infty)\times X$ we have
\begin{equation}\label{eq:lif D-bilip}
1\le \frac{d_{\cone(Y)}\big((s,f(x)),(t,f(y))\big)}{d_{\cone(X)}\big((s,x),(t,y)\big)}\le D.
\end{equation}
\end{fact}

\begin{proof}
Since $\psi(a)\eqdef
\sqrt{1-\cos(\min\{\pi,a\})}=\sqrt{2}\sin(\min\{\pi/2,a/2\})$ is
concave and increasing on $[0,\infty)$, and $\psi(0)=0$,  for every
$a,b\in [0,\infty)$ with $a\le b$ we have $\psi(b)\le b\psi(a)/a$.
Hence by~\eqref{eq:D-bilip} we have
\begin{align*}
1-\cos\left(\min\{\pi,d_X(x,y)\}\right)&\le 1-\cos\left(\min\{\pi,d_Y(f(x),f(y))\}\right)\\
&\le D^2\left(1-\cos\left(\min\{\pi,d_X(x,y)\}\right)\right),
\end{align*}
which implies~\eqref{eq:lif D-bilip} due to~\eqref{eq:1-cos identity}.
\end{proof}

\begin{corollary}\label{cor:cone into cone}
For every metric space $(X,d_X)$ and every Banach space $(Y,\|\cdot\|_Y)$ we have
$
c_{\cone(Y)}(\cone(X))\le c_Y(X).
$
\end{corollary}
\begin{proof}
 For $D>c_Y(X)$ take $f:X\to Y$ that satisfies~\eqref{eq:D-bilip}. Define $F:(0,\infty)\times X\to (0,\infty)\times Y$ by $F(s,x)=(s,f(x))$, and use Fact~\ref{fact:easy}. The role of $Y$ being a Banach space was only to ensure that we can take the scaling factor $s$ in the definition~\eqref{eq:def distortion} to be equal to $1$.
\end{proof}

Proposition~\ref{thm:cone in L_1} below, whose proof is given in
Section~\ref{sec:cone L1}, asserts that $c_1(\cone(L_1))<\infty$, a
fact that plays an important role in our proofs of
Theorem~\ref{thm:main1} and Theorem~\ref{thm:main2}.

\begin{proposition}\label{thm:cone in L_1} $\cone(L_1)$ admits
a bi-Lipschitz embedding into $L_1$.
\end{proposition}

By combining Proposition~\ref{thm:cone in L_1} with Corollary~\ref{cor:cone into cone} we obtain the following useful corollary.

\begin{corollary}\label{cor:not harder to embed cones}
Every metric space $(X,d_X)$ satisfies
$$
c_1(\cone(X))\lesssim c_1(X).
$$
\end{corollary}

\subsection{$CAT(0)$ cones}\label{sec:cone is cat0} An important
theorem of Berestovski{\u\i}~\cite{Ber-Cone} (see also~\cite{ABN86}
and Theorem 3.14 of~\cite[Ch.~II.3]{BH-book}) asserts that if $X$ is
a $CAT(1)$ metric space then $\cone(X)$ is a $CAT(0)$ metric space.
Consequently, we have the following lemma.

\begin{lemma}\label{lem:cones over graph families}
Let $\F$ be a family of connected graphs satisfying~\eqref{eq:def
R_F}, and let $(U_\F,d_\F)$ be the metric space defined
by~\eqref{eq:def U_F} and~\eqref{eq:def metric on U_F}. Then
$\cone(U_\F,d_\F)$ is a Hadamard space. Moreover, for every $G\in
\F$ the metric space $(\Sigma(G),d_{\Sigma(G)})$ embeds into
$\cone(U_\F,d_\F)$ with $O(1)$ distortion, namely,
\begin{equation}\label{eq:cone slice distortion}
c_{\cone(U_\F,d_\F)}\left(\Sigma(G),d_{\Sigma(G)}\right)\le \frac{\pi(R_\F+1)}{2}.
\end{equation}
\end{lemma}

\begin{proof} In Section~\ref{sec:CAT} we verified that
$(U_\F,d_\F)$ is a $CAT(1)$ space, so the fact that
$\cone(U_\F,d_\F)$ is a Hadamard space is a consequence of
Berestovski{\u\i}'s theorem. Fixing $G\in \F$, define
$f:\Sigma(G)\to (0,\infty)\times U_\F$ by $f(x)=(1/\sqrt{2},x)$.
Recalling~\eqref{eq:def metric on U_F} and~\eqref{eq:def cone}, for
every $x,y\in \Sigma(G)$,
\begin{equation*}
d_{\cone(U_\F,d_\F)}(f(x),f(y))=\sqrt{1-\cos\left(\min\left\{\pi,\frac{2\pi
d_{\Sigma(G)}(x,y)}{\girth(G)}\right\}\right)}.
\end{equation*}
Due to the elementary inequalities
\begin{equation}\label{eq:cosine inequalities}
\forall\, \theta\in [0,\pi],\qquad \frac{2\theta^2}{\pi^2}\le 1-\cos\theta\le \frac{\theta^2}{2},
 \end{equation}
 it follows that
 \begin{equation}\label{eq:cone slice Lipschitz}
d_{\cone(U_\F,d_\F)}(f(x),f(y))\le \frac{\pi d_{\Sigma(G)}(x,y)}{\sqrt{2}\girth(G)},
 \end{equation}
 and
 \begin{eqnarray}\label{eq:a priori upper distance in simplicial}
\nonumber d_{\cone(U_\F,d_\F)}(f(x),f(y))&\ge&
\frac{\sqrt{2}}{\pi}\cdot\min\left\{\pi,\frac{2\pi d_{\Sigma(G)}(x,y)}{\girth(G)}\right\}\\
&\ge&\frac{\sqrt{2}d_{\Sigma(G)}(x,y)}{R_\F\girth(G)+1},
 \end{eqnarray}
 where the last step of~\eqref{eq:a priori upper distance in simplicial}
 relies on the penultimate inequality in~\eqref{eq:rescaled diameter
 bound}. The desired distortion estimate~\eqref{eq:cone slice
 distortion} follows from~\eqref{eq:cone slice Lipschitz}
 and~\eqref{eq:a priori upper distance in simplicial}.
\end{proof}

\subsection{Auxiliary families of graphs} The  family of graphs to which we will apply Lemma~\ref{lem:cones over graph families} is given in the following definition.

\begin{definition}\label{def:good family}
Fix $K\in (1,\infty)$ and denote by $\F_K$ all those connected graphs $G$
with the following properties.
\begin{enumerate}
\item $\diam(G)\le K\girth(G)$.
\item For every $S\subseteq \Sigma(G)$ with $|S|\le \sqrt{|V_G|}$ we have
\begin{equation}\label{eq:sqrt subset in l1}
c_1\left(\cone\left(S,\frac{2\pi}{\girth(G)}\cdot d_{\Sigma(G)}\right)\right)<K.
\end{equation}
\end{enumerate}
 We also define
\begin{equation}\label{eq:def X_K}
\mathscr{X}_K\eqdef \cone\left(U_{\F_K},d_{\F_K}\right).
\end{equation}
\end{definition}

By Lemma~\ref{lem:cones over graph families} the metric space $\X_K$
is a Hadamard space and for every $G\in \F_K$ we have
$c_{\mathscr{X}_K}(\Sigma(G),d_{\Sigma(G)})\lesssim K$. Eventually
$K$ will be chosen to be a large enough universal constant.

\begin{definition}\label{def:L class}
Fix two integers $n,d\ge 3$ and $K\in (1,\infty)$. Denote by
$\L_K^{n,d}$ the set of all those connected $n$-vertex $d$-regular
graphs $H$ for which there exists a subset of edges $I\subseteq E_H$
with $|I|\le \sqrt{n}$ such that if we let $L$ be the graph
$(V_H,E_H\setminus I)$ (i.e. we remove the edges in $I$ from $H$)
then $L$ has the following properties.
\begin{enumerate}
\item[(a)] $L\in \F_K$.
\item[(b)] $\diam(L)\le K\log_d n$.
\item[(c)] For every $S\subseteq V_L$ with $|S|\le n/2$ we have
\begin{equation*}\label{eq:expansion L}
|E_L(S,V_L\setminus S)|\ge \frac{|S|}{K}.
\end{equation*}
\end{enumerate}
\end{definition}

The following lemma asserts that a random $n$-vertex $d$-regular graph is asymptotically almost surely in $\L_K^{n,d}$. Due to item (a) of Definition~\ref{def:L class}, it follows in particular that for a large enough universal constant $K$, the graph family  $\F_K$ contains many graphs.

\begin{lemma}\label{lem:random graph in good family}
There exists a universal constant $K\in (1,\infty)$ and for every integer $d\ge 3$ there exists $C(d)\in (0,\infty)$ such that for all $n\in \N$,
\begin{equation}\label{eq:probability of L}
\mathcal{G}_{n,d}\left(\L_K^{n,d}\right)\ge 1-\frac{C(d)}{\sqrt[3]{n}}.
\end{equation}
\end{lemma}

From now on we will assume that $K$ is the universal constant from Lemma~\ref{lem:random graph in good family}.  We shall now proceed to establish some useful corollaries of Lemma~\ref{lem:random graph in good family}; the proof of Lemma~\ref{lem:random graph in good family} itself will be given in Section~\ref{sec:proof of L family lemma}.

\begin{corollary}\label{cor:random graphs are not expanders wrt X_K}
There exists a universal constant $c\in (0,\infty)$ such that for every two integers $d,n\ge 3$ we have
$$
\mathcal{G}_{n,d}\left(\left\{H\in \G_n:\ \gamma\!\left(H,d_{\X_K}^2\right)\ge c(\log_d n)^2\right\}\right)\ge 1-\frac{C(d)}{\sqrt[3]{n}},
$$
where $K,C(d)$ are the constant from Lemma~\ref{lem:random graph in good family} and $\X_K$ is as in~\eqref{eq:def X_K}.
\end{corollary}

\begin{proof}
By Lemma~\ref{lem:random graph in good family}, it suffices to show that there exists $c\in (0,\infty)$ such that if $H\in \L_K^{n,d}$ then $\gamma\!\left(H,d_{\X_K}^2\right)\ge c(\log_d n)^2$. So, supposing that $H\in \L_K^{n,d}$, let $I\subseteq E_H$ and $L=(V_H,E_H\setminus I)$ be as in Definition~\ref{def:L class}.

By part (a) of Definition~\ref{def:L class} we have $L\in \F_K$, so that by Lemma~\ref{lem:cones over graph families} we have $c_{\X_K}(V_H,d_L)\le \pi K$. This means that there exists a mapping $f:V_H\to \X_K$ and a scaling factor $s\in (0,\infty)$ such that
$$
\forall\, x,y\in V_H,\qquad sd_L(x,y)\le d_{\X_K}\left(f(x),f(y)\right)\le \pi K s d_L(x,y).
$$
By the definition of $\gamma\!\left(H,d_{\X_K}^2\right)$, we therefore have
\begin{equation}\label{eq:use def gamma for L}
\frac{1}{n^2}\sum_{(x,y)\in V_H\times V_H} d_L(x,y)^2\lesssim \frac{\gamma\!\left(H,d_{\X_K}^2\right)}{dn}\sum_{\{u,v\}\in E_H} d_L(x,y)^2.
\end{equation}

Because $L$ was obtained from $H$ by deleting edges, we trivially
have $d_L(x,y)\ge d_H(x,y)$ for every $x,y\in V_H$. Since $H$ is a
$d$-regular graph, for a universal constant fraction of the pairs
$(x,y)\in V_H\times V_H$ we have $d_H(x,y)\gtrsim \log_d n$ (for a
proof of this standard and simple fact, see e.g. page 193
of~\cite{Mat97}). Hence,
\begin{equation}\label{eq:to construct log lower}
\frac{1}{n^2}\sum_{(x,y)\in V_H\times V_H} d_L(x,y)^2\ge \frac{1}{n^2}\sum_{(x,y)\in V_H\times V_H} d_H(x,y)^2\gtrsim \left(\log_d n\right)^2.
\end{equation}

Since $E_L=E_H\setminus I$, if $\{x,y\}\in E_H\setminus I$ then $d_L(x,y)=1$, and by part (b) of Definition~\ref{def:L class} we have $d_L(x,y)\le K\log_d n$ if $\{x,y\}\in I$. Hence,
\begin{align}\label{eq:to contrast I}
\nonumber \frac{1}{dn}\sum_{\{u,v\}\in E_H} d_L(x,y)^2&\le \frac{\left(|E_H|-|I|\right)+K^2(\log_d n)^2|I|}{dn}\\
&\lesssim \frac{dn+(\log_d n)^2\sqrt{n}}{dn}\lesssim 1,
\end{align}
where we used the fact that $|I|\le \sqrt{n}$, as stipulated in
Definition~\ref{def:L class}. The desired lower bound
$\gamma\!\left(H,d_{\X_K}^2\right)\gtrsim (\log_d n)^2$ now follows
by contrasting~\eqref{eq:use def gamma for L} with~\eqref{eq:to
construct log lower} and~\eqref{eq:to contrast I}.
\end{proof}

\begin{corollary}\label{cor:Gamma_n} Let $K$ and $C(3)$ be the
constants from Lemma~\ref{lem:random graph in good family} (with
$d=3$). Then for every integer $n>C(3)^3$ there exists an $n$-vertex
$3$-regular graph $\Gamma_n$ such that
\begin{equation}\label{eq:use cheeger for Gamma_n}
\sup_{n>C(3)^3} \lambda_2(\Gamma_n)<1,
\end{equation}
and
\begin{equation}\label{eq:embedding Gamma_n}
\sup_{n>C(3)^3} c_{\X_K}\left(V_{\Gamma_n},d_{\Gamma_n}\right)<\infty,
\end{equation}
where $\X_K$ is as in~\eqref{eq:def X_K}.
\end{corollary}

\begin{proof}
Since the right hand side of~\eqref{eq:probability of L} (with
$d=3$) is positive, by Lemma~\ref{lem:random graph in good family}
there exists an $n$-vertex $3$-regular graph $H$ such that $H\in
\L_K^{n,d}$. Let $L=(V_H,E_H\setminus I)$ be as in
Definition~\ref{def:L class}. Since $L$ was obtained from $H$ by
deleting edges, we can add self-loops to those vertices of $V_H$
whose degree in $L$ is less than $3$  so that the resulting graph,
which will be denoted $\Gamma_n$, is $3$-regular. By part (c) of
Definition~\eqref{def:L class}, for every $S\subseteq V_{\Gamma_n}$
with $|S|\le n/2$ we have $|E_{\Gamma_n}(S,V_{\Gamma_n}\setminus
S)|\gtrsim |S|$. By Cheeger's inequality~\cite{Che70,AM85}, this
implies~\eqref{eq:use cheeger for Gamma_n}.  Also,
$d_{\Gamma_n}=d_L$, so~\eqref{eq:embedding Gamma_n} follows from
Lemma~\ref{lem:cones over graph families}.
\end{proof}


\subsection{On the structure of cones over random
graphs}\label{sec:structure cone} We have already described the
tools that will suffice for the proof of Theorem~\ref{thm:main1},
but in order to also prove Theorem~\ref{thm:main2} we need to have a
better understanding of the structure of the Euclidean cone over a
random graph. This is the content of the following proposition,
whose proof appears in Section~\ref{sec:proof of structure cone
prop} below.

\begin{proposition}\label{prop:structure cone random}
 For every integer $d\ge 3$ there exists $C'(d)\in (0,\infty)$ such that for every
 $n\in \N$, if $H$ is distributed according to $\mathcal{G}_{n,d}$ then with probability at
 least $1-C'(d)/\sqrt[3]{n}$ the graph $H$ is connected and there exist $\sigma\in (0,\infty)$ and two subsets $A_1,A_2\subseteq \Sigma(H)$
 such that the following assertions hold true.
 \begin{enumerate}
 \item[(I)] $A_1\cup A_2=\Sigma(H)$.
 \item[(II)] $d_{\Sigma(H)}\left(A_1\setminus A_2,A_2\setminus
 A_1\right)\gtrsim \frac{1}{\sigma}$.
 \item[(III)] $c_1\left(\cone\left(A_1,\sigma d_{\Sigma(H)}\right)\right)\lesssim
 1$.
 \item[(IV)] $c_{\X_K}\left(\cone\left(A_2,\sigma d_{\Sigma(H)}\right)\right)\lesssim 1$.
 \item[(V)] $c_{\cone\left(\Sigma(H),\sigma
 d_{\Sigma(H)}\right)}\left(\Sigma(H),d_{\Sigma(H)}\right)\lesssim
 1$.
 \end{enumerate}
Here $K$ is the constant from Lemma~\ref{lem:random graph in good family} and $\X_K$ is as in~\eqref{eq:def X_K}.
\end{proposition}

By naturally identifying the cusps of $\cone(A_1)$ and $\cone(A_2)$
with the cusp of $\cone\left(\Sigma(H)\right)$, we have
 \begin{equation}\label{eq:cone union}
\cone\left(\Sigma(H)\right)= \cone(A_1)\cup \cone(A_2).
 \end{equation}
Proposition~\ref{prop:structure cone random} therefore says that
$\cone\left(\Sigma(H),\sigma
 d_{\Sigma(H)}\right)$ can be written as a (non-disjoint) union of
 two subsets, the first of which embeds bi-Lipschitzly into $L_1$ and the second of which embeds bi-Lipschitzly into
 the Hadamard space $\X_K$. The relevance of the remaining
 assertions of Proposition~\ref{prop:structure cone random} will
 become clear in Section~\ref{sec:union sketch} below.

\subsection{Unions of
cones}\label{sec:union sketch}

Proposition~\ref{prop:structure cone random} will be used to prove
Theorem~\ref{thm:main2} via the following strategy. We will
construct a sequence of $3$-regular graphs $\{G_n\}_{n=1}^\infty$
satisfying~\eqref{eq:bounded ratios} and such that
$$
\sup_{n\in \N} \gamma_+\!\left(G_n,d_{\X_K}^2\right)<\infty.
$$
By assertion $(IV)$ of Proposition~\ref{prop:structure cone random}
we therefore have
\begin{equation}\label{eq:gamma+ for cone over A2}
\sup_{n\in \N} \gamma_+\!\left(G_n,d_{\cone\left(A_2,\sigma d_{\Sigma(H)}\right)}^2\right)<\infty.
\end{equation}
Since $\{G_n\}_{n=1}^\infty$ are necessarily also classical
expanders, we will argue using Matou\v{s}ek's
extrapolation~\cite{Mat97} that
$$
\sup_{n\in \N} \gamma_+\!\left(G_n,d_{L_1}^2\right)<\infty,
$$
where $d_{L_1}(f,g)\eqdef \|f-g\|_1$ is the standard metric on
$L_1$. Assertion $(III)$ of Proposition~\ref{prop:structure cone
random} therefore implies that
\begin{equation}\label{eq:gamma+ for cone over A1}
\sup_{n\in \N} \gamma_+\!\left(G_n,d_{\cone\left(A_1,\sigma d_{\Sigma(H)}\right)}^2\right)<\infty.
\end{equation}
We would like to combine~\eqref{eq:gamma+ for cone over A2}
and~\eqref{eq:gamma+ for cone over A1} to deduce that
$\{G_n\}_{n=1}^\infty$ are expanders with respect to
$\cone\left(\Sigma(H),\sigma
 d_{\Sigma(H)}\right)$. For this purpose, we prove the following
 lemma in Section~\ref{sec:union}.

\begin{lemma}\label{lem:cone union}
Fix $\beta\in (0,\pi]$ and $n\in \N$. Let $(X,d_X)$ be a metric
space and suppose that $A,B\subset X$ satisfy $A\cup B=X$ and
\begin{equation}\label{eq:AB distant}
d_X(A\setminus B,B\setminus A)\ge \beta.
\end{equation}
Then every $n$ by $n$ symmetric stochastic matrix $M=(m_{ij})$
satisfies
\begin{equation}\label{eq;AB cone}
\gamma_+\!\left(M,d_{\cone(X)}^2\right)\lesssim \frac{\gamma_+\big(M,d_{\cone(A)}^2\big)+\gamma_+\big(M,d_{\cone(B)}^2\big)}{\beta^4}.
\end{equation}
\end{lemma}

By virtue of assertion $(II)$ of Proposition~\ref{prop:structure
cone random}, we may use Lemma~\ref{lem:cone union} to deduce
from~\eqref{eq:gamma+ for cone over A2} and~\eqref{eq:gamma+ for
cone over A1} that
$$
\sup_{n\in \N} \gamma_+\!\left(G_n,d_{\cone\left(\Sigma(H),\sigma
 d_{\Sigma(H)}\right)}^2\right)<\infty.
$$
By assertion $(V)$ of Proposition~\ref{prop:structure cone random}
it follows that
$$
\sup_{n\in \N} \gamma_+\!\left(G_n,d_{\Sigma(H)}^2\right)<\infty,
$$
implying Theorem~\ref{thm:main2}. It remains, of course, to explain
how to construct the graphs $\{G_n\}_{n=1}^\infty$; this is done
in~Section~\ref{sec:proofs and main statement} below.

\begin{remark}
In light of~\eqref{eq:cone union}, a positive answer to the
following natural open question would yield an alternative (more
general) way to carry out the above argument. Suppose that $(X,d_X)$
is a metric space and $A,B\subseteq X$ satisfy $A\cup B=X$. Write
$d_A$ and $d_B$ for the restriction of $d_X$ to $A$ and $B$,
respectively. Let $\{G_n\}_{n=1}^\infty$ be a sequence of
$3$-regular graphs such that $\sup_{n\in
\N}\gamma_+(G_n,d_A^2)<\infty$ and $\sup_{n\in
\N}\gamma_+(G_n,d_B^2)<\infty$. Does this imply that $\sup_{n\in
\N}\gamma_+(G_n,d_X^2)<\infty$? The same question could be also
asked with $\gamma_+(\cdot,\cdot)$ replaced throughout by
$\gamma(\cdot,\cdot)$. We speculate that the answer to the above
questions is negative, in which case one would ask for conditions on
$A$ and $B$ which would imply a positive answer. Lemma~\ref{lem:cone
union} is a step in this direction.
\end{remark}

\section{Proof of Theorem~\ref{thm:main1} and
Theorem~\ref{thm:main2}}\label{sec:proofs and main statement}

Here we prove Theorem~\ref{thm:main1} and Theorem~\ref{thm:main2}
while assuming only the validity of Lemma~\ref{lem:random graph in
good family}, Proposition~\ref{prop:structure cone random} and
Lemma~\ref{lem:cone union}, which will be proven in
Section~\ref{sec:proof of L family lemma}, Section~\ref{sec:proof of
structure cone prop} and Section~\ref{sec:union}, respectively.
Proposition~\ref{thm:cone in L_1}, whose proof appears in
Section~\ref{sec:cone L1}, serves as a step towards the proof of
Proposition~\ref{prop:structure cone random}.

\subsection{A stronger theorem}\label{sec:slightly stronger} We
state below a theorem that directly implies both
Theorem~\ref{thm:main1} and Theorem~\ref{thm:main2}.

\begin{theorem}\label{thm:main3}
There exist a Hadamard space $(X,d_X)$ and two sequences of
connected $3$-regular graphs $\{G_n\}_{n=1}^\infty$ and
$\{\Gamma_n\}_{n=1}^\infty$ satisfying $\lim_{n\to\infty}
|V_{G_n}|=\lim_{n\to \infty} |V_{\Gamma_n}|=\infty$ and $ \sup_{n\in \N}
|V_{G_{n+1}}|/|V_{G_n}|<\infty$, such that the following properties
hold true.
\begin{enumerate}
\item $\sup_{n\in \N} \gamma_+(G_n,d_X^2)<\infty$.
\item $\sup_{n\in \N} \lambda_2(\Gamma_n)<1$.
\item $\sup_{n\in \N}
c_{X}\left(V_{\Gamma_n},d_{\Gamma_n}\right)<\infty$.

\item There exists $c,C\in (0,\infty)$ and for every integer $d\ge 1$
there exists $\kappa(d)\in (0,\infty)$ such that for every $n\in
\N$ we have
\begin{equation*}
\mathcal{G}_{n,d}\left(\left\{H\in \G_n:\ \gamma\!\left(H,d_{X}^2\right)
\ge c(\log_d n)^2\right\}\right)\ge 1-\frac{\kappa(d)}{\sqrt[3]{n}},
\end{equation*}
and
\begin{equation*}\label{eq:stronger main thm random graph}
\ \ \ \ \ \mathcal{G}_{n,d}\left(\left\{H\in \G_n^{\mathrm{con}}:\ \sup_{k\in \N}
\gamma_+\!\left(G_k,d_{\Sigma(H)}^2\right)< C\right\}\right)\ge 1-\frac{\kappa(d)}{\sqrt[3]{n}}.
\end{equation*}
\end{enumerate}
\end{theorem}
Theorem~\ref{thm:main3} is stronger than Theorem~\ref{thm:main1} and
Theorem~\ref{thm:main2}. Indeed, the estimate~\eqref{eq:G_n are
indeed expanders} in Theorem~\ref{thm:main1} is weaker than
assertion {\it (1)} of Theorem~\ref{thm:main3} because
$\gamma_+(G_n,d_X^2)\ge \gamma(G_n,d_X^2)$. The
estimate~\eqref{eq:random graphs are not expanders} in
Theorem~\ref{thm:main1} and the estimate~\eqref{eq:main thm random
graph} in Theorem~\ref{thm:main2} follow from assertion {\it (4)} of
Theorem~\ref{thm:main3}.

The role of the graphs $\{\Gamma_n\}_{n=1}^\infty$ of
Theorem~\ref{thm:main3} is to supply the following additional
information that is not contained in Theorem~\ref{thm:main1} and
Theorem~\ref{thm:main2}. By assertions {\it (2)} and {\it (3)} of
Theorem~\ref{thm:main3}, we know that $X$ contains bi-Lipschitz
copies of some $3$-regular (classical) expander sequence.
Using~\eqref{eq:distortion lower gamma} we deduce from assertion
{\it (3)} of Theorem~\ref{thm:main3} $\gamma(\Gamma_n,d_X^2)\gtrsim
(\log n)^2$. As was discussed in the paragraph that immediately
follows the statement of Theorem~\ref{thm:main1}, we already knew
that as a consequence of Theorem~\ref{thm:main1} there exists a
sequence of $3$-regular graphs $\{H_n\}_{n=1}^\infty$ for which
$\gamma(H_n,d_X^2)\gtrsim (\log n)^2$. Theorem~\ref{thm:main3} says
that moreover, we can even arrange it for $X$ to contain
bi-Lipschitz copies of such an expander sequence.

\subsection{Proof of Theorem~\ref{thm:main3}}\label{sec:main3} In
the proof of Theorem~\ref{thm:main3} we will use the following easy
lemma, which is in the spirit of Lemma~\ref{lem:cone union} but
simpler to prove.

\begin{lemma}\label{lem:cone pi separated multi union}
Let $(X,d_X)$ be a metric space and suppose that $\{A_i\}_{i\in I}$
are subsets of $X$ such that
\begin{equation}\label{eq:union and separation}
X=\bigcup_{i\in I}A_i,\quad\mathrm{and}\quad \forall\, i,j\in I,\ (i\neq j)\implies d_X(A_i,A_j)\ge \pi.
\end{equation}
Then for every $n\in \N$, any $n$ by $n$ symmetric stochastic
$M=(m_{ij})$ satisfies
\begin{equation}\label{eq;Ai cone+}
\gamma_+\!\left(M,d_{\cone(X)}^2\right)\lesssim 2\sup_{i\in I} \gamma_+\!\left(M,d_{\cone(A_i)}^2\right),
\end{equation}
and
\begin{equation}\label{eq;Ai cone}
\gamma\left(M,d_{\cone(X)}^2\right)\lesssim 2\sup_{i\in I} \gamma\left(M,d_{\cone(A_i)}^2\right),
\end{equation}
\end{lemma}

\begin{proof} Let $o\in \cone(X)$ denote the equivalence class obtained by identifying all
the points in $\{0\}\times X$, thus
$d_{\cone(X)}\big((s,x),o\big)=s$ for every $(s,x)\in
(0,\infty)\times X$. With a slight abuse of notation we think of $o$
 as the cusp $\{\cone(A_i)\}_{i\in I}$ as well (formally the cusps of $\{\cone(A)\}_{i\in I}$
 should be labeled differently, as the equivalence classes
 $\{\{0\}\times A_i\}_{i\in I}$, but dropping this notation
 will not cause confusion in what follows).

For every $i\in I$ define $\a_i:\cone(X)\to \cone(A_i)$ by
$$
\a_i(s,x)\eqdef\left\{\begin{array}{cc}(s,x) &\mathrm{if\ } x\in A_i\\o &\mathrm{otherwise.} \end{array}\right.
$$
Due to~\eqref{eq:union and separation} every $(s,x),(t,y)\in
\cone(X)$ satisfy
\begin{equation}\label{eq:l1 union}
d_{\cone(X)}\big((s,x),(t,y)\big)=\sum_{i\in I}d_{\cone(A_i)}\big(\a_i(s,x),\a_i(t,y)\big),
\end{equation}
where the sum in the right-hand side of~\eqref{eq:l1 union} contains
at most two nonzero terms. Consequently,
\begin{equation*}
1\le \frac{d_{\cone(X)}\big((s,x),(t,y)\big)^2}{\sum_{i\in I}d_{\cone(A_i)}\big(\a_i(s,x),\a_i(t,y)\big)^2}\le 2.
\end{equation*}
Every $(s_1,x_1),\ldots, (s_n,x_n),(t_1,y_1),\ldots,(t_n,y_n)\in
\cone (X)$ therefore satisfy
\begin{align*}
&\frac{1}{n^2}\sum_{j=1}^n\sum_{k=1}^n d_{\cone(X)}\big((s_j,x_j),(t_k,y_k)\big)^2\\&\le 2\sum_{i\in I}\frac{1}{n^2}\sum_{j=1}^n\sum_{k=1}^n d_{\cone(A_i)}\big(\a_i(s_j,x_j),\a_i(t_k,y_k)\big)^2\\
&\le 2\sum_{i\in I}\frac{\gamma_+\big(M,d_{\cone(A_i)}^2\big)}{n}\sum_{j=1}^n\sum_{k=1}^n m_{jk}d_{\cone(A_i)}\big(\a_i(s_j,x_j),\a_i(t_k,y_k)\big)^2\\
&\le \frac{2\sup_{i\in I} \gamma_+\big(M,d_{\cone(A_i)}^2\big)}{n}\sum_{j=1}^n\sum_{k=1}^n m_{jk}d_{\cone(X)}\big((s_j,x_j),(t_k,y_k)\big)^2.
\end{align*}
This proves~\eqref{eq;Ai cone+}. The proof of~\eqref{eq;Ai cone} is
analogous.
\end{proof}

The following lemma will be used (twice) in the proof of Theorem~\ref{thm:main3}. Its proof relies on Matou\v{s}ek's extrapolation lemma for Poincar\'e inequalities~\cite{Mat97}, combined with an idea from~\cite{NS11}.

\begin{lemma}\label{lem:matousek extrpolation plus snowflake} Fix $d,n\in \N$ and let $G$ be a $d$-regular $n$-vertex graph with $\lambda_2(G)<1$. Then
$$
\gamma\!\left(G,d_{L_1}^2\right)\lesssim \frac{1}{\left(1-\lambda_2(G)\right)^2}.
$$
\end{lemma}

\begin{proof} Since $\gamma(G,d_\R^2)=1/(1-\lambda_2(G))$, it follows from
Matou\v{s}ek's extrapolation lemma for
Poincar\'e inequalities (see~\cite[Prop.~3]{Mat97} for Matou\v{s}ek's original version and~\cite[Lem.~5.5]{BLMN05} for the form that we use here) that  for every
$f:V_{G}\to L_4$ we have
\begin{multline}\label{eq:here matousek is used}
\frac{1}{|V_{G}|^2}\sum_{(u,v)\in V_{G}\times V_{G}}\|f(u)-f(v)\|_4^4\\
\lesssim \frac{1}{(1-\lambda_2(G))^2|E_{G}|}
\sum_{\{u,v\}\in E_{G}} \|f(u)-f(v)\|_4^4.
\end{multline}
The metric space $(L_1,\sqrt{d_{L_1}})$ admits an isometric embedding into $L_2$ (see e.g.~\cite{DZ}), and $L_2$ admits an isometric embedding into $L_4$ (see
e.g.~\cite{Woj91}), so $(L_1,\sqrt{d_{L_1}})$ admits an isometric embedding into $L_4$. It therefore follows from~\eqref{eq:here matousek is used} that

\begin{multline*}
\frac{1}{|V_{G}|^2}\sum_{(u,v)\in V_{G}\times V_{G}}\|f(u)-f(v)\|_1^2\\
\lesssim \frac{1}{(1-\lambda_2(G))^2|E_{G}|}
\sum_{\{u,v\}\in E_{G}} \|f(u)-f(v)\|_1^2.\tag*{\qedhere}
\end{multline*}
\end{proof}

\begin{proof}[Proof of Theorem~\ref{thm:main3}] Let $K$ be the constant from Lemma~\ref{lem:random graph in
good family} and write $X=\X_K$. The existence of the sequence of
$3$-regular graphs $\{\Gamma_n\}_{n=1}^\infty$, together with
assertions {\it (2)} and {\it (3)} of Theorem~\ref{thm:main3}, is
precisely Corollary~\ref{cor:Gamma_n}. The first displayed equation
in assertion {\it (4)} of Theorem~\ref{thm:main3} follows from
Corollary~\ref{cor:random graphs are not expanders wrt X_K}.

Our goal is to construct $\{G_n\}_{n=1}^\infty$ using
Theorem~\ref{thm:zigzag tool}. To this end, since the
assumption~\eqref{eq:spectral calculus condition} holds true by
virtue of the fact that $(X,d_X)$ is a Hadamard space and
Theorem~\ref{thm:MNext}, we need to prove the existence of a ``base
graph" that satisfies the assumption~\eqref{eq:base graph
assumption}.

Below it will be convenient to give a name to the implicit universal
constant in~\eqref{eq:from MN13}. Thus let $\alpha\in (1,\infty)$ be
such that for every $CAT(0)$ metric space $(Y,d_Y)$, every $m,n\in \N$ and
every $n$ by $n$ symmetric and stochastic matrix $M$, we have
\begin{equation}\label{eq:from MN13-alpha}
\gamma_+\!\left(\frac{1}{m}\sum_{t=0}^{m-1}M^t,d_Y^2\right)
\le\alpha \max\left\{1,\frac{\gamma_+(M,d_Y^2)}{m}\right\}.
\end{equation}

Fix from now on any sequence  $\{W_k\}_{k=1}^\infty$ of  $3$-regular
 graphs of even cardinality that forms an expander sequence, i.e., $|V_{W_k}|$ is even for every $k\in \N$,
$\lim_{k\to \infty} |V_{W_k}|=\infty$ and $\sup_{k\in
\N}\lambda_2(W_k)<1$. By Lemma~\ref{lem:matousek extrpolation plus snowflake}  we have $\sup_{k\in \N}\gamma(W_k,d_{L_1}^2)<\infty$. By Lemma~\ref{lem:from MN plus replacement} there exists a
sequence of $3$-regular graphs $\{W_k^*\}_{k=1}^\infty$ with
\begin{equation}\label{eq:W_k^* size}
\forall\, k\in \N,\qquad \left|V_{W_k^*}\right|=54\left|V_{W_k}\right|,
\end{equation}
such that
\begin{equation}\label{eq:def beta sup}
\beta\eqdef \sup_{k\in \N} \gamma_+\!\left(W_k^*,d_{L_1}^2\right)<\infty.
\end{equation}

Choose the minimum $k\in \N$ such that
\begin{equation}\label{eq:the large size choice}
|V_{W_k^*}|\ge 3^{48\alpha\beta^2 K^4}.
\end{equation}
Note that since $\alpha,\beta,K$ are all universal constants, so are
$k$ and $|V_{W_k^*}|$.

Define a subfamily $\F_K^*$ of the graph family $\F_K$ by
\begin{equation}\label{eq:def F_K^*}
\F_K^*\eqdef \left\{G\in \F_K:\ |V_G|\ge 4 |V_{W_k^*}|^2\right\}.
\end{equation}
Fix $G\in \F_K^*$ and two mappings
\begin{equation}\label{eq:f and p def}
\f,\p:V_{W_k^*}\to \cone
\left(\Sigma(G),\frac{2\pi}{\girth(G)}\cdot d_{\Sigma(G)}\right).
\end{equation}
Observe that
$$
\left|\f\left(V_{W_k^*}\right)\cup \p\left(V_{W_k^*}\right)\right|
\le 2\left|V_{W_k^*}\right|\stackrel{\eqref{eq:def F_K^*}}{\le} \sqrt{|V_G|}.
$$
Consequently, there exists $S\subseteq \Sigma(G)$ with $|S|\le
\sqrt{|V_G|}$ such that
$$
\f\left(V_{W_k^*}\right)\cup \p\left(V_{W_k^*}\right)\subset \cone
\left(S,\frac{2\pi}{\girth(G)}\cdot d_{\Sigma(G)}\right).
$$
By assertion (2) of Definition~\ref{def:good family}, it follows
that
\begin{equation}\label{eq:use c1 K bound}
c_1\left(\f\left(V_{W_k^*}\right)\cup
\p\left(V_{W_k^*}\right),d_{\cone
\left(\Sigma(G),\frac{2\pi}{\girth(G)}\cdot d_{\Sigma(G)}\right)}\right)\le K.
\end{equation}
Since~\eqref{eq:use c1 K bound} holds true for every $\f,\psi$ as in~\eqref{eq:f and p def}, it follows from the definition of $\beta$ in~\eqref{eq:def beta sup} that
\begin{equation}\label{eq:gamma+ wrt big G}
\gamma_+\!\left(W_k^*,d_{\cone
\left(\Sigma(G),\frac{2\pi}{\girth(G)}\cdot d_{\Sigma(G)}\right)}^2\right)\le \beta K^2.
\end{equation}

Recalling the definitions~\eqref{eq:def U_F} and~\eqref{eq:def
metric on U_F}, $U_{\F_K^*}$ is the disjoint union of $\{\Sigma(G)\}_{G\in \F_K^*}$, where the distances (in terms of $d_{\F_K^*}$, which is the same as the restriction to $U_{\F_K^*}$ of $d_{{\F_K}}$) between different sets in this disjoint union being at least $2\pi$. We are therefore allowed to use Lemma~\ref{lem:cone pi separated multi union}, thus
 concluding from~\eqref{eq:gamma+ wrt big G} that
\begin{equation}\label{eq:check size condition to use zigzag
theorem}
\gamma_+\!\left(W_k^*,d_{\cone
\left(U_{\F_K^*},d_{\F_K^*}\right)}^2\right)\le 2\beta K^2\stackrel{\eqref{eq:the large size choice}}{\le}
\sqrt{\frac{1}{2\alpha}\left\lfloor\frac{\log\left|V_{W_k^*}\right|}{3\log 3}\right\rfloor} .
\end{equation}

By the definition of $\alpha$ in~\eqref{eq:from MN13-alpha}, due
to~\eqref{eq:check size condition to use zigzag theorem} (which
corresponds to condition~\eqref{eq:base graph assumption}) we can
apply Theorem~\ref{thm:zigzag tool} to the $CAT(0)$ metric space
$\cone \left(U_{\F_K^*},d_{\F_K^*}\right)$ and the ``base graph"
$H=W_k^*$, thus obtaining a sequence of $3$-regular graphs
$\{F_n\}_{n=1}^\infty$ with

\begin{equation}\label{eq:G_n size exponential}
\forall\, n\in \N,\qquad \left|V_{F_n}\right|=81\left|V_{W_k^*}\right|^n\stackrel{\eqref{eq:W_k^* size}}{=}81\cdot \left(54\left|V_{W_k}\right|\right)^n,
\end{equation}
and
\begin{equation}\label{eq:new graphs are good for star class}
\sup_{n\in \N} \gamma_+\!\left(F_n,d_{\cone
\left(U_{\F_K^*},d_{\F_K^*}\right)}^2\right)<\infty.
\end{equation}

Finally, define $G_n\eqdef F_n^*$, where $F_n^*$ is obtained from $F_n$ by Lemma~\ref{lem:from MN plus replacement} (we are allowed to apply Lemma~\ref{lem:from MN plus replacement} here since by~\eqref{eq:G_n size exponential} we know that $|V_{F_n}|$ is even). Then $\left|V_{G_n}\right|=54\left|V_{F_n}\right|$ for every $n\in \N$. Moreover, by Lemma~\ref{lem:from MN plus replacement} for every $n\in \N$ we have
\begin{multline}\label{eq:G_n is good for F_K^*}
\sup_{n\in \N}\gamma_+\!\left(G_n,d_{\cone
\left(U_{\F_K^*},d_{\F_K^*}\right)}^2\right)\lesssim \sup_{n\in \N}\gamma\!\left(F_n,d_{\cone
\left(U_{\F_K^*},d_{\F_K^*}\right)}^2\right)\\\le \sup_{n\in \N} \gamma_+\!\left(F_n,d_{\cone
\left(U_{\F_K^*},d_{\F_K^*}\right)}^2\right)<\infty.
\end{multline}
Due to~\eqref{eq:new graphs are good for star class} the graphs
$\{F_n\}_{n=1}^\infty$ are necessarily also classical expanders, i.e.,
$\sup_{n\in \N} \lambda_2(F_n)<1$. By another application of Lemma~\ref{lem:matousek extrpolation plus snowflake} we therefore conclude that $\sup_{n\in \N}\gamma(F_n,d_{L_1}^2)<\infty$, and by Lemma~\ref{lem:from MN plus replacement} this implies that
\begin{equation}\label{eq:G_n is good for L_1}
\sup_{n\in \N}\gamma_+(G_n,d_{L_1}^2)\lesssim \sup_{n\in \N}\gamma(F_n,d_{L_1}^2)<\infty.
\end{equation}

If $H\in \F_K\setminus \F_K^*$ then $|V_H|<4|V_{W_k^*}|^2$. Hence,
as explained in Section~\ref{sec:prem bi-lip} we have $$\sup_{H\in \F_K\setminus \F_K^*} c_1(\Sigma(H))<\infty.$$ By Corollary~\ref{cor:not harder to embed cones} we therefore have
\begin{equation}\label{eq:cones on small graphs are in L1}
\sup_{H\in \F_K\setminus \F_K^*} c_1\left(\cone\left(\Sigma(H),\frac{2\pi}{\girth(H)}\cdot d_{\Sigma(H)}\right)\right)<\infty.
\end{equation}
In conjunction with~\eqref{eq:G_n is good for L_1}, it follows from~\eqref{eq:cones on small graphs are in L1} that
\begin{equation}\label{eq:good gamm+ on small graphs}
\sup_{H\in \F_K\setminus \F_K^*} \sup_{n\in \N}\gamma_+\!\left(G_n,d_{\cone
\left(\Sigma(H),\frac{2\pi}{\girth(H)}\cdot d_{\Sigma(H)}\right)}^2\right)<\infty.
\end{equation}

Recalling the definitions~\eqref{eq:def U_F} and~\eqref{eq:def
metric on U_F}, $U_{\F_K}$ is the disjoint union of $U_{\F_K^*}$ and $\{\Sigma(H)\}_{H\in \F_K\setminus \F_K^*}$ with the distances (in terms of $d_{\F_K}$) between different sets in this disjoint union being at least $2\pi$. We can therefore apply Lemma~\ref{lem:cone pi separated multi union} to~\eqref{eq:G_n is good for F_K^*} and~\eqref{eq:good gamm+ on small graphs}, yielding assertion {\it (1)} of Theorem~\ref{thm:main3}.

It remains to justify the second displayed equation in assertion
{\it (4)} of Theorem~\ref{thm:main3}. This is done following the
strategy that was sketched in Section~\ref{sec:union sketch}. Let
$H$ be a connected $3$-regular $n$-vertex graph that satisfies the
conclusion of Proposition~\ref{prop:structure cone random}. By
Proposition~\ref{prop:structure cone random} we know that the
$\mathcal{G}_{n,d}$-probability that this occurs is at least
$1-C'(d)/\sqrt[3]{n}$. Continuing here with the notation of
Proposition~\ref{prop:structure cone random}, it follows from
assertion {\it (1)} of Theorem~\ref{thm:main3} (which we already
proved), combined with assertion  $(IV)$ of
Proposition~\ref{prop:structure cone random} that
\begin{equation}\label{eq:gamma+ for cone over A2-new}
\sup_{n\in \N} \gamma_+\!\left(G_n,d_{\cone\left(A_2,\sigma d_{\Sigma(H)}\right)}^2\right)<\infty.
\end{equation}
Due to~\eqref{eq:G_n is good for L_1},   assertion $(III)$ of Proposition~\ref{prop:structure cone
random} implies that
\begin{equation}\label{eq:gamma+ for
cone over A1-new}
\sup_{n\in \N} \gamma_+\!\left(G_n,d_{\cone\left(A_1,\sigma d_{\Sigma(H)}\right)}^2\right)<\infty.
\end{equation}
By assertions $(I)$ and $(II)$ of Proposition~\ref{prop:structure
cone random}, we may use Lemma~\ref{lem:cone union} to deduce
from~\eqref{eq:gamma+ for cone over A2-new} and~\eqref{eq:gamma+ for
cone over A1-new} that
$$
\sup_{n\in \N} \gamma_+\!\left(G_n,d_{\cone\left(\Sigma(H),\sigma
 d_{\Sigma(H)}\right)}^2\right)<\infty.
$$
By assertion $(V)$ of Proposition~\ref{prop:structure cone random}
it follows that
\begin{equation*}
\sup_{n\in \N} \gamma_+\!\left(G_n,d_{\Sigma(H)}^2\right)<\infty.\qedhere
\end{equation*}
\end{proof}

\section{On the Lipschitz structure  of Euclidean cones}

Here we study various aspects of the Lipschitz structure of Euclidean cones. In particular, we will prove Proposition~\ref{thm:cone in L_1} and Lemma~\ref{lem:cone union}. We will also establish several additional geometric results of independent interest, yielding as a side product progress on a question that we raised in~\cite{MN-quotients}. Despite the fact that in this paper the value of universal constants is mostly insignificant, when proving results such as Proposition~\ref{thm:cone in L_1} we will attempt to state reasonably good (though still suboptimal) explicit distortion bounds, because such geometric results are interesting in their own right and might be useful elsewhere.

Before proceeding we record for future use the following simple facts.

\begin{lemma}\label{lem:sinus}
Let $(X,d_X)$ be a metric space. For every $s,t\in (0,\infty)$ and $x,y\in X$ we have
\begin{equation}\label{eq:cos^2}
d_{\cone(X)}\big((s,x),(t,y)\big)\ge \max\{s,t\}\cdot \sin\left(\min\left\{\frac{\pi}{2},d_X(x,y)\right\}\right).
\end{equation}
\end{lemma}

\begin{proof}
If $d_X(x,y)\ge \pi/2$ then
$\cos\left(\min\{\pi,d_X(x,y)\}\right)\le 0$ and it follows
from~\eqref{eq:def cone} that $d_{\cone(X)}\big((s,x),(t,y)\big)\ge
\sqrt{s^2+t^2}\ge \max\{s,t\}$. We may therefore assume that
$d_X(x,y)<\pi/2$ and $s\ge t$. The minimum of the function $t\mapsto
s^2+t^2-2st\cos \left(d_X(x,y)\right)$ is attained at $t=s\cos
\left(d_X(x,y)\right)$, implying~\eqref{eq:cos^2}.
\end{proof}

\begin{lemma}\label{lem:rescaling}
Let $(X,d_X)$ be a bounded metric space. Suppose that $f:X\to (0,\infty)$ is Lipschitz and define $F:\cone(X)\to \cone(X)$ by
$$
F(s,x)\eqdef \big(f(x)s,x\big).
$$
Then $F$ is Lipschitz with
$$
\|F\|_{\Lip}\le \sqrt{\diam(X)^2\cdot \|f\|_{\Lip}^2+2\|f\|_\infty^2}.
$$
\end{lemma}

\begin{proof}
Fix $x,y\in X$ and set $\theta\eqdef
\min\left\{\pi,d_X(x,y)\right\}$. For  $s,t\in (0,\infty)$ with
$s\le t$ we have
\begin{align}\label{eq:no pwer 2}
\nonumber|f(x)s-f(y)t|&\le |f(x)-f(y)|s+|f(y)|\cdot|s-t|\\&\nonumber\le \|f\|_{\Lip}d_X(x,y)s+\|f\|_\infty|s-t|\\&\nonumber\le \frac{\diam(X)\|f\|_{\Lip}}{\pi}\theta\sqrt{st}+\|f\|_\infty |s-t|\\
&\le \frac{\diam(X)\|f\|_{\Lip}}{\sqrt{2}}\sqrt{st(1-\cos\theta)}+\|f\|_\infty |s-t|,
\end{align}
where in~\eqref{eq:no pwer 2} we used~\eqref{eq:cosine
inequalities}. By squaring~\eqref{eq:no pwer 2} we see that
\begin{equation}\label{eq:explicit squaring for manor}
\left(f(x)s-f(y)t\right)^2\le 2\|f\|_\infty^2(s-t)^2+ \diam(X)^2\|f\|_{\Lip}^2st(1-\cos\theta),
\end{equation}
and therefore,
\begin{align*}
&d_{\cone(X)}\big(F(s,x),F(t,y)\big)^2\\&\stackrel{\eqref{eq:1-cos identity}}{=} \left(f(x)s-f(y)t\right)^2+2f(x)f(y)st(1-\cos\theta)\\
&\stackrel{\eqref{eq:explicit squaring for manor}}{\le}2\|f\|_\infty^2(s-t)^2+\left(\diam(X)^2\|f\|_{\Lip}^2+2\|f\|_\infty^2\right)st(1-\cos\theta)\\
&\stackrel{\eqref{eq:1-cos identity}}{\le} \left(\diam(X)^2\|f\|_{\Lip}^2+2\|f\|_\infty^2\right)d_{\cone(X)}\big((s,x),(t,y)\big)^2.\qedhere
\end{align*}
\end{proof}

\begin{lemma}\label{lem:comparison to max}
Let $(X,d_X)$ be a metric space and suppose that $(s,x)$ and $(t,y)$ are distinct points in $\cone(X)$. Then
$$
\frac13\le \frac{d_{\cone(X)}\big((s,x),(t,y)\big)}{\max\left\{|s-t|,\max\{s,t\}\cdot \sqrt{2\left(1-\cos\left(\min\{\pi,d_X(x,y)\}\right)\right)}\right\}}\le \sqrt{2}.
$$
\end{lemma}

\begin{proof} Denoting $\theta\eqdef \min\left\{\pi,d_X(x,y)\right\}$, it follows from~\eqref{eq:1-cos identity} that
\begin{equation}\label{eq:s-t version}
d_{\cone(X)}\big((s,x),(t,y)\big)=\sqrt{(s-t)^2+2st(1-\cos\theta)}.
\end{equation}
The desired upper bound on $d_{\cone(X)}\big((s,x),(t,y)\big)$ is therefore immediate from~\eqref{eq:s-t version}, using $\sqrt{st}\le \max\{s,t\}$. Since $\sqrt{st}\ge \max\{s,t\}-|s-t|$, it follows from~\eqref{eq:s-t version} that
\begin{align*}
&d_{\cone(X)}\big((s,x),(t,y)\big)\\& \ge \max\left\{|s-t|,\max\{s,t\}\cdot\sqrt{2(1-\cos\theta)}-|s-t|\sqrt{2(1-\cos\theta)}\right\}\\
&\ge \frac{\max\left\{|s-t|,\max\{s,t\}\cdot
\sqrt{2(1-\cos\theta)}\right\}}{1+\sqrt{2(1-\cos\theta)}},
\end{align*}
yielding the desired lower bound on $d_{\cone(X)}\big((s,x),(t,y)\big)$.
\end{proof}

\subsection{The Euclidean cone over $L_1$}\label{sec:cone L1}
Our goal here is to prove Proposition~\ref{thm:cone in L_1}. In
preparation for this  we first show that if one equips  $L_1$ with
the metric $\rho(x,y)=\min\{1,\|x-y\|_1\}$, then the resulting
metric space admits a bi-Lipschitz embedding into $L_1$.
In~\cite{MN-quotients} we showed that the metric space
$(L_2,\min\{1,\|x-y\|_2\})$ admits a bi-Lipschitz embedding into
$L_2$ but left open the question whether
$c_p(L_p,\min\{1,\|x-y\|_p\})<\infty$ for $p\in [1,2)$;
Lemma~\ref{lem:L1 truncation} below shows that this is indeed the
case when $p=1$, but the question remains open for $p\in (1,2)$. As
explained in~\cite[Rem.~5.12]{MN-quotients}, for every $p\in
(2,\infty]$ the metric space $(L_p,\min\{1,\|x-y\|_p\})$ does not
admit a bi-Lipschitz embedding into $L_q$ for any $q\in [1,\infty)$.

\begin{lemma}[truncated $L_1$ embeds into $L_1$]\label{lem:L1 truncation} For every $M\in (0,\infty)$ there exists a mapping $T_M:L_1\to L_1$ such that
\begin{itemize}
\item $\|T_M(x)\|_1=M$ for every $x\in L_1$,
\item For every distinct $x,y\in L_1$,
\begin{equation}\label{eq:T_M property}
1-\frac{1}{e}\le \frac{\|T_M(x)-T_M(y)\|_1}{\min\{M,\|x-y\|_1\}}\le 1.
\end{equation}
\end{itemize}
\end{lemma}

\begin{proof}
Fix an integer $n\ge 2$. For every $A\subseteq \{1,\ldots, n\}$ the Walsh function
$W_A:\{0,1\}^n\to \{-1,1\}$ is defined as usual by
$$
\forall\, z=(z_1,\ldots,z_n)\in \{0,1\}^n,\qquad W_A(z)\eqdef \prod_{j\in A} (-1)^{z_j}.
$$
For $\lambda\in [1/2,1]$  observe that every $x,y\in \{0,1\}^n$
satisfy
\begin{align}\label{eq:walsh identity}
&1-(2\lambda-1)^{\|x-y\|_1}=1-\prod_{j=1}^n \left(\lambda+(-1)^{x_j+y_j}(1-\lambda)\right)\nonumber\\
&= \sum_{A\subseteq \{1,\ldots,n\}} \lambda^{n-|A|}(1-\lambda)^{|A|}\left(1-W_A(x+y)\right)\nonumber\\
&= \sum_{A\subseteq \{1,\ldots,n\}} \lambda^{n-|A|}(1-\lambda)^{|A|}\left|W_A(x)-W_A(y)\right|.
\end{align}
Let $\{h_A\}_{A\subseteq \{1,\ldots,n\}}\subseteq L_1$ be disjointly
supported functions with unit $L_1$ norm, and define
$\Phi_M^n:\{0,1\}^n\to L_1$ by setting for every $z\in \{0,1\}^n$,
$$
\Phi^n_M(z)\eqdef M\sum_{A\subseteq \{1,\ldots,n\}} \left(\frac{1+e^{-1/M}}{2}\right)^{n-|A|}\left(\frac{1-e^{-1/M}}{2}\right)^{|A|} W_A(z)h_A.
$$
Then $\|\Phi_M^n(z)\|_1=M$ for every $z\in \{0,1\}^n$, and it follows from~\eqref{eq:walsh identity} that for every $x,y\in \{0,1\}^n$ we have
\begin{equation}\label{eq:Phi_n conditions}
\frac{\|\Phi^n_M(x)-\Phi^n_M(y)\|_1}{\min\{M,\|x-y\|_1\}}=\frac{M\left(1-e^{-\|x-y\|_1/M}\right)}{\min\{M,\|x-y\|_1\}}\in\left[1-\frac{1}{e},1\right].
\end{equation}

This proves the
existence of the desired embedding for the metric space
$(\{0,1\}^n,\min\{M,\|\cdot\|_1\})$; we pass to all of $L_1$ via the
following standard approximation argument. If $S\subseteq L_1$ is finite and $\e\in (0,1)$ then by approximating by step functions we can choose $d,K,Q\in \N$ and $k_1,\ldots k_d:S\to \{1,\ldots,K\}$ such that for every $x,y\in S$,
\begin{equation}\label{eq:rational approx}
\frac{1}{Q}\sum_{j=1}^d |k_i(x)-k_i(y)|\le \|x-y\|_1\le \frac{1+\e}{Q}\sum_{j=1}^d |k_i(x)-k_i(y)|.
\end{equation}
Write $n=dK$ and define $\psi:S\to \{0,1\}^n$ by
$$
\forall\, x\in S,\qquad \psi(x)\eqdef \sum_{i=1}^d\sum_{j=(i-1)K+1}^{(i-1)K+k_i(x)} e_j,
$$
where $e_1,\ldots,e_n$ is the standard basis of $\R^n$. Then by~\eqref{eq:rational approx},
\begin{equation}\label{eq:almost embedding into cube}
\forall\, x,y\in S,\qquad\frac{Q}{1+\e}\|x-y\|_1\le \|\psi(x)-\psi(y)\|_1\le Q\|x-y\|_1.
\end{equation}
Define  $F_\e:S\to L_1$ by
$$
F_\e\eqdef \frac{1}{Q}\Phi^n_{QM}\circ \psi.
$$
Then $\|F_\e(x)\|_1=M$ for every
$x\in S$, and by~\eqref{eq:Phi_n conditions} and~\eqref{eq:almost embedding
into cube} we have
$$
\frac{1}{1+\e}\left(1-\frac1{e}\right)\le \frac{\|F_\e(x)-F_\e(y)\|_1}{\min\{M,\|x-y\|_1\}}\le 1
$$
for all distinct $x,y\in S$. Since this holds for every finite
$S\subset L_1$ and every $\e\in (0,1)$, the existence of the mapping
$T_M$ from the statement of Lemma~\ref{lem:L1 truncation} now
follows from a standard ultrapower argument (see~\cite[Ch.~9]{Kap09}
for background on ultrapowers of metric spaces), using the fact that
any ultrapower of $L_1$ is an $L_1(\mu)$ space~\cite{Hei80}.
\end{proof}

\begin{remark}
As explained in~\cite[Rem.~5.4]{MN-quotients}, it follows formally
from Lemma~\ref{lem:L1 truncation} that there exists a universal
constant $C\in (0,\infty)$ such that if $\omega:[0,\infty)\to
[0,\infty)$ is a concave nondecreasing function with $\omega(0)=0$
and $\omega(t)>0$ for $t>0$ then the metric space
$(L_1,\omega(\|x-y\|_1))$ embeds into $L_1$ with distortion at most
$C$. The analogous assertion with $L_1$ replaced by $L_p$ for $p\in
(1,2)$ remains open.
\end{remark}

\begin{proof}[Proof of Proposition~\ref{thm:cone in L_1}]
Let $T_\pi:L_1\to L_1$ be the mapping constructed in
Lemma~\ref{lem:L1 truncation} (with $M=\pi$). Define
$F:\cone(L_1)\to L_1$ by
\begin{equation}\label{eq:def F truncation}
\forall(s,x)\in (0,\infty)\times L_1,\qquad F(s,x)\eqdef sT_\pi(x).
\end{equation}
Fix $(s,x),(t,y)\in (0,\infty)\times L_1$ and assume without loss of
generality that $t\ge s$. Recalling~\eqref{eq:1-cos identity}, we have
\begin{equation}\label{eq:write cone with theta}
d_{\cone(L_1)}\big((s,x),(t,y)\big)=\sqrt{(s-t)^2+2st(1-\cos\theta)},
\end{equation}
where
\begin{equation}\label{eq:def theta}
\theta\eqdef\min \{\pi,\|x-y\|_1\}.
\end{equation}
Moreover,
\begin{equation}\label{eq:norms xy}
\|T_\pi(x)\|_1=\|T_\pi(y)\|_1=\pi
\end{equation}
 and~\eqref{eq:T_M property} becomes
\begin{equation}\label{eq:T_pi property theta}
\left(1-\frac1{e}\right)\theta\le \|T_\pi(x)-T_\pi(y)\|_1\le \theta.
\end{equation}
Using also~\eqref{eq:cosine inequalities}, we therefore have
\begin{align}
\|F(s,x)-F(t,y)\|_1 &=\|sT_\pi(x)-tT_\pi(y)\|_1 \label{eq:use F def}
\\ &\le s\|T_\pi(x)-T_\pi(y)\|_1+(t-s)\|T_\pi(y)\|_1\label{eq:use triangle ineq truncation}\\
&\le s\theta\label{eq:use T_pi}+\pi(t-s)\\
&\le\frac{\pi }{2}\sqrt{2st(1-\cos\theta)}+\pi(t-s)\label{eq:use lower cosinus}\\
&\le  \frac{\pi\sqrt{5}}{2}\sqrt{2st(1-\cos\theta)+(s-t)^2}\label{eq:use Cauchy-schw}
\\&= \frac{\pi\sqrt{5}}{2}  d_{\cone(L_1)}\big((s,x),(t,y)\big)\label{eq:use theta form cone},
\end{align}
where in~\eqref{eq:use F def} we used~\eqref{eq:def F truncation},  in~\eqref{eq:use triangle ineq truncation} we used the triangle inequality in $L_1$, in~\eqref{eq:use T_pi} we used~\eqref{eq:norms xy} and  the rightmost inequality in~\eqref{eq:T_pi property theta}, in~\eqref{eq:use lower cosinus} we used the leftmost inequality in~\eqref{eq:cosine inequalities} and that $s\le \sqrt{st}$, in~\eqref{eq:use Cauchy-schw} we used the Cauchy-Schwarz inequality, and in~\eqref{eq:use theta form cone} we used~\eqref{eq:write cone with theta}.

Next, we claim that
\begin{equation}\label{eq:F colip}
\|F(s,x)-F(t,y)\|_1\ge \frac{\pi(e-1)}{\sqrt{4\pi^2e^2+(e-1)^2}} d_{\cone(L_1)}\big((s,x),(t,y)\big).
\end{equation}
Once proved, \eqref{eq:F colip} in conjunction with~\eqref{eq:use theta form cone} would imply that
\begin{equation}\label{eq:11.17}
c_1\big(\cone(L_1)\big)\le \frac{\sqrt{20\pi^2e^2+5(e-1)^2}}{2(e-1)}\le 11.17,
\end{equation}
thus completing the proof of Proposition~\ref{thm:cone in L_1}.

We prove~\eqref{eq:F colip} by distinguishing between two cases as follows.

\medskip

\noindent{\bf Case 1:} $t-s\ge \frac{e-1}{2\pi e}t\theta$. Since $t\ge \sqrt{st}$ this assumption combined with the rightmost inequality appearing in~\eqref{eq:cosine inequalities} implies that
\begin{equation}\label{eq:t-s big}
t-s\ge  \frac{e-1}{2\pi e}\sqrt{2st(1-\cos\theta)}.
\end{equation}
Consequently,
\begin{equation}\label{upper cone for case 1}
d_{\cone(L_1)}\big((s,x),(t,y)\big)\stackrel{\eqref{eq:write cone with theta}\wedge\eqref{eq:t-s big}}{\le} (t-s)\frac{\sqrt{4\pi^2e^2+(e-1)^2}}{e-1},
\end{equation}
implying the desired inequality~\eqref{eq:F colip} as follows.
\begin{eqnarray*}
\|F(s,x)-F(t,y)\|_1&\stackrel{\eqref{eq:def F truncation}}{=}&\|sT_\pi(x)-tT_\pi(y)\|_1\\& \ge& t\|T_\pi(y)\|_1-s\|T_\pi(x)\|_1\\&\stackrel{\eqref{eq:norms xy}}{=}&\pi(t-s)\\&\stackrel{\eqref{upper cone for case 1}}{\ge}& \frac{\pi(e-1)}{\sqrt{4\pi^2e^2+(e-1)^2}}  d_{\cone(L_1)}\big((s,x),(t,y)\big).
\end{eqnarray*}

\medskip

\noindent{\bf Case 2:} we have
\begin{equation}\label{eq:t-s small}
t-s< \frac{e-1}{2\pi e}t\theta.
\end{equation}
It follows that
\begin{eqnarray}\label{eq:lower colip case 2}
\|F(s,x)-F(t,y)\|_1&\stackrel{\eqref{eq:def F truncation}}{=}&\|sT_\pi(x)-tT_\pi(y)\|_1 \nonumber\\&\ge& t\|T_\pi(y)-T_\pi(x)\|_1-(t-s)\|T_\pi(x)\|_1\nonumber\\&\stackrel{\eqref{eq:T_pi property theta}\wedge\eqref{eq:norms xy}}{\ge}& \left(1-\frac1{e}\right)t\theta-\pi(t-s)\nonumber\\&\stackrel{\eqref{eq:t-s small}}
>&\frac{e-1}{2e}t\theta.
\end{eqnarray}
At the same time, since $st\le t^2$, it follows from~\eqref{eq:write cone with theta}, combined with~\eqref{eq:t-s small}  and the rightmost inequality appearing in~\eqref{eq:cosine inequalities}, that
\begin{equation}\label{eq:t theta}
d_{\cone(L_1)}\big((s,x),(t,y)\big)\le t\theta\frac{\sqrt{4\pi^2e^2+(e-1)^2}}{2\pi e}.
\end{equation}
In combination with~\eqref{eq:lower colip case 2}, it follows
from~\eqref{eq:t theta} that the desired inequality~\eqref{eq:F
colip} holds true in Case 2 as well, thus completing the proof of
Proposition~\ref{thm:cone in L_1}.
\end{proof}

\subsection{Snowflakes of cones}\label{sec:snowflake}

For $\alpha\in (0,1]$, the $\alpha$-snowflake of a metric space
$(X,d)$ is defined to be the metric space $(X,d^\alpha)$.  It is
well known that the $\frac12$-snowflake of $L_1$ embeds
isometrically into $L_2$ (see e.g.~\cite{DZ} or~\cite[Sec.~3]{Nao10}). Consequently,
Proposition~\ref{thm:cone in L_1} implies that the
$\frac12$-snowflake of $\cone(L_1)$ admits a bi-Lipschitz distortion
into $L_2$. Proposition~\ref{prop:snowflake} below yields an
alternative proof of this fact. Note that
Proposition~\ref{prop:snowflake} does not imply
Proposition~\ref{thm:cone in L_1} since by~\cite{KV04} there exist
metric spaces which do not admit a bi-Lipschitz embedding into $L_1$
yet their $\frac12$-snowflake admits an isometric embedding into
$L_2$.

In the proof of Theorem~\ref{thm:main3} (most significantly, in the proofs of Lemma~\ref{lem:random graph in good family} and Proposition~\ref{prop:structure cone random}) we use Proposition~\ref{thm:cone in L_1}, but one can use mutatis mutandis Proposition~\ref{prop:snowflake} instead. We include Proposition~\ref{prop:snowflake} here due to its intrinsic interest, but it will not be used in this paper other than as an indication of an alternative approach to the parts of the proof of Theorem~\ref{thm:main3}  eluded to above (note, however, that the approach below yields a better bound; see Remark~\ref{rem:better bound}). Readers who familiarized themselves with the contents of Section~\ref{sec:cone L1} can skip the present section if they  only wish to understand the proof of Theorem~\ref{thm:main3}.

\begin{proposition}\label{prop:snowflake} For every separable metric space $(X,d_X)$ and every $\alpha\in (0,1]$ we have
\begin{equation}\label{eq:snowflake cone}
c_2\left(\cone(X),d_{\cone(X)}^\alpha\right)\lesssim c_2\left(X,d_X^\alpha\right).
\end{equation}
\end{proposition}

\begin{proof}
By a classical theorem of Schoenberg~\cite{Sch38} (see also~\cite{WW75}) there exists a mapping $h_\alpha:\R\to L_2$ with $h_\alpha(0)=0$ such that
\begin{equation}\label{eq:helix}
\forall\, s,t\in \R,\qquad \|h_\alpha(s)-h_\alpha(t)\|_2=|s-t|^\alpha.
\end{equation}
Also, as shown in~\cite[Lem.~5.2]{MN-quotients}, there exists
$\tau_\alpha:L_2\to L_2$ such that every distinct $x,y\in L_2$
satisfy

\begin{equation}\label{eq:truncation L2 1}
\|\tau_\alpha(x)\|_2=\|\tau_\alpha(y)\|_2=\frac{\pi^\alpha}{\sqrt{2}},
\end{equation}
and
\begin{equation}\label{eq:truncation L2 2}
\sqrt{1-\frac{1}{e}}\le \frac{\|\tau_\alpha(x)-\tau_\alpha(y)\|_2}{\min\{\pi^\alpha,\|x-y\|_2\}}\le 1.
\end{equation}

Writing $D\eqdef c_2\left(X,d_X^\alpha\right)$, there exists $f:X\to L_2$ such that
\begin{equation}\label{eq:snoflake embed D}
\forall\, x,y\in X\qquad \frac{d_X(x,y)^\alpha}{D}\le \|f(x)-f(y)\|_2\le d_X(x,y)^\alpha.
\end{equation}
We can now define an embedding
$
\phi:\cone(X)\to L_2\otimes L_2
$
by
\begin{equation}\label{eq:def tensor phi}
\forall(s,x)\in (0,\infty)\times X,\qquad \phi(s,x)\eqdef h_\alpha(s)\otimes \tau_\alpha(f(x)).
\end{equation}

Here the tensor product $L_2\otimes L_2$ is equipped with the Hilbert space tensor product (Hilbert-Schmidt) norm. Thus $L_2\otimes L_2$ is isometric to $L_2$ and $\|x\otimes y\|_{L_2\otimes L_2}=\|x\|_2\cdot\|y\|_2$ for every $x,y\in L_2$. Consequently, for every $a,b,x,y\in L_2$ we have
\begin{multline}\label{eq:tensor identity}
\left\|a\otimes x-b\otimes y\right\|_{L_2\otimes L_2}^2=\|a\|_2^2\cdot\|x\|_2^2+\|b\|_2^2\cdot\|y\|_2^2\\-\frac12\left(\|a\|_2^2+\|b\|_2^2-\|a-b\|_2^2\right)\left(\|x\|_2^2+\|y\|_2^2-\|x-y\|_2^2\right).
\end{multline}

Fix $x,y\in X$ and $s,t\in (0,\infty)$ with $t\ge s$. Since
$h_\alpha(0)=0$ it follows from~\eqref{eq:helix} that
$\|h_\alpha(s)\|_2=s^\alpha$ and $\|h_\alpha(t)\|_2=t^\alpha$. In
conjunction with~\eqref{eq:truncation L2 1}, this shows that due
to~\eqref{eq:helix}, \eqref{eq:def tensor phi} and~\eqref{eq:tensor
identity},
\begin{align}\label{eq:use tensor identity}
&\|\phi(s,x)-\phi(t,y)\|_2^2\\&=\frac{\pi^{2\alpha}}{2}\left(s^{2\alpha}+t^{2\alpha}\right)\nonumber\\\nonumber&\qquad
-\frac12\left(s^{2\alpha}+t^{2\alpha}-(t-s)^{2\alpha}\right)\left(\pi^{2\alpha}-
\left\|\tau_\alpha(f(x))-\tau_\alpha(f(y))\right\|_2^2\right)\\
&=\frac{\pi^{2\alpha}}{2}(t-s)^{2\alpha}+\frac{s^{2\alpha}
+t^{2\alpha}-(t-s)^{2\alpha}}{2}\left\|\tau_\alpha(f(x))-\tau_\alpha(f(y))\right\|_2^2.
\nonumber
\end{align}

Let $\theta\in [0,\pi]$ be given by
\begin{equation}\label{eq:new def theta}
\theta\eqdef \min\{\pi,d_X(x,y)\}.
\end{equation}
Recalling~\eqref{eq:1-cos identity}, we therefore have
\begin{equation}\label{eq:write cone with theta2}
d_{\cone(X)}\big((s,x),(t,y)\big)=\sqrt{(s-t)^2+2st(1-\cos\theta)},
\end{equation}
and
\begin{equation}\label{eq;upper snowflke tau part}
\left\|\tau_\alpha(f(x))-\tau_\alpha(f(y))\right\|_2\stackrel{\eqref{eq:truncation L2 2}\wedge\eqref{eq:snoflake embed D}}{\le} \theta^\alpha\stackrel{\eqref{eq:cosine inequalities}}{\le} \frac{\pi^\alpha}{2^{\alpha/2}}(1-\cos\theta)^{\alpha/2}.
\end{equation}
To bound~\eqref{eq:use tensor identity} from above we also note the following elementary inequality, which holds for every $s\in [0,\infty)$, $t\in[s,\infty)$ and $\alpha\in [0,1]$.
\begin{equation}\label{eq:stalpha}
s^{2\alpha}+t^{2\alpha}-(t-s)^{2\alpha}\le 2(st)^\alpha.
\end{equation}
To verify the validity of~\eqref{eq:stalpha}, normalization by $t^{2\alpha}$ shows that it suffices to check that  $\psi(\alpha)= 2u^\alpha-u^{2\alpha}-1+(1-u)^{2\alpha}\ge 0$ for every fixed $u\in (0,1)$. Indeed, $\psi'(\alpha)=2u^\alpha(1-u^{\alpha})\log u+2(1-u)^{2\alpha}\log(1-u)\le 0$, so $\psi(\alpha)\ge \psi(1)=0$.

By substituting~\eqref{eq;upper snowflke tau part} and~\eqref{eq:stalpha} into~\eqref{eq:use tensor identity} we have
\begin{align}\label{eq:holder use alpha}
\nonumber\|\phi(s,x)-\phi(t,y)\|_2&\le \frac{\pi^\alpha}{2^{\alpha+1/2}}\sqrt{2^{2\alpha}|s-t|^{2\alpha}+(2st(1-\cos\theta))^\alpha}\\\nonumber
&\le \frac{\pi^\alpha\left(2^{2\alpha/(1-\alpha)}+1\right)^{1-\alpha}}{2^{\alpha+1/2}}d_{\cone(X)}\big((s,x),(t,y)\big)^\alpha\\
&\lesssim d_{\cone(X)}\big((s,x),(t,y)\big)^\alpha,
\end{align}
where the penultimate step of~\eqref{eq:holder use alpha} uses H\"older's inequality and~\eqref{eq:write cone with theta2}.

To bound~\eqref{eq:use tensor identity} from below, observe first that
\begin{multline}\label{eq;lower snowflke tau part}
\left\|\tau_\alpha(f(x))-\tau_\alpha(f(y))\right\|_2\stackrel{\eqref{eq:truncation L2 2}\wedge\eqref{eq:snoflake embed D}}{\ge} \sqrt{1-\frac1{e}}\cdot\min\left\{\pi^\alpha,\frac{d_X(x,y)^\alpha}{D}\right\}\\\stackrel{\eqref{eq:new def theta}}{\ge} \frac{\sqrt{e-1}}{D\sqrt{e}}\theta^\alpha\stackrel{\eqref{eq:cosine inequalities}}{\ge} \frac{2^{\alpha/2}\sqrt{e-1}}{D\sqrt{e}}(1-\cos\theta)^{\alpha/2}.
\end{multline}
Substituting~\eqref{eq;lower snowflke tau part} and the trivial estimate  $s^{2\alpha}+t^{2\alpha}-(t-s)^{2\alpha}\ge s^{2\alpha}$ into~\eqref{eq:use tensor identity} shows that
\begin{equation}\label{eq:lower phi}
\|\phi(s,x)-\phi(t,y)\|_2\ge \frac{\sqrt{\frac{\pi^{2\alpha}eD^2}{e-1}|s-t|^{2\alpha}+(2s^2(1-\cos\theta))^{\alpha}}}{D\left(\frac{2e}{e-1}\right)^{\alpha/2}}.
\end{equation}
But, recalling~\eqref{eq:write cone with theta2}, since $st\le 2s^2+(s-t)^2$  we have
\begin{multline}\label{eq:D colip}
d_{\cone(X)}\big((s,x),(t,y)\big)^\alpha\le \left(3(s-t)^2+2s^2(1-\cos\theta)\right)^{\alpha/2}\\
\le \sqrt{3^\alpha |s-t|^{2\alpha}+(2s^2(1-\cos\theta))^{\alpha}}
\stackrel{\eqref{eq:lower phi}}{\lesssim} D\|\phi(s,x)-\phi(t,y)\|_2.
\end{multline}
Due to~\eqref{eq:holder use alpha}  and~\eqref{eq:D colip} the proof of Proposition~\ref{prop:snowflake} is complete.
\end{proof}

\begin{remark}\label{rem:better bound}
When $\alpha=\frac12$ and $D=1$, a more careful analysis of the proof of Proposition~\ref{prop:snowflake} yields the estimate
$$
c_2\left(\cone(L_1),d_{\cone(L_1)}^{1/2}\right)\le 2.574.
$$
We omit the (tedious) proof of this estimate, noting that it is better than the bound that follows from an application of Proposition~\ref{thm:cone in L_1} (specifically, recall~\eqref{eq:11.17}) and the fact that the $\frac12$-snowflake of $L_1$ embeds isometrically into $L_2$.
\end{remark}

\subsection{Poincar\'e inequalities with respect to unions of
cones}\label{sec:union} Our goal here is to prove Lemma~\ref{lem:cone union}. Before doing so, we record the following simple observation.

\begin{lemma}\label{lem:sum of squares}
Fix $\lambda,\kappa\in (0,\infty)$ and $n\in \N$. Let $(A,d_A)$, $(B,d_B)$ and $(X,d_X)$ be metric spaces.   Suppose that there exist $\lambda$-Lipschitz mappings $\mathfrak{a}:X\to A$ and $\mathfrak{b}:X\to B$ such that for every $x,y\in X$,
\begin{equation}\label{eq:ab sum of squares}
 \sqrt{d_A(\mathfrak{a}(x),\mathfrak{a}(y))^2+d_B(\mathfrak{b}(x),\mathfrak{b}(y))^2}\ge \frac{d_X(x,y)}{\kappa}.
\end{equation}
Then every $n$ by $n$ symmetric stochastic matrix $M=(m_{ij})$ satisfies
\begin{equation}\label{eq:AB gamma+}
\gamma_+\!\left(M,d_X^2\right)\le (\kappa\lambda)^2\left(\gamma_+\!\left(M,d_A^2\right)+\gamma_+\!\left(M,d_B^2\right)\right).
\end{equation}
\end{lemma}

\begin{proof}
For every $x_1,\ldots,x_n,y_1,\ldots,y_n\in X$ we have
\begin{multline}\label{eq:sum of squares}
\frac{1}{n^2}\sum_{i=1}^n\sum_{j=1}^n d_X(x_i,y_j)^2\\\stackrel{\eqref{eq:ab sum of squares}}{\le} \frac{\kappa^2}{n^2}\sum_{i=1}^n\sum_{j=1}^n \left(d_A(\a(x_i),\a(y_j))^2+d_B(\b(x_i),\b(y_j))^2\right).
\end{multline}
By the definition of $\gamma_+\!\left(M,d_A^2\right)$ and
$\gamma_+\!\left(M,d_B^2\right)$, combined with the fact that $\a$
and $\b$ are $\lambda$-Lipschitz,
\begin{align*}
&\frac{1}{n^2}\sum_{i=1}^n\sum_{j=1}^n \left(d_A(\a(x_i),\a(y_j))^2+d_B(\b(x_i),\b(y_j))^2\right)\\
&\le \frac{\gamma_+\!\left(M,d_A^2\right)}{n}\sum_{i=1}^n\sum_{j=1}^n m_{ij} d_A(\a(x_i),\a(y_j))^2\\&\qquad+\frac{\gamma_+\!\left(M,d_B^2\right)}{n}\sum_{i=1}^n\sum_{j=1}^n m_{ij} d_B(\b(x_i),\b(y_j))^2\\
&\le \frac{\lambda^2\left(\gamma_+\!\left(M,d_A^2\right)+\gamma_+\!\left(M,d_B^2\right)\right)}{n}\sum_{i=1}^n\sum_{j=1}^n m_{ij} d_X(x_i,y_j)^2.
\end{align*}
In combination with~\eqref{eq:sum of squares}, this completes the proof of Lemma~\ref{lem:sum of squares}.
\end{proof}

\begin{remark}\label{rem:remove+1}
The identical argument shows that~\eqref{eq:AB gamma+} also holds
true with the quantities $\gamma_+\!\left(M,d_X^2\right),
\gamma_+\!\left(M,d_A^2\right), \gamma_+\!\left(M,d_B^2\right)$
replaced by the quantities $\gamma\left(M,d_X^2\right),
\gamma\left(M,d_A^2\right), \gamma\left(M,d_B^2\right)$,
respectively.
\end{remark}

\begin{proof}[Proof of Lemma~\ref{lem:cone union}] Since the Euclidean cone over $(X,d_X)$ is isometric to the Euclidean cone over the metric space $(X,\min\{\pi,d_X\})$, it suffices to prove~\eqref{eq;AB cone} under the additional assumption
\begin{equation}\label{eq:pi diameter}
\diam(X)\le \pi.
 \end{equation}

 Let $o$ denote the cusp of $\cone(X)$, i.e, the equivalence class obtained by identifying all of the
 points in $\{0\}\times X$.  Thus $d_{\cone(X)}\big((s,x),o\big)=s$ for every
 $(s,x)\in (0,\infty)\times X$.
 Define $\mathfrak{a}:\cone(X)\to
 \cone(A)$ and $\mathfrak{b}:\cone(X)\to \cone(B)$ by setting for every $(s,x)\in (0,\infty)\times X$,
$$
\mathfrak{a}(s,x)\eqdef \big(d_X(x,B\setminus A)s,x\big)\quad \mathrm{and}\quad
\mathfrak{b}(s,x)\eqdef \big(d_X(x,A\setminus B)s,x\big).
$$

We first show that $\mathfrak{a}$ and $\mathfrak{b}$ are $\sqrt{3}\pi$-Lipschitz. By symmetry it suffices check this for $\a$, i.e., to show that for every $s,t\in (0,\infty)$ and $x,y\in X$,
\begin{equation}\label{eq:pi lip condition}
d_{\cone(A)}\big(\a(s,x),\a(t,y)\big)\le \sqrt{3}\pi d_{\cone(X)}\big((s,x),(t,y)\big).
\end{equation}
\eqref{eq:pi lip condition} is trivial if $d_X(x,B\setminus
A)=d_X(y,B\setminus A)=0$. If $d_X(x,B\setminus A)>0$ and
$d_X(y,B\setminus A)=0$ then

\begin{align}
\nonumber&d_{\cone(A)}\big(\a(s,x),\a(t,y)\big)\\\nonumber&=d_{\cone(A)}\big(\big(d_X(x,B\setminus A)s,x\big),o\big)\\
\nonumber&=d_X(x,B\setminus A)s\\\nonumber &\le d_X(x,y)s\\&\le \frac{d_X(x,y)}{\sin\left(\min\left\{\frac{\pi}{2},d_X(x,y)\right\}\right)}d_{\cone(X)}\big((s,x),(t,y)\big)\label{eq:use sinus}\\
&\le \pi d_{\cone(X)}\big((s,x),(t,y)\big)\label{eq:use diameter},
\end{align}
where~\eqref{eq:use sinus} uses Lemma~\ref{lem:sinus} and~\eqref{eq:use diameter} uses~\eqref{eq:pi diameter} and the fact that the function $u\mapsto u/\sin u$ is increasing on $(0,\pi/2]$. The validity of~\eqref{eq:pi lip condition} when $d_X(x,B\setminus A), d_X(y,B\setminus A)>0$  follows from Lemma~\ref{lem:rescaling} with $X$ replaced by $X\setminus (\overline{B\setminus A})$ and $f(z)=d_X(z,B\setminus A)$ (using~\eqref{eq:pi diameter} and consequently $\|f\|_\infty\le \pi$, combined with $\|f\|_{\Lip}\le1$).

Having proved~\eqref{eq:pi lip condition}, it follows from Lemma~\ref{lem:sum of squares} that Lemma~\ref{lem:cone union} will be proven once we show that for every $s,t\in (0,\infty)$ and $x,y\in X$,
\begin{multline}\label{eq:goal sum of squares beta}
d_{\cone(A)}\big(\a(s,x),\a(t,y)\big)^2+d_{\cone(B)}\big(\b(s,x),\b(t,y)\big)^2\\\ge \frac{\beta^4}{72\pi^2}d_{\cone(X)}\big((s,x),(t,y)\big)^2.
\end{multline}
We will actually show that
\begin{multline}\label{eq:implication}
s\ge t\quad\mathrm{and}\quad d_X\left(x,B\setminus A\right)\ge \frac{\beta}{2}\\\implies d_{\cone(A)}\big(\a(s,x),\a(t,y)\big)^2\ge \frac{\beta^4}{72\pi^2}d_{\cone(X)}\big((s,x),(t,y)\big)^2.
\end{multline}
The validity of the implication~\eqref{eq:implication} yields~\eqref{eq:goal sum of squares beta} since by replacing $(s,x)$ by $(t,y)$ if necessary we may assume without loss of generality that $s\ge t$, and~\eqref{eq:implication} implies that we are done if $d_X\left(x,B\setminus A\right)\ge \beta/2$. If $d_X\left(x,B\setminus A\right)< \beta/2$ then note that by the triangle inequality,
$$
d_X(x,B\setminus A)+d_X(x,A\setminus B)\ge d_X(B\setminus A,A\setminus B)\stackrel{\eqref{eq:AB distant}}{\ge} \beta.
$$
Consequently, since we are assuming that $d_X\left(x,B\setminus A\right)< \beta/2$, we have $d_X(x,A\setminus B)>\beta/2$, so an application of~\eqref{eq:implication} with $A$ replaced by $B$ and $\a$ replaced by $\b$ implies that $$d_{\cone(B)}\big(\b(s,x),\b(t,y)\big)^2\ge \frac{\beta^4}{18\pi^2}d_{\cone(X)}\big((s,x),(t,y)\big)^2,$$ yielding~\eqref{eq:goal sum of squares beta} in the remaining case.

It therefore remains to prove the implication~\eqref{eq:implication}. So, suppose from now on that
\begin{equation}\label{eq:two assumptions}
s\ge t\qquad \mathrm{and}\qquad d_X\left(x,B\setminus A\right)\ge \frac{\beta}{2}.
\end{equation}
It will be convenient to use below the following notation.
\begin{equation}\label{eq:three notation}
d_x\eqdef d_X(x,B\setminus A),\qquad d_y\eqdef d_X(y,B\setminus A),\qquad \theta\eqdef d_X(x,y).
\end{equation}
We proceed by considering the cases $d_x\ge d_y$ and $d_x< d_y$ separately.

\medskip

\noindent{\bf Case 1:} $d_x\ge d_y$. In this case, since we are assuming~\eqref{eq:two assumptions} we have $sd_x\ge td_y$. Consequently,
\begin{align}
&d_{\cone(A)}\big(\a(s,x),\a(t,y)\big)\nonumber\\
&\ge \frac13\max\left\{d_xs-d_yt,d_xs\sqrt{2(1-\cos\theta)}\right\}\label{eq:use comparison lemma 1}\\
&\ge \frac{d_x}{3} \max\left\{s-t,s\sqrt{2(1-\cos\theta)}\right\}\nonumber\\
&\ge \frac{\beta}{6\sqrt{2}}d_{\cone(X)}\big((s,x),(t,y)\big)\label{eq:use comparison lemma 2},
\end{align}
where~\eqref{eq:use comparison lemma 1} uses
Lemma~\ref{lem:comparison to max} and~\eqref{eq:use comparison lemma
2} uses~\eqref{eq:two assumptions} and Lemma~\ref{lem:comparison to
max}. Since $\beta\in [0,\pi]$, \eqref{eq:use comparison lemma 2} implies the conclusion of~\eqref{eq:implication}.

\medskip

\noindent{\bf Case 2:} $d_x< d_y$. Fix $\xi\in (0,1)$ that will be
determined later. By the convexity of the function $u\mapsto u^2$ we
have
\begin{equation*}\label{eq:squares}
\forall\,u,v\in \R,\qquad (u-v)^2\ge \xi
u^2-\frac{\xi}{1-\xi} v^2.
\end{equation*}
Consequently,
\begin{align}
\nonumber\left(d_xs-d_yt\right)^2&= \left(d_x(s-t)+(d_x-d_y)t\right)^2\\
\nonumber&\ge \xi d_x^2(s-t)^2-\frac{\xi}{1-\xi} (d_y-d_x)^2t^2\\&\ge \frac{\xi\beta^2}{4}(s-t)^2-\frac{\xi}{1-\xi}\theta^2t^2\label{eq;triangle dx dy}
\\&\ge \frac{\xi\beta^2}{4}(s-t)^2-\frac{\pi^2\xi}{2(1-\xi)}st(1-\cos\theta),\label{eq:use s>t}
\end{align}
where~\eqref{eq;triangle dx dy} uses the assumption $d_x\ge \beta/2$
and, by the triangle inequality, that $|d_y-d_x|\le d_X(x,y)=
\theta$. In~\eqref{eq:use s>t} we used~\eqref{eq:cosine
inequalities} and the assumption $s\ge t$. Now,

\begin{align}
\nonumber&d_{\cone(A)}\big(\a(s,x),\a(t,y)\big)^2\\\nonumber&=(d_xs-d_yt)^2+2d_xd_yst(1-\cos\theta)\\
& \ge (d_xs-d_yt)^2+\frac{\beta^2}{2}st(1-\cos\theta)\label{eq:use case 2}\\
&\ge \frac{\xi\beta^2}{4}(s-t)^2+\frac14\left(\beta^2-\frac{\pi^2\xi}{1-\xi}\right)\cdot 2st(1-\cos\theta)\label{eq:use xi}\\
&\ge \frac14\min\left\{\xi\beta^2,\beta^2-\frac{\pi^2\xi}{1-\xi}\right\}\cdot d_{\cone(X)}\big((s,x),(t,y)\big)^2\label{eq:min xi},
\end{align}
where in~\eqref{eq:use case 2} we used the fact that in Case 2 we
are assuming that $d_y>d_x\ge \beta/2$ and in~\eqref{eq:use xi} we
used~\eqref{eq:use s>t}. Choosing $\xi=\beta^2/(3\pi^2)$
in~\eqref{eq:min xi} implies the conclusion
of~\eqref{eq:implication} in Case 2 as well.
\end{proof}

\section{Embedding $(1+\delta)$-sparse graphs in $L_1$}\label{sec:embed
sparse}

In this section all graphs are simple. Following~\cite{ABLT}, given
$\delta\in (0,1)$ we say that a
 graph $G$ is $(1+\d)$-sparse if
 \begin{equation}\label{eq:def delta sparse}
\forall\, S\subset V_G,\qquad \left|E_G(S)\right|\le (1+\d)|S|,
 \end{equation}
where we recall that $E_G(S)$ denotes the set of those edges in
$E_G$ both of whose endpoints lie in $S$, i.e., $E_G(S)=\{\{u,v\}\in
E_G:\ u,v\in S\}$.

For $t>0$ we denote the set of all cycles of $G$ of size less than
$t$  by
\begin{equation}\label{eq:def C_t}
\mathfrak{C}_t(G)\eqdef \left\{C\subset V:\ C\ \mathrm{is\ a\ cycle\ of\ }G\ \mathrm{and\ }|C|< t\right\}.
\end{equation}

\begin{claim}\label{claim:induced}
Fix $\d\in (0,1)$ and a graph $G$ that is $(1+\d)$-sparse. Then
every $C\in \mathfrak{C}_{1/\d}(G)$ is an induced cycle, i.e.,
$|E_G(C)|=|C|$.
\end{claim}
\begin{proof}
If $C$ is a cycle of $G$ and $|E_G(C)|> |C|$ then it follows from~\eqref{eq:def delta sparse} that $|C|+1\le|E_G(C)|\le (1+\d)|C|$, implying that $|C|\ge 1/\d$.
\end{proof}

The starting point of our investigations in the present section is
the following result from~\cite{ALNRRV}.

\begin{lemma}[Lemma 3.11 of~\cite{ALNRRV}]\label{lem:with s-1} Fix $\d\in (0,1)$ and a connected graph
$G$ that satisfies
\begin{equation}\label{eq:|S|-1}
\forall\, S\subset V_G,\qquad |E_G(S)|\le (1+\d)(|S|-1).
\end{equation}
Then
\begin{equation}\label{eq:ALNRVV c1 bound}
c_1(V_G,d_G)\lesssim 1+\d\diam(G).
\end{equation}
\end{lemma}

While the condition~\eqref{eq:|S|-1} is very close to the
$(1+\d)$-sparseness condition~\eqref{eq:def delta sparse}, it in
fact implies that $G$ also has high girth. The relation
between~\eqref{eq:|S|-1} and~\eqref{eq:def delta sparse} is
clarified in the following lemma.

\begin{lemma}\label{lem:link between sparse and s-1} Fix $\d\in
(0,1)$ and an integer $g\ge 3$. Suppose that a graph $G$
satisfies~\eqref{eq:|S|-1}. Then $G$ is $(1+\d)$-sparse and the
girth of $G$ is at least $1+1/\d$. Conversely, suppose that $G$ is
$(1+\d)$-sparse and the girth of $G$ is at least $g$. Then
$$
\forall\, S\subset V_G,\qquad |E_G(S)|\le (1+\d)\left(1+\frac{1}{g-1}\right)(|S|-1).
$$
\end{lemma}

\begin{proof}
If $G$ contains a cycle of size $k$ then $k\le (1+\d)(k-1)$ due
to~\eqref{eq:|S|-1}. Thus $k\ge 1+1/\d$, proving the first assertion
of Lemma~\ref{lem:link between sparse and s-1}. Conversely, suppose
that $G$ is $(1+\d)$-sparse and the girth of $G$ is at least $g$.
Take $S\subseteq V_G$. If $|S|<g$ then the graph $(S,E_G(S))$ cannot
contain a cycle, and it is therefore a forest. Consequently
$|E_G(S)|\le |S|-1$. If $|S|\ge g$ then it follows
from~\eqref{eq:def delta sparse} that
\begin{equation*}
|E_G(S)|\le (1+\d)|S|\le(1+\d)\left(|S|-1+\frac{|S|-1}{g-1}\right).\qedhere
\end{equation*}
\end{proof}

Lemma~\ref{lem:with s-1} therefore implicitly assumes that the graph
in question has high girth, an assumption that will not be available
in our context. In fact, in our setting $\delta$ will tend to $0$ as
$|V_G|\to \infty$, so the implicit high girth assumption of
Lemma~\ref{lem:with s-1} is quite restrictive. We observe that this high girth
assumption is essential for the proof of Lemma~\ref{lem:with s-1}
in~\cite{ALNRRV}. Indeed, denoting the set of all spanning trees of
a connected graph $G$ by $\mathfrak{T}(G)$, the main step of the
proof of Lemma~\ref{lem:with s-1} in~\cite{ALNRRV} is the following
statement.

\begin{claim}[Claim 3.14 of~\cite{ALNRRV}]\label{claim:spanning}
Fix $\d\in (0,1)$ and suppose that a connected graph $G$ satisfies
the assumptions of Lemma~\ref{lem:with s-1}. Then there exists a
probability measure $\mu$ on $\T(G)$ such that
\begin{equation}\label{eq:spanning tree measure ALNRVV}
\forall\, e\in E_G,\qquad \mu\big(\{T=(V_G,E_T)\in \T(G):\ e\in E_T\}\big)\ge \frac{1}{1+\d}.
\end{equation}
\end{claim}
The existence of a probability measure $\mu$ on $\T(G)$ that
satisfies~\eqref{eq:spanning tree measure ALNRVV} implies that the
girth of $G$ is at least $1+1/\d$. Indeed, if $C\subseteq V_G$ is a
cycle of $G$ then since every $T\in \T(G)$ satisfies $|E_T\cap
E_G(C)|\le |C|-1$ it would follow from~\eqref{eq:spanning tree
measure ALNRVV} that
\begin{multline*}
|C|-1\ge \int_{\T(G)}|E_T\cap E_G(C)|d\mu(T)\\=\sum_{e\in
E_G(C)}\mu\big(\{T\in \T(G):\ e\in E_T\}\big)\ge
\frac{|E_G(C)|}{1+\d}\ge \frac{|C|}{1+\d},
\end{multline*}
implying that $|C|\ge 1+1/\d$.

 In what follows, given a connected graph $G$, a spanning tree $T$ of its associated
 $1$-dimensional simplicial complex $\Sigma(G)$ is defined as follows. Given a
 spanning tree $T_0=(V_G,E_{T_0})\in \T(G)$,  include in $T$ all the edges of $E_{T_0}$ as unit intervals.
 If $e=\{x,y\}\in E_G$ is an edge of $G$ that is not an edge of $T_0$ then partition the unit
 interval in $\Sigma(G)$ that corresponds to $e$ into two connected subsets $I_x,I_y$,
 where $x\in I_x$ and $y\in I_y$ (thus one of $I_x,I_y$ is a closed interval and the
 other is a half open interval).  $T$ is now obtained from $T_0$ by gluing $I_x$ to $T_0$
 at the vertex $x$, and gluing $I_y$ to $T_0$ at the vertex $y$. Notice that the completion
 of $T$ (i.e., taking the closures of the half open intervals that were glued to $T_0$) is
 the $1$-dimensional simplicial complex of a graph theoretical tree. The set of all spanning
 trees of $\Sigma(G)$ is denoted below by $\T(\Sigma(G))$.

Here we prove the following theorem, which assumes only that the
graph in question is connected and $(1+\d)$-sparse.

\begin{theorem}\label{thm:our spanning no girth}
Fix $\d\in (0,1)$ and a $(1+\d)$-sparse graph $G$. There exists a
probability measure $\mu$ on $\T(\Sigma(G))$  with respect to which
every $x,y\in \Sigma(G)$ satisfy
\begin{multline}\label{eq:no girth spanning}
\int_{\T(\Sigma(G))} \min\left\{d_{T}(x,y),\diam(G)\right\}d\mu(T)\\\lesssim (1+\d\diam(G))d_{\Sigma(G)}(x,y).
\end{multline}
\end{theorem}

Before passing to the proof of Theorem~\ref{thm:our spanning no
girth} we state the following corollary that explains its link to
Lemma~\ref{lem:with s-1}.

\begin{corollary}\label{cor:no girth into L_1}
Under the assumptions of Theorem~\ref{thm:our spanning no girth} we
have
$$
c_1(\Sigma(G))\lesssim 1+\d\diam(G).
$$
\end{corollary}

\begin{proof}
By Lemma~\ref{lem:L1 truncation} there exists a mapping
$\Phi:L_1\to L_1$ such that
$$
\forall x,y\in L_1,\qquad \|\Phi(x)-\Phi(y)\|_1\asymp \min\left\{\|x-y\|_{1},\diam(G)\right\}.
$$
For every $T\in \T(\Sigma(G))$ fix an isometric embedding
$F_T:(T,d_T)\to L_1$. Let $L_1(\mu,L_1)$ denote the space of all
mappings $f:\T(\Sigma(G))\to L_1$, equipped with the norm
$$
\|f\|_{L_1(\mu,L_1)}\eqdef \int_{\T(\Sigma(G))} \|f(T)\|_1d\mu(T).
$$
Then $L_1(\mu,L_1)$ is isometric to $L_1$. Define $h:\Sigma(G)\to L_1(\mu,L_1)$ by
$$
\forall\, T\in \T(\Sigma(G)),\ \forall\, x\in \Sigma(G),\qquad h(x)(T)=\Phi(F_T(x)).
$$
Then, for every $x,y\in \Sigma(G)$,
\begin{equation}\label{eq:h integral minimum}
\|h(x)-h(y)\|_{L_1(\mu,L_1)}\asymp \int_{\T(\Sigma(G))} \min\left\{d_{T}(x,y),\diam(G)\right\}d\mu(T).
\end{equation}
Observe that $\min\left\{d_{T}(x,y),\diam(G)\right\}\gtrsim
d_{\Sigma(G)}(x,y)$ for every spanning tree $T\in \T(\Sigma(G))$ and
$x,y\in \Sigma(G)$, because $d_{T}(x,y)\ge d_{\Sigma(G)}(x,y)$ and
$d_{\Sigma(G)}(x,y)\le \diam(G)+1$ (recall~\eqref{eq:diam
complex}). Hence, it follows from~\eqref{eq:no girth spanning}
and~\eqref{eq:h integral minimum} that
\begin{equation*}
d_{\Sigma(G)}(x,y)\lesssim \|h(x)-h(y)\|_{L_1(\mu,L_1)} \lesssim (1+\d\diam(G))d_{\Sigma(G)}(x,y).\qedhere
\end{equation*}
\end{proof}

\begin{remark} Lemma~\ref{lem:with s-1} is proved in~\cite{ALNRRV} by first
establishing the same statement as Theorem~\ref{thm:our spanning no
girth} under the stronger assumption~\eqref{eq:|S|-1}, with the
conclusion~\eqref{eq:no girth spanning} holding for every $x,y\in
V_G$. It is then shown in~\cite[Claim~3.13]{ALNRRV} that for every
tree
 $T$ and every threshold $a\in (0,\infty)$ the metric $\min\{d_T,a\}$ embeds with  distortion $O(1)$
 into a convex combination of dominating tree metrics. Thus for the purpose of proving
 Corollary~\ref{cor:no girth into L_1} one can use Claim~3.13 of~\cite{ALNRRV} rather
 than the stronger statement of Lemma~\ref{lem:L1 truncation}, which deals with truncation
 of all of $L_1$ and not just trees.
\end{remark}

\begin{remark}
The fact that
 Theorem~\ref{thm:our spanning no girth} yields an embedding of $\Sigma(G)$ rather than $(V_G,d_G)$
 is needed in the ensuing
arguments, but it does
  not add significant difficulties to the proof of Theorem~\ref{thm:our spanning no girth}.
\end{remark}

The new contribution of Theorem~\ref{thm:our spanning no girth} is
that its conclusion holds under the assumption that $G$ is
$(1+\d)$-sparse rather than under the assumption~\eqref{eq:|S|-1},
i.e., without requiring that the girth of $G$ is large. This is
achieved via the following strategy: in~\cite{ALNRRV}, given an edge
$\{x,y\}\in E_G$ and a spanning tree $T\in \T(G)$ that does not
contain $\{x,y\}$ as an edge, the quantity
$\min\left\{d_{T}(x,y),\diam(G)\right\}$ is bounded from above by
$\diam(G)$. But by Claim~\ref{claim:induced}, if $\{x,y\}$ is an
edge of a cycle $C\in \mathfrak{C}_{1/\d}(G)$ and $|E_T\cap
E_G(C)|\ge |C|-1$ then  $d_{T}(x,y)\le |C|-1$. Our strategy is
therefore to use this better bound on $d_{T}(x,y)$, and to show that
it is possible to only  deal with probability measures $\mu$ on
$\T(G)$ that are supported on those spanning trees $T\in \T(G)$
satisfying $|E_T\cap E_G(C)|\ge |C|-1$ for {\em every} small cycle
$C\in \mathfrak{C}_{1/(3\d)}(G)$.

We next prove some lemmas as steps towards the proof of
Theorem~\ref{thm:our spanning no girth}, starting with the following
combinatorial fact whose obvious and short proof is included for
completeness.

\begin{lemma}\label{lem:union of cycles}
Let $G$ be a graph and $C_1,C_2\subseteq V$ be distinct cycles of
$G$ such that $C_1\cap C_2\neq \emptyset$. Then $|E_G(C_1\cup
C_2)|\ge |C_1\cup C_2|+1$.
\end{lemma}

\begin{proof} By removing redundant edges and vertices we may assume without loss of
generality that $V_G=C_1\cup C_2$ and if $C_1=\{x_1,\ldots,x_m\}$
and $C_2=\{y_1,\ldots,y_n\}$ then $E_G=E_1\cup E_2$, where
$$
E_1\eqdef\big\{\{x_1,x_2\},\{x_2,x_3\},\ldots,\{x_{m-1},x_m\},\{x_m,x_1\}\big\},
$$
and
$$
E_2\eqdef \big\{\{y_1,y_2\},\{y_2,y_3\},\ldots,\{y_{n-1},y_n\},\{y_n,y_1\}\big\}.
$$
If $E_1\cap E_2=\emptyset $ then $$|E_G(C_1\cup
C_2)|=|E|=|C_1|+|C_2|\ge |C_1\cup C_2|+1,$$ where the final
inequality uses $C_1\cap C_2\neq \emptyset$. So, suppose that there
exist $i,j\in \{1,\ldots, m\}$ and $s,t\in \{1,\ldots,n\}$ with
$|i-j|\in \{1,m-1\}$, $|s-t|\in \{1,n-1\}$ and $x_i=y_s$, $x_j=y_t$.
If $|C_1|,|C_2|>3$ then we may contract the edge $\{x_i,x_j\}$ while
identifying its endpoints, thus obtaining a graph $G'$ which is the
union of two strictly smaller cycles, with $|V_{G'}|=|V|-1$ and
$|E_{G'}|=|E|-1$. By continuing in this manner we see that it
suffices to prove the desired result under the additional
assumptions $|C_1|=3$ and $E_1\cap E_2\neq \emptyset$. Thus either
$|C_1\cap C_2|=2$ or $C_1\subset C_2$. In the former case we have
$$|E_G(C_1\cup C_2)|=|C_2|+2=|C_1\cup C_2|+1,$$ and in the latter
case, since $C_1\neq C_2$, we have
\begin{equation*}
|E_G(C_1\cup C_2)|\ge |C_2|+1= |C_1\cup C_2|+1.\qedhere
\end{equation*}
\end{proof}

\begin{lemma}\label{lem:cycles are far}
Fix $\d\in (0,1)$ and a  connected $(1+\d)$-sparse graph $G$.
Suppose that $C_1,C_2\subseteq V_G$ are distinct cycles of $G$. Then
\begin{equation}\label{eq:cycles are distant}
d_G(C_1,C_2)\ge 1+\frac{1}{\d}-|C_1|-|C_2|.
\end{equation}
\end{lemma}

\begin{proof} It suffices to prove~\eqref{eq:cycles are distant} under the assumption
\begin{equation}\label{eq:C1+C2}
|C_1|+|C_2|< 1+\frac{1}{\d}.
\end{equation}

We first note that~\eqref{eq:C1+C2} implies that $C_1\cap
C_2=\emptyset$. Indeed, otherwise, using Lemma~\ref{lem:union of
cycles} and the  fact that $G$ is $(1+\d)$-sparse, we have $$
|C_1\cup C_2|+1\le |E_G(C_1\cup C_2)|\le (1+\d)|C_1\cup C_2|.$$
Consequently $|C_1\cup C_2|\ge 1/\d$, which
contradicts~\eqref{eq:C1+C2} since $C_1\cap C_2\neq \emptyset$.

Having proved that $C_1\cap C_2=\emptyset$, we have $t\eqdef
d_G(C_1,C_2)>0$. Take $x\in C_1$ and $y\in C_2$ with $d_G(x,y)=t$
and choose $w_0,\ldots,w_{t}\in V_G$ such that $w_0=x$, $w_{t}=y$
and $\{w_{i-1},w_i\}\in E_G$ for every $i\in \{1,\ldots,t\}$. By the
choice of $x$ and $y$ as the vertices at which $d_G(C_1,C_2)$ is
attained, necessarily $w_1,\ldots,w_{t-1}\in V_G\setminus(C_1\cup
C_2)$. Hence,
$$|C_1\cup C_2\cup\{w_1,\ldots,w_{t-1}\}|=|C_1|+|C_2|+t-1,
$$
 and
$$
|E_G\left(C_1\cup C_2\cup\{w_1,\ldots,w_{t-1}\}\right)|\ge |C_1|+|C_2|+t.
$$
Since $G$ is $(1+\d)$-sparse, it follows that
$$
|C_1|+|C_2|+t\le (1+\d)(|C_1|+|C_2|+t-1),
$$
which simplifies to give $d_G(C_1,C_2)=t\ge
1+\frac{1}{\d}-|C_1|-|C_2|$.
\end{proof}

Fixing $\d\in (0,1)$ and a connected $(1+\d)$-sparse graph $G$,
write
\begin{equation}\label{eq:def t}
t\eqdef \frac{1}{3\d},
\end{equation}
and define $\Gamma\subseteq V_G$ by
\begin{equation}\label{eq:defGamma}
\Gamma\eqdef \bigcup_{C\in \mathfrak{C}_{t}(G)}C.
\end{equation}

We shall work below with a ``quotient graph"
$$
G/\C_t(G)=\big(V_{G/\C_t(G)},E_{G/\C_t(G)}\big),
$$
which is defined as follows.
\begin{equation}\label{eq:defW}
V_{G/\C_t(G)}\eqdef (V_G\setminus \Gamma)\cup \mathfrak{C}_{t}(G),
\end{equation}
i.e., the vertex set of $G/\C_t(G)$ consists of those vertices of
$G$ that do not belong to any cycle of $G$ of length less than $t$,
and we append to these vertices an additional vertex for every cycle
$C\in \mathfrak{C}_{t}(G)$. Thus
\begin{equation}\label{eq:number of vertices quotient}
\left|V_{G/\C_t(G)}\right|=|V_G|-|\Gamma|+|\C_t(G)|.
\end{equation}
The edges of $G/\C_t(G)$ are defined as follows. For every $x,y\in
V_G\setminus \Gamma$,
$$
\{x,y\}\in E_{G/\C_t(G)}\iff \{x,y\}\in E_G,
$$
and for every  $(x,C)\in (V_G\setminus \Gamma)\times
\mathfrak{C}_{t}(G)$,
\begin{equation*}
\{x,C\}\in E_{G/\C_t(G)}\iff d_G(x,C)=1.
\end{equation*}

Under this definition $G/\C_t(G)$ is simple and connected. Moreover,
by Lemma~\ref{lem:cycles are far} every distinct $C_1,C_2\in
\C_t(G)$ are disjoint, and by Claim~\ref{claim:induced} we have
$E_G(C)=|C|$ for every $C\in \C_t(G)$. Consequently,
\begin{equation}\label{eq:number of edges in quotient}
\left|E_{G/\C_t(G)}\right|=|E_G|-\sum_{C\in \C_t(G)}|C|=|E_G|-|\Gamma|.
\end{equation}

\begin{lemma}\label{lem:cycles are far in quotient}
For every distinct $C_1,C_2\in \C_t(G)$ we have
$$
d_{G/\C_t(G)}(C_1,C_2)>t+1.
$$
\end{lemma}

\begin{proof} Suppose that $d_{G/\C_t(G)}(C_1,C_2)$ is minimal among all distinct cycles
$C_1,C_2\in \C_t(G)$. Then by the definition of $G/\C_t(G)$, the
non-endpoint vertices on the shortest path joining $C_1$ and $C_2$
in $G/\C_t(G)$ consist of vertices in $V_G\setminus \Gamma$.
Consequently,
$$
d_{G/\C_t(G)}(C_1,C_2) \ge d_G(C_1,C_2)\ge 1+\frac{1}{\d}-|C_1|-|C_2|>1+\frac{1}{\d}-2t=t+1,
$$
where we used Lemma~\ref{lem:cycles are far}, the fact that $|C_1|,|C_2|<t$, and~\eqref{eq:def t}.
\end{proof}

\begin{lemma}\label{lem:edges in entire quotient}
If $0<\d<\frac13$ then
\begin{equation}\label{eq:few edges in quotient}
\left|E_{G/\C_t(G)}\right|\le \frac{1+\d}{1-3\d}\left(\left|V_{G/\C_t(G)}\right|-1\right).
\end{equation}
\end{lemma}

\begin{proof}
If $\C_t(G)=\emptyset$ then by definition $G=G/\C_t(G)$ and the
girth of $G$ is at least $t$.  Recalling the choice of $t$
in~\eqref{eq:def t}, we see that~\eqref{eq:few edges in quotient}
holds true due to the fact that $G$ is $(1+\d)$-sparse and
Lemma~\ref{lem:link between sparse and s-1}. We may therefore assume
from now on that $\C_t(G)\neq \emptyset$. We may also assume from
now on that $G/\C_t(G)$ contains a cycle, since otherwise it is a
tree and therefore $|E_{G/\C_t(G)}|\le |V_{G/\C_t(G)}|-1$.

Suppose first that $|\C_t(G)|=1$ and write $\C_t(G)=\{C\}$. Thus,
recalling~\eqref{eq:defGamma}, we have $|\Gamma|=|C|<t$. Since
$G/\C_t(G)$ contains a cycle, it follows that $|V_G\setminus
\Gamma|\ge t/2$. Indeed, if there is a cycle in $G/\C_t(G)$ that
does not contain the vertex $C\in V_{G/\C_t(G)}$ then since no cycle
of $G$ other than $C$ has length less than $t$, this cycle contains
at least $t$ vertices of $V_G\setminus \Gamma$. If on the other hand
there exist distinct $x_1,\ldots, x_k\in V_G\setminus \Gamma$ such
that $\{C,x_1\},\{x_1,x_2\},\ldots,\{x_{k-1},x_k\},\{x_k,C\}\in
E_{G/\C_t(G)}$ then $d_G(x_1,C)=d_G(x_k,C)=1$. By adding to
$\{x_1,\ldots, x_k\}$ the vertices on the shortest path in $C$
joining the nearest neighbors of $x_1$ and $x_k$ in $C$, we obtain a
cycle of length less than $k+t/2$ in $G$ that differs from $C$.
Since we are assuming that no cycle of $G$ other than $C$ has length
less than $t$, it follows that $k+t/2\ge t$, implying that
$|V_G\setminus \Gamma|\ge k\ge t/2$, as required. We have thus shown
that
\begin{equation}\label{eq:Gamma small 1}
|\Gamma|<t\le 2|V_G\setminus \Gamma|,
\end{equation}
 and consequently, since $G$ is $(1+\d)$-sparse,
\begin{multline}\label{eq:case of one cycle}
\left|E_{G/\C_t(G)}\right|\stackrel{\eqref{eq:number of edges in
quotient}}{=}|E_G|-|\Gamma|\le (1+\d)|V_G|-|\Gamma|
=(1+\d)|V_G\setminus \Gamma|+\d|\Gamma|\\\stackrel{\eqref{eq:Gamma
small 1}}{\le} (1+3\d)|V_G\setminus \Gamma|=
(1+3\d)\left(\left|V_{G/\C_t(G)}\right|-1\right),
\end{multline}
where the last step of~\eqref{eq:case of one cycle} uses~\eqref{eq:number of vertices quotient} and the fact that we are treating the case  $|\C_t(G)|=1$. Since $1+3\d\le (1+\d)/(1-3\d)$, the desired bound~\eqref{eq:few edges in quotient} follows from~\eqref{eq:case of one cycle}.

It remains to prove~\eqref{eq:few edges in quotient} under the assumption $|\C_t(G)|\ge 2$. In this case, for every $C\in \C_t(G)$ fix an arbitrary $C'\in \C_t(G)$ with $C\neq C'$. By  Lemma~\ref{lem:cycles are far in quotient} if $\{C=v_0,v_1,\ldots,v_k=C'\}$ is the shortest path in $G/\C_t(G)$ joining $C$ and $C'$ then $k>t+1=1/(3\d)+1\ge 2$ and $v_1,\ldots,v_{\lceil t\rceil}\in V\setminus \Gamma$. If we define $A(C)\subseteq V\setminus \Gamma$ by $$
A(C)\eqdef  \left\{v_1,\ldots,v_{\lfloor(t+1)/2\rfloor}\right\},$$
then $A(C_1)\cap A(C_2)=\emptyset$ for every distinct $C_1,C_2\in \C_t(G)$, since otherwise $d_{G/\C_t(G)}(C_1,C_2)\le 2\lfloor(t+1)/2\rfloor\le t+1$, contradicting Lemma~\ref{lem:cycles are far in quotient}. This shows that
\begin{multline}\label{eq:(t-1)/2}
|V_G\setminus \Gamma|\ge \left|\bigcup_{C\in
\C_t(G)}A(C)\right|=\sum_{C\in \C_t(G)}|A(C)|\\\ge |\C_t(G)|\cdot
\left\lfloor\frac{t+1}{2}\right\rfloor\ge \frac{t-1}{2}|\C_t(G)|.
\end{multline}
Hence,
\begin{multline}\label{eq:factor 3}
|\Gamma|=\sum_{C\in
\C_t(G)}|C|<t|\C_t(G)|\\\stackrel{\eqref{eq:(t-1)/2}}{\le}\frac{2t}{t-1}|V_G\setminus
\Gamma| \stackrel{\eqref{eq:def t}}{=}\frac{2}{1-3\d}|V_G\setminus
\Gamma|.
\end{multline}
Now, arguing similarly to~\eqref{eq:case of one cycle} we have
\begin{multline*}
\left|E_{G/\C_t(G)}\right|\stackrel{\eqref{eq:number of edges in
quotient}}{=}|E_G|-|\Gamma|\le (1+\d)|V_G|-|\Gamma|
=(1+\d)|V_G\setminus \Gamma|+\d|\Gamma|\\\stackrel{\eqref{eq:factor
3}}{\le} \left(1+\d+\frac{2\d}{1-3\d}\right)|V_G\setminus
\Gamma|\stackrel{\eqref{eq:number of vertices quotient}}{\le}
\frac{1-3\d^2}{1-3\d}\left(\left|V_{G/\C_t(G)}\right|-2\right).\tag*{\qedhere}
\end{multline*}
\end{proof}

\begin{corollary}\label{coro:s-1 condition for the quotient}
If $0<\d<\frac13$ then the quotient graph $G/\C_t(G)$ satisfies~\eqref{eq:|S|-1} with $\d$ replaced by $4\d/(1-3\d)$, i.e,
$$
\forall\, S\subseteq V_{G/\C_t(G)},\qquad \left|E_{G/\C_t(G)}(S)\right|\le \frac{1+\d}{1-3\d}(|S|-1).
$$
\end{corollary}

\begin{proof} Fix $S\subset V_{G/\C_t(G)}$ and let $S_1,\ldots,S_m\subseteq S$ be the
connected components of the graph $(S,E_{G/\C_t(G)}(S))$. For every $i\in \{1,\ldots,m\}$
 we lift each $S_i$ to a subset $U_i\subseteq V_G$ that is given by
\begin{equation*}
U_i\eqdef \big(S_i\cap (V_G\setminus \Gamma)\big)\bigcup\left(\bigcup_{C\in S_i\cap \C_t(G)}C\right).
\end{equation*}
Since $(S_i,E_{G/\C_t(G)}(S_i))$ is connected, the graph $H_i\eqdef
(U_i,E_G(U_i))$ is connected as well. By definition
$H_i/\C_t(H_i)=(S_i,E_{G/\C_t(G)}(S_i))$, so, since $H_i$ is
connected and $(1+\d)$-sparse, by Lemma~\ref{lem:edges in entire
quotient} we have
\begin{equation}\label{eq:edges in component}
\forall\,i\in \{1,\ldots,m\},\qquad \left|E_{G/\C_t(G)}(S_i)\right|\le \frac{1+\d}{1-3\d}(|S_i|-1).
\end{equation}
Consequently,
\begin{multline*}
\left|E_{G/\C_t(G)}(S)\right|=\sum_{i=1}^m \left|E_{G/\C_t(G)}(S_i)\right|\stackrel{\eqref{eq:edges in component}}{\le} \frac{1+\d}{1-3\d}\left(\sum_{i=1}^m |S_i|-m\right)\\=\frac{1+\d}{1-3\d}\left(|S|-m\right)\le \frac{1+\d}{1-3\d}\left(|S|-1\right).\tag*{\qedhere}
\end{multline*}
\end{proof}

Before proceeding it will be convenient to introduce the following
notation. Firstly, recalling~\eqref{eq:defGamma}, the set of edges
$E_G(\Gamma)\subseteq E_G$ consists of those edges in $E_G$ that
belong to a cycle of length less than $t$. Since, by
Lemma~\ref{lem:cycles are far}, these cycles are pairwise disjoint,
we can associate to every $e\in E_G(\Gamma)$ a unique cycle $C_e\in
\C_t(G)$ such that $e\in E_G(C_e)$. Secondly, let  $\T_{\mathrm
{good}}(G)\subset \T(G)$ be the set of those spanning trees $T$ of
$G$ satisfying $|E_T\cap E_G(C)|=|C|-1$ for every $C\in \C_t(G)$,
i.e.,
\begin{equation*}\label{eq:def good trees}
\T_{\mathrm {good}}(G)\eqdef \left\{T\in \T(G):\ \forall\, C\in \C_t(G),\quad |E_T\cap E_G(C)|=|C|-1\right\}.
\end{equation*}

\begin{lemma}\label{lem:measure on good trees}
Suppose that $0<\d<\frac13$. There exists a probability measure
$\mu$ on $\T_{\mathrm {good}}(G)$ such that for every $e\in
E_G\setminus E_G(\Gamma)$ we have
\begin{equation}\label{eq:prob e on long cycle}
\mu\left(\left\{T\in
\T_{\mathrm {good}}(G):\ e\in E_T\right\}\right)\ge \frac{1-3\d}{1+\d},
\end{equation}
and in addition for every $e\in E_G(\Gamma)$ we have
\begin{equation}\label{eq:prob e on short cycle}
 \mu\left(\left\{T\in \T_{\mathrm {good}}(G):\ e\in E_T\right\}\right)\ge \frac{|C_e|-1}{|C_e|}.
\end{equation}
\end{lemma}

\begin{proof} Due to Corollary~\ref{coro:s-1 condition for the quotient}, we can apply Claim~\ref{claim:spanning}
to the quotient graph $G/\C_t(G)$, thus obtaining a probability
measure $\nu$ on $\T(G/\C_t(G))$ such that for every $e\in
E_{G/\C_t(G)}$ we have
\begin{equation}\label{eq:quotient measure nu}
\nu\big(\{T\in \T(G/\C_t(G)):\ e\in E_T\}\big)\ge \frac{1-3\d}{1+\d}.
\end{equation}
By the definition of $G/\C_t(G)$, there is a bijection between the
quotient edges $E_{G/\C_t(G)}$ and $E_G\setminus E_G(\Gamma)$. Under
this bijection,  $\nu$ can be lifted to a probability measure
$\sigma$ on the subsets of $E_G\setminus E_G(\Gamma)$. Let $\tau$ be
the probability measure on the subsets of $E_G(\Gamma)$ given by
selecting a subset of size $|C|-1$  uniformly at random from each
$C\in \C_t(\Gamma)$, where these selections are performed
independently for different cycles in $\C_t(G)$. If $A\subseteq
E_G\setminus E_G(\Gamma)$ and $B\subseteq E_G(\Gamma)$ are such that
$\sigma(A),\tau(B)>0$ then $A\cup B$ form the edges of a spanning
tree of $G$, which by design belongs to  $\T_{\mathrm {good}}(G)$.
Thus $\sigma\times \tau$ induces a probability measure $\mu$ on
$\T_{\mathrm {good}}(G)$. The desired estimate~\eqref{eq:prob e on
long cycle} holds true due to~\eqref{eq:quotient measure nu}. The
desired estimate~\eqref{eq:prob e on short cycle} holds true because
if $e\in E_G(\Gamma)$ and $T\in \T_{\mathrm {good}}(G)$ is
distributed according to $\mu$ then each subset of $C_e$ of size
$|C_e|-1$ is contained in $E_T$ with probability $1/|C_e|$.
\end{proof}

\begin{proof}[Proof of Theorem~\ref{thm:our spanning no girth}]
Let $\mu$ be the probability measure on $\T_{\mathrm {good}}(G)$
from Lemma~\ref{lem:measure on good trees}. We obtain from $\mu$ a
probability distribution $\mu_\Sigma$ over spanning trees of the
$1$-dimensional simplicial complex $\Sigma(G)$ as follows. Given
$T_0\in \T_{\mathrm {good}}(G)$ that is distributed according to
$\mu$, define $T\in \T(\Sigma(G))$ as follows. Include all the edges
of $T_0$ as unit intervals in $T$. Let $\{U_e\}_{E_G\setminus
E_{T_0}}$ be i.i.d. random variables that are uniformly distributed
on $[0,1]$, and attach to $T$ a closed interval of length $U_e$ at
the vertex $x$ and a half open interval of length $1-U_e$ at  the
vertex $y$ (here we arbitrarily choose a labeling of the endpoints
of every edge in $E_G$ as $x$ and $y$). The distribution of the
resulting spanning tree $T\in \T(\Sigma(G))$ is denoted
$\mu_\Sigma$.

It suffices to prove~\eqref{eq:no girth spanning} when $x$ and $y$
lie on the same interval corresponding to an edge $e\in E_G$. Let
$[x,y]$ denote the geodesic joining $x$ and $y$ in $\Sigma(G)$ (it
is a sub-interval of the unit interval corresponding to $e$). Given
$T\in \T(\Sigma(G))$, write $[x,y]\subseteq T$ if and only if the
geodesic $[x,y]$ is also a geodesic in $T$. We distinguish between
two cases.

\medskip

\noindent{\bf Case 1.} We have $e\in E_G\setminus E_G(\Gamma)$. Then
there exists a sub-interval $I_{x,y}\subseteq [0,1]$ of length
$d_{\Sigma(G)}(x,y)$ such that

\begin{eqnarray}\label{use prob on big cycle}
&&\!\!\!\!\!\!\!\!\!\!\!\!\!\!\!\!\!\!\!\!\!\!\!\!\!\!\!\!\!\!\!\nonumber
\mu_{\Sigma}\big(\left\{T\in \T(\Sigma(G)):\ [x,y]\not\subseteq
T\right\}\big)\\ \nonumber &=&\mu\big(\left\{T\in \T(G):\ e\notin
E_T\right\}\big)\cdot\Pr\left[U_e\in
I_{x,y}\right]\\&\stackrel{\eqref{eq:prob e on long cycle}}{\le}&
\frac{4\d}{1+\d}d_{\Sigma(G)}(x,y).
\end{eqnarray}
Consequently,
\begin{eqnarray*}
&&\!\!\!\!\!\!\!\!\!\!\!\!\!\!\!\!\!\!\!\!\!\!\!\int_{\T(\Sigma(G))} \min\left\{d_{T}(x,y),\diam(G)\right\}d\mu_\Sigma(T)\\
&\lesssim& \mu_{\Sigma}\big(\left\{T\in \T(\Sigma(G)):\ [x,y]\subseteq T\right\}\big)\cdot d_{\Sigma(G)}(x,y)\\&&\quad+ \mu_{\Sigma}\big(\left\{T\in \T(\Sigma(G)):\ [x,y]\not\subseteq T\right\}\big)\cdot \diam(\Sigma(G))\\
&\stackrel{\eqref{use prob on big cycle}}{\lesssim}& \big(1+\d\diam(G)\big)d_{\Sigma(G)}(x,y),
\end{eqnarray*}
which is the desired inequality~\eqref{eq:no girth spanning}.

\medskip

\noindent{\bf Case 2.} We have $e\in E_G(\Gamma)$. Then arguing as in~\eqref{use prob on big cycle} with the use of~\eqref{eq:prob e on long cycle} replaced by the use of~\eqref{eq:prob e on short cycle},
\begin{equation}\label{eq:1/Ce}
\mu_{\Sigma}\big(\left\{T\in \T(\Sigma(G)):\ [x,y]\not\subseteq T\right\}\big)\le \frac{d_{\Sigma(G)}(x,y)}{|C_e|}.
\end{equation}
Moreover, if $T_0\in \T_{\mathrm {good}}(G)$ then, since $T_0$ contains all but one of the edges of $C_e$, we have $d_T(x,y)\le |C_e|$. Since $\mu$ is supported on $\T_{\mathrm {good}}(G)$, we conclude that
\begin{eqnarray*}
&&\!\!\!\!\!\!\!\!\!\!\!\!\!\!\!\!\!\!\!\!\!\!\!\!\!\!\!\!\!\int_{\T(\Sigma(G))} d_{T}(x,y)d\mu_\Sigma(T)\\
&\le& \mu_{\Sigma}\big(\left\{T\in \T(\Sigma(G)):\ [x,y]\subseteq T\right\}\big)\cdot d_{\Sigma(G)}(x,y)\\&&\quad+ \mu_{\Sigma}\big(\left\{T\in \T(\Sigma(G)):\ [x,y]\not\subseteq T\right\}\big)\cdot |C_e|\\
&\stackrel{\eqref{eq:1/Ce}}{\le}& 2d_{\Sigma(G)}(x,y),
\end{eqnarray*}
proving the desired inequality~\eqref{eq:no girth spanning} in Case 2 as well.
\end{proof}

\begin{remark}
If $\d\le 1/(15\diam(G))$ in Corollary~\ref{cor:no girth into L_1}
then one can prove using Lemma~\ref{lem:cycles are far} that $G$
contains at most one cycle, hence $c_1(G)=1$.
\end{remark}

\section{Geometric properties of random regular graphs}\label{sec:random}


Fix $n,d\in \N$. While we are interested in proving estimates about
the probability distribution $\mathcal{G}_{n,d}$, it is often
simpler to argue about a more tractable probability distribution on
$\G_n$, denoted $\mathcal{P}_{n,d}$ and called the {\em pairing
model}; see~\cite[Sec.~2]{Wormald-survey} and the references
therein.

To define $\mathcal{P}_{n,d}$ assume from now on that $nd$ is even
and let $M$ be a uniformly random perfect matching of the set
$\{1,\ldots,n\}\times \{1,\ldots,d\}$ (recall that a perfect
matching is a partition into subsets of size $2$). Now define a
random $d$-regular graph $G$ with $V_{G}=\{1,\ldots, n\}$ and
$$
\forall\, i,j\in \{1,\ldots,n\},\qquad E_G(i,j)\eqdef
\sum_{a=1}^d\sum_{b=1}^d \1_{\left\{\{(i,a),(j,b)\}\in M\right\}}.
$$
In other words, the edges of $G$ are obtained by ``projecting" the
perfect matching $M$ onto $\{1,\ldots,n\}$, i.e., whenever
$\{(i,a),(j,b)\}\in M$ we add an edge joining $i$ and $j$ in $G$.
The graph $G$ that is obtained in this way from a perfect matching
$M$ of $\{1,\ldots,n\}\times \{1,\ldots,d\}$ is denoted $G(M)$.

The probability measures $\mathcal{G}_{n,d}$ and $\mathcal{P}_{n,d}$
are {\em contiguous} in the following sense. As explained
in~\cite[Sec.~2.2]{Wormald-survey}, there exists $\alpha(d)\in
(0,\infty)$ such that
\begin{equation}\label{eq:continguous}
\forall\, A\in \G_n,\qquad
\mathcal{G}_{n,d}(A)\le\alpha(d)\cdot \mathcal{P}_{n,d}(A).
\end{equation}
By~\cite{Wormald-survey}, for $d\lesssim \sqrt[3]{n}$
we have $\log\alpha(d)\lesssim d^2$, but we will not explicitly
state the dependence on $d$ from now on.

There is another standard probability distribution on $n$-vertex
simple graphs: the Erd\"os-R\'enyi model $G(n,p)$ for $p\in (0,1)$.
The statements below on the sparsity of subsets of random graphs
have been proved in~\cite{ABLT} for the $G(n,p)$ model, and here we need to extend them to the
$\mathcal{G}_{n,d}$ model. This leads to some technical
changes, but the essence of the argument is the same as
in~\cite{ABLT}.

The following lemma is well-known, and it follows from much more precise estimates that are available in the literature (see e.g.~\cite{McK81}). We include its straightforward proof for the sake of completeness.
\begin{lemma}\label{lem:edge prob upper}
Let $F$ be a set of unordered pairs of vertices in $\n$ with
$|F|<nd/4$. Then
\begin{equation}\label{eq:matching model edge containment}
\mathcal{P}_{n,d}\left(\left\{G\in \G_n:\ F\subset E_G\right\}\right)\le \left(\frac{2d}{n}\right)^{|F|}.
\end{equation}
\end{lemma}

\begin{proof} Write $k=|F|$ and $F=\{f_1,\ldots,f_k\}$. Let $M$ be a uniformly random matching of
$\n\times \{1,\ldots,d\}$. For every $\ell\in \{1,\ldots,k\}$ write
$f_\ell=\{i_\ell,j_\ell\}$.

We claim that for every $\ell\in \{1,\ldots,k\}$ we have
\begin{equation}\label{eq:conditional probability}
\Pr\left[f_\ell\in G(M)\big| f_1,\ldots,f_{\ell-1}\in G(M)\right]\le \frac{2d}{n},
\end{equation}
where $\Pr$ is the uniform probability on the matching $M$.
Once~\eqref{eq:conditional probability} is proved, the desired
estimate~\eqref{eq:matching model edge containment} would follow
because by the definition of $\mathcal{P}_{n,d}$ we have
\begin{multline*}
\mathcal{P}_{n,d}\left(\left\{G\in \G_n:\ F\subset
E_G\right\}\right) \\
=\prod_{\ell=1}^k \Pr\left[f_\ell\in G(M)\big|
f_1,\ldots,f_{\ell-1}\in G(M)\right]\stackrel{\eqref{eq:conditional
probability}}{\le} \left(\frac{2d}{n}\right)^{k}.
\end{multline*}

To verify~\eqref{eq:conditional probability} observe that in order
for $\{i_\ell,j_\ell\}$ to be in $E_{G(M)}$ there must be $a,b\in
\{1,\ldots,d\}$ such that $\{(i_\ell,a),(j_\ell,b)\}\in M$. For
every fixed $a\in \{1,\ldots,d\}$, if  we know that
$f_1,\ldots,f_{\ell-1}\in G(M)$ then there are $nd-2(\ell-1)-1$
possible elements of $\n\times \{1,\ldots,d\}$ that can be matched
by $M$ to $(i_\ell,a)$. We are assuming that $\ell\le k<nd/4$, so
that $nd-2(\ell-1)-1\ge nd/2$. We therefore have at least $nd/2$
pairs in $\n\times \{1,\ldots,d\}$ that can be matched to
$(i_\ell,a)$, from which at most $d$ can project to the edge
$f_\ell$. The probability for this to happen is therefore at most
$2/n$. The desired estimate~\eqref{eq:conditional probability} now
follows since there are at most $d$ possible values of $a$.
\end{proof}

\begin{definition}\label{def:S class}
For every $\e,\d\in (0,\infty)$ denote by $\mathcal{S}_{\e,\d}$ the
set of all graphs $G$ with the property that $|E_G(S)|< (1+\d)|S|$
for every $S\subseteq V_G$ satisfying $|S|\le |V_G|^{1-\e}$.
\end{definition}

Lemma~\ref{lem:random graphs are sparse} below is similar to~\cite[Lem~2.8]{ABLT}, which treats the same question for Erd\"os-R\'enyi graphs (see also~\cite[Lem.~3.10]{ALNRRV}).

\begin{lemma}\label{lem:random graphs are sparse}
For $\e\in (0,1)$ and two integers $n,d\ge 3$ define
\begin{equation}\label{eq:def our delta sparse}
\d\eqdef \frac{7\log d}{\e\log n}.
\end{equation}
Then
\begin{equation}\label{eq:gnd sparse statement}
1-\mathcal{G}_{n,d}\left(\mathcal{S}_{\e,\d}\right)\lesssim_{d,\e} \frac{1}{n^{1-\e}}.
\end{equation}
\end{lemma}

\begin{proof} By adjusting the implicit constant in~\eqref{eq:gnd sparse
statement} if necessary, we may assume in the computations below
that $n$ is larger than an appropriate constant that may depend only
on $\e$ and $d$. In particular, we assume throughout that $n>
d^{7/\e}$, or equivalently that $\d\in (0,1)$.

For every $k\in \{3,\ldots,n\}$ define $\B^k_\d$ to be the set of
all graphs $G$ for which there exists $S\subseteq V_G$ with $|S|=k$
and $|E_G(S)|\ge (1+\d)k$.  The complement of the event
$\mathcal{S}_{\e,\d}$ is the union of $\B^k_\d$ for $k\in
\{3,\ldots, \lfloor n^{1-\e}\rfloor\}$. Therefore, due
to~\eqref{eq:continguous} we have
\begin{equation}\label{eq:use continuity}
1-\mathcal{G}_{n,d}\left(\mathcal{S}_{\e,\d}\right)\lesssim_d
1-\mathcal{P}_{n,d}\left(\mathcal{S}_{\e,\d}\right)\le \sum_{k=3}^{\lfloor n^{1-\e}\rfloor}
\mathcal{P}_{n,d}\left(\B^k_\d\right).
\end{equation}
Since our assumption $n> d^{7/\e}$ implies that $k\le nd/16$, by
Lemma~\ref{lem:edge prob upper},
\begin{align}
\nonumber \mathcal{P}_{n,d}\left(\B^k_\d\right) &\le
\binom{n}{k}\binom{\binom{k}{2}}{\lceil (1+\d)k\rceil}
\left(\frac{2d}{n}\right)^{\lceil (1+\d)k\rceil}\\ \label{eq:use stirling} &\le
\left(\frac{en}{k}\right)^k \left(\frac{ek(k-1)d}{n\lceil
(1+\d)k\rceil}\right)^{\lceil (1+\d)k\rceil}\\  &\le
\left(\frac{en}{k}\left(\frac{e(k-1)d}{(1+\d)n}\right)^{1+\d}\right)^k\le
\left(\frac{e^3k^\d d^{1+\d}}{n^\d}\right)^k.\label{eq:B event upper bound}
\end{align}
In~\eqref{eq:use stirling}  we used the standard estimate
$\binom{m}{\ell}\le \left(\frac{em}{\ell}\right)^\ell$, which holds
for every $m\in \N$ and $\ell\in \{1,\ldots,m\}$. In the penultimate
inequality of~\eqref{eq:B event upper bound} we used the fact that
the function $x\mapsto (ek(k-1)d/(nx))^x$ is decreasing when $x\ge
dk(k-1)/n$, and $\lceil (1+\d)k\rceil\ge (1+\d)k\ge dk(k-1)/n$ by
our assumption $n> d^{7/\e}$.

If $k< 1/\d$ then $\lceil (1+\d)k\rceil=k+1$, and
therefore~\eqref{eq:use stirling} implies that $
\mathcal{P}_{n,d}\left(\B^k_\d\right)\lesssim k(e^2d)^{k+1}/n$.
Using this estimate in combination with~\eqref{eq:B event upper
bound} for $k\ge 1/\d$, it follows from~\eqref{eq:use continuity}
that

\begin{align}\label{eq:to split integral}
\nonumber 1-\mathcal{G}_{n,d}\left(\mathcal{S}_{\e,\d}\right)&\lesssim_d \frac{1}{n}\int_3^{1+1/\d}x(e^2d)^{x+1}dx
+\int_{1/\d}^{n^{1-\e}+1}\left(\frac{e^3x^\d d^{1+\d}}{n^\d}\right)^xdx
\\ &\lesssim_d \frac{(e^2d)^{1/\d}}{n\d}+\int_{1/\d}^{2n^{1-\e}}\left(\frac{e^3x^\d d^{1+\d}}{n^\d}\right)^xdx.
\end{align}

To estimate the final integral in~\eqref{eq:to split integral},
observe first that due to~\eqref{eq:def our delta sparse}, its
integrand is less than $1$, i.e,  $e^32^\d
n^{(1-\e)\d}d^{1+\d}/n^\d<1$. Since $\d\le 1$ and $n^{\e\d}=d^7$,
this would follow from $d^7>2e^3d^2$, which is true since $d\ge 3$.
Now fix $M\in [1/\d,2n^{1-\e}]$ and proceed as follows.
\begin{align}\label{eq:optimize M}
\nonumber&\int_{1/\d}^{2n^{1-\e}}\left(\frac{e^3x^\d d^{1+\d}}{n^\d}\right)^xdx
\\\nonumber &\le \int_{1/\d}^M
\left(\frac{e^3M^\d d^{1+\d}}{n^\d}\right)^xdx+ \int_M^{2n^{1-\e}}
\left(\frac{e^3(2n^{(1-\e)})^\d d^{1+\d}}{n^\d}\right)^xdx\\ \nonumber
&\le \int_{1/\d}^\infty
\left(\frac{e^3M^\d d^{1+\d}}{n^\d}\right)^xdx+ \int_M^{\infty}
\left(\frac{2^\d e^3d^{1+\d}}{n^{\e\d}}\right)^xdx\\ \nonumber
&= \frac{\left(e^3M^\d d^{1+\d}n^{-\d}\right)^{1/\d}}{\log\left(n^\d e^{-3}M^{-\d}d^{-1-\d}\right)}+
\frac{\left(2^\d e^3 d^{1+\d}n^{-\e\d}\right)^{M}}{\log\left(n^{\e\d} 2^{-\d} e^{-3}d^{-1-\d}\right)}\\
&\lesssim_d \frac{M(e^3d)^{1/\d}}{n}+\frac{1}{d^M},
\end{align}
where in~\eqref{eq:optimize M} we used the fact that $n^{\e\d}=d^7$
and $d\ge 3$. Choosing $M=\log n$ and
substituting~\eqref{eq:optimize M} into~\eqref{eq:to split integral}
while using~\eqref{eq:def our delta sparse} and the fact that
$e^3d\le d^4$ (because $d\ge 3$), we conclude that
\begin{equation*}
1-\mathcal{G}_{n,d}\left(\mathcal{S}_{\e,\d}\right)\lesssim_d \frac{d^{4/\d}\log n}{n}
\stackrel{\eqref{eq:def our delta sparse}}{=}\frac{n^{4\e/7}\log n}{n}\lesssim_\e \frac{1}{n^{1-\e}}.\qedhere
\end{equation*}
\end{proof}

The final preparatory lemma that we will need about graphs sampled
from $\mathcal{G}_{n,d}$ is that they have only a few short cycles. Specifically, Lemma~\ref{lem:number of short cycles} below asserts that the probability that a graph sampled from $\mathcal{G}_{n,d}$ has less than $\sqrt{n}$ cycles of length $\lceil (\log_{d-1}n)/7\rceil$ is at least $1-c(d)/\sqrt[3]{n}$, where $c(d)\in (0,\infty)$ may depend only on $d$. This is a standard fact but we include its simple proof here since we could not locate the statement below in the literature. Precise asymptotics of the expected numbers of short cycles of a
graph sampled from $\mathcal{G}_{n,d}$ were obtained in~\cite{MWW},
and Lemma~\ref{lem:number of short cycles} follows from the
estimates of~\cite{MWW} by Markov's inequality.

\begin{lemma}\label{lem:number of short cycles}
For every two integers $n,d\ge 3$ we have
$$
\mathcal{G}_{n,d}\left(\left\{G\in \G_n:\ \left|\C_{\lceil (\log_{d-1}n)/7\rceil}(G)\right|\ge \sqrt{n}\right\}\right)\lesssim_d \frac{1}{\sqrt[3]{n}}.
$$
(Recall that the set of cycles of length less than $t\in \N$ in a
graph $G$ was denoted in~\eqref{eq:def C_t} by $\C_t(G)$.)
\end{lemma}

\begin{proof}
For $r\in \N$ write $X_r(G)\eqdef |\C_{r+1}(G)|-|\C_r(G)|$, i.e.,
$X_r(G)$ is the number of cycles of length $r$ in $G$.
By~\cite[Eq.~(2.2)]{MWW}, for $r\le (\log_{d-1}n)/7$ we have the
expectation bound $\int_{\G_n} X_r(G)d\mathcal{G}_{n,d}\lesssim
(d-1)^r$. Hence by Markov's inequality, for every integer $r\le
(\log_{d-1}n)/7$,
\begin{multline*}
\mathcal{G}_{n,d}\left(\left\{G\in \G_n:\ X_r(G)\ge \frac{\sqrt{n}}{\lceil (\log_{d-1}n)/7\rceil}\right\}\right)\\\lesssim
\frac{(d-1)^{\lceil (\log_{d-1}n)/7\rceil}\log_{d-1}n}{\sqrt{n}}\lesssim_d \frac{\log n}{n^{5/14}}.
\end{multline*}
Consequently,
\begin{align*}
&\mathcal{G}_{n,d}\left(\left\{G\in \G_n:\ \left|\C_{\lceil (\log_{d-1}n)/7\rceil}(G)\right|\ge \sqrt{n}\right\}\right)\\&= \mathcal{G}_{n,d}\left(\left\{G\in \G_n:\ \sum_{r=3}^{\lceil (\log_{d-1}n)/7\rceil} X_r(G)\ge \sqrt{n}\right\}\right)\\
&\le \sum_{r=3}^{\lceil(\log_{d-1}n)/7\rceil}\mathcal{G}_{n,d}\left(\left\{G\in \G_n:\ X_r(G)\ge \frac{\sqrt{n}}{\lceil (\log_{d-1}n)/7\rceil}\right\}\right)\\&\lesssim_d \frac{(\log n)^2}{n^{5/14}} \lesssim \frac{1}{\sqrt[3]{n}}.\qedhere
\end{align*}
\end{proof}

\subsection{Proof of Lemma~\ref{lem:random graph in
good family}}\label{sec:proof of L family lemma}  By
lemma~\ref{lem:random graphs are sparse} we know that if $G$ is
sampled from $\mathcal{G}_{n,d}$ then with high probability for
every $S\subseteq V_G$ with $|S|\le n^{1-\e}$ the graph $(S,E_G(S))$
is $(1+\d)$-sparse. One is then tempted to use Corollary~\ref{cor:no
girth into L_1} in order to embed the metric space $(S,d_G)$ into
$L_1$, but this is problematic since the metric that the graph
$(S,E_G(S))$ induces on $S$ can be very different from the
restriction of $d_G$ to $S$ (in fact, $(S,E_G(S))$ may not even be
connected). Following an idea of~\cite{ALNRRV}, the following lemma will be used to remedy this
matter.

\begin{lemma}\label{lem:add geodesics}
Fix $n,d,r\in \N$ and let $G$ be a connected $n$-vertex graph whose
maximum degree is $d$. Suppose that $S,T\subset V_G$ satisfy
\begin{equation}\label{eq:S in T neighborhood}
S\subset \bigcup_{u\in T}B_G(u,r),
\end{equation}
where $B_G(u,r)\eqdef \{v\in V_G:\ d_G(u,v)\le r\}$ denotes the
ball of radius $r$ and center $u$ in the metric space $(V_G,d_G)$. Then there exists $U\subset V_G$ with $S\subset U$ such that
\begin{equation}\label{eq:U size}
|U|\le |T|\left(d(d-1)^{3r-1}+\diam(G)\right),
\end{equation}
and if we let $H$ denote the graph $(U,E_G(U))$ then
$$
\diam(H)\le 6r+2\diam(G),
$$
and for every $a,b\in S$,
\begin{equation}\label{eq:undistorted subgraph}
d_G(a,b)\le d_H(a,b)\le 2\left(\frac{\diam(G)}{r}+1\right) d_G(a,b).
\end{equation}
\end{lemma}

\begin{proof}
For every $x,y\in V_G$ let $P_{x,y}\subseteq V_G$ be an arbitrary
shortest path joining $x$ and $y$ in $G$. Fix $z\in V_G$ and define
$$
U\eqdef \left(\bigcup_{x\in T}B_G(x,3r)\right)\bigcup\left(\bigcup_{x\in T}P_{x,z}\right).
$$
Then,
\begin{align}\label{eq:use ball growth}
\nonumber |U|&\le \sum_{x\in T} |B_G(x,3r)|+\sum_{x\in T} |P_{x,z}|\\&\le |T|d(d-1)^{3r-1}+\sum_{x\in T}(d_G(x,z)-1)+1\\
&\le |T|d(d-1)^{3r-1}+|T|(\diam(G)-1)+1\nonumber,
\end{align}
where in~\eqref{eq:use ball growth} we used the fact that since $G$
has maximum degree $d$, for every $k\in \N$ and $w\in G$ we have
$|B_G(w,k)|\le d(d-1)^{k-1}$. The fact that $\diam(H)\le
6r+2\diam(G)$ is immediate from the definition of $U$.

Having proved~\eqref{eq:U size}, we proceed to
prove~\eqref{eq:undistorted subgraph}. Because $H$ is a subgraph of
$G$, the leftmost inequality in~\eqref{eq:undistorted subgraph}
holds true for every $a,b\in V_G$. So, fixing $a,b\in S$, it remains
to prove the rightmost inequality in~\eqref{eq:undistorted
subgraph}. To do so we distinguish between two cases.

\medskip

\noindent{\bf Case 1.} There exists $x\in T$ such that $a\in
B_G(x,r)$ and $b\in B_G(x,2r)$. Then $d_G(a,b)\le
d_G(a,x)+d_G(b,x)\le 3r$. Thus $|P_{a,b}|\le 3r+1$. Write
$P_{a,b}=\{w_0=a,w_1,\ldots,w_m=b\}$ with $m\le 3r$ and
$\{w_{i-1},w_i\}\in E_G$ for every $i\in \{1,\ldots,m\}$. For every
$i\in \{0,\ldots,2r\}$ we have $d_G(w_i,a)\le 2r$, and therefore
$d_G(w_i,x)\le d_G(w_i,a)+d_G(x,a)\le 3r$. Similarly, for every
$i\in \{2r+1,\ldots,m\}$ we have $d_G(w_i,b)\le r$, and therefore
$d_G(w_i,x)\le 3r$. This shows that $P_{a,b}\subseteq
B_G(x,3r)\subseteq U$, and hence $d_G(a,b)=d_H(a,b)$.

\medskip

\noindent{\bf Case 2.} There exist distinct $x,y\in T$ such that
$a\in B_G(x,r)$ and $b\in B_G(y,r)\setminus B_G(x,2r)$. Then
\begin{equation}\label{eq:d_G(a,b) lower r}
d_G(a,b)\ge d_G(x,b)-d_G(x,a)\ge r.
\end{equation}
Since $a\in B_G(x,r)$ we have $P_{a,x}\subset B_G(x,r)\subset U$.
For the same reason, $P_{b,y}\subset U$. Also, by the definition of
$U$ we have $P_{x,z},P_{y,z}\subset U$. By considering the path
$P_{a,x}\cup P_{x,z}\cup P_{y,z}\cup P_{b,y}\subseteq U$ we see that
\begin{multline*}
d_H(a,b)\le d_G(a,x)+d_G(x,z)+d_G(y,z)+d_G(b,y)\\\le 2r+2\diam(G)
\stackrel{\eqref{eq:d_G(a,b) lower r}}{\le}
2\left(1+\frac{\diam(G)}{r}\right)d_G(a,b),
\end{multline*}
completing the verification of~\eqref{eq:undistorted subgraph} in
this case as well.

\medskip

Note that due to~\eqref{eq:S in T neighborhood} one of the above two
cases must occur, so the proof of~\eqref{eq:undistorted subgraph} is
complete.
\end{proof}

\begin{definition}\label{def:deleted cycle graph}
Fix two integers $d,t\ge 3$ and a $d$-regular simple graph $G$. For
every cycle $C$ of $G$ fix an arbitrary edge $e_C\in E_G(C)$. Denote
$$
I_G^t\eqdef \left\{e_C:\ C\in  \C_t(G)\right\}.
$$
Thus $I_G^t$ contains a single representative edges from each cycle
of $G$ of length less than $t$ (formally, $I_G$ depends also on the
choices of the edges $e_C$, but we fix such a choice once and for
all and do not indicate this dependence explicitly). Denote the
graph $(V_G,E_G\setminus I^t_G)$ by $L^t_G$. Note that by definition
we have $\girth(L_G^t)\ge t$.
\end{definition}

Before proceeding we record for future use two simple lemmas about
the concepts that were introduced in Definition~\ref{def:deleted
cycle graph}.

\begin{lemma}\label{lem:edge deletion distant cycles} Fix three integers $r,t,d\ge 3$ and write
\begin{equation}\label{eq:choose eta}
\eta\eqdef \frac{1}{r+2t-3}.
\end{equation}
Suppose that $G$ is a connected $d$-regular graph such that for
every subset $S\subseteq V_G$ we have
\begin{equation}\label{eq:sparsity assumption edge deletion}
|S|\le d(d-1)^{\frac{t+r}{2}-1}\cdot |\C_t(G)| \implies |E_G(S)|\le (1+\eta)|S|.
\end{equation}
Then for every distinct cycles $C_1,C_2\in \C_t(G)$ we have
$d_G(C_1,C_2)\ge r$. Moreover, the graph $L^t_G$ is connected and
\begin{equation}\label{eq:L_G diameter}
\diam\left(L^t_G\right)\le \frac{t+r-1}{r+1}\cdot \diam(G)+\frac{r(t-2)}{r+1}.
\end{equation}
\end{lemma}

\begin{proof}
Define a set of edges $F\subseteq E_G$ by
\begin{equation}\label{eq:def F cycle deletion}
F\eqdef \bigcup_{e\in I_G^t}\left\{f\in E_G:\ d_G(e,f)\le \left\lfloor \frac{t-1}{2}\right\rfloor+
\left\lfloor \frac{r-1}{2}\right\rfloor\right\}.
\end{equation}
Also, let $U\eqdef \bigcup_{f\in F} f\subset V_G$ be the set of
vertices that belong to an edge in $F$. Let $H$ be the graph
$(U,F)$. For every $C\in \C_t(G)$ choose an arbitrary vertex $v_C\in
e_C$. By the definition of $F$, for every $u\in U$  there exists
$C\in \C_t(G)$ such that
$$d_G(u,v_C)\le 1+\left\lfloor\frac{t-1}{2}\right\rfloor+\left\lfloor \frac{r-1}{2}\right\rfloor\le
\frac{t+r}{2}.$$ Consequently,
$$
|U|\le \sum_{C\in \C_t(G)}\left|B_G\left(v_C,\frac{t+r}{2}\right)\right|\le d(d-1)^{\frac{t+r}{2}-1}|\C_t(G)|.
$$
By our assumption~\eqref{eq:sparsity assumption edge deletion},  the
graph $H$ is $(1+\eta)$-sparse. An application of
Lemma~\ref{lem:cycles are far} to the connected components of $H$
shows that if $C_1$ and $C_2$ are distinct cycles of $H$ of length
less than $t$ then either they are contained in different connected
components of $H$ or
\begin{equation}\label{eq:cycles in H are far}
d_H(C_1,C_2)\ge 1+\frac{1}{\eta}-2(t-1)\stackrel{\eqref{eq:choose eta}}{=} r.
\end{equation}

It follows from~\eqref{eq:cycles in H are far} that any distinct
cycles $C_1,C_2\in \C_t(G)$ also satisfy $d_G(C_1,C_2)\ge r$.
Indeed, observe first that for every $C\in \C_t(G)$ we have
$C\subseteq U$. This is true because the diameter (in the metric
$d_G$) of $C$ is at most $\lfloor(|C|-1)/2\rfloor \le \lfloor
(t-1)/2\rfloor$, and therefore by~\eqref{eq:def F cycle deletion} we
have $E_G(C)\subseteq F$. This also shows that $\C_t(G)=
\C_t(H)$. Now suppose for the sake of obtaining a contradiction that
$C_1,C_2\in \C_t(G)$ and $0<d_G(C_1,C_2)\le r-1$. Then there exist
$u_1,\ldots,u_r\in V_G$ such that $u_1\in C_1$, $u_r\in C_2$ and
$\{u_{i-1},u_i\}\in E_G$ for every $i\in \{2,\ldots,r\}$. We have
$d_G(u_1,e_{C_1})\le \lfloor(t-1)/2\rfloor$ and $d_G(u_r,e_{C_2})\le
\lfloor(t-1)/2\rfloor$. This implies that $d_G(u_i,C_1\cup C_2)\le
\lfloor(t-1)/2\rfloor+\lfloor(r-1)/2\rfloor$ for every $i\in
\{1,\ldots,r\}$, and therefore by~\eqref{eq:def F cycle deletion} we
have $\{u_{i-1},u_i\}\in F$ for every $i\in \{2,\ldots,r\}$,
implying that $d_H(C_1,C_2)\le r-1$, in contradiction
to~\eqref{eq:cycles in H are far}.

To prove~\eqref{eq:L_G diameter}, take $u,v\in V_G$ and write
$d_G(u,v)=m\le \diam(G)$. Let $\{w_0=w,w_1,\ldots,w_m=v\}\subseteq
V_G$ be distinct vertices such that $\{w_{i-1},w_i\}\in E_G$ for
every $i\in \{1,\ldots,m\}$. We proceed to examine the edges on this
path that are not in $E_{L_G^t}$. Thus, suppose that
$j_1,\ldots,j_k\in \{0,\ldots,m-1\}$ are such that for every
$\ell\in \{1,\ldots,k\}$ we have
$\{w_{j_\ell},w_{j_{\ell}+1}\}=e_{C_\ell}$ for some $C_\ell\in
\C_t(G)$, and $\{w_{i-1},w_i\}\neq e_C$ for every $i\in
\{1,\ldots,m\}\setminus \{j_1,\ldots,j_k\}$ and every $C\in
\C_t(G)$. Since we proved that distinct cycles in $\C_t(G)$ are at
distance at least $r$ (in the metric $d_G$), it follows that
$j_{\ell+1}-j_\ell\ge r+1$ for every $\ell\in \{1,\ldots,k-1\}$.
Hence $m\ge (k-1)(r+1)+1$, or equivalently $k\le (m+r)/(r+1)$. Now,
for every $\ell\in \{1,\ldots,k\}$ replace each edge
$\{w_{j_\ell},w_{j_{\ell}+1}\}$ by the path $C_\ell\setminus
\{e_{C_\ell}\}$, whose length is at most $t-2$. By the definition of
$L_G^t$, we thus obtain a new path joining $u$ and $v$, all of whose
edges are edges of $L_G^t$, and its length is at most $m+k(t-2)$. By
substituting the above upper bound on $k$ we conclude that
\begin{equation*}
d_H(u,v)\le m+\frac{(m+r)(t-2)}{r+1}\le \frac{t+r-1}{r+1}\cdot \diam(G)+\frac{r(t-2)}{r+1}. \qedhere
\end{equation*}
\end{proof}

\begin{lemma}\label{lem:expansion deleted}
Fix three integers $M,t,d\ge 3$ and suppose that $G$ is a connected
$d$-regular simple graph such that $d_G(C_1,C_2)>8M$ for every
distinct $C_1,C_2\in \C_t(G)$. Suppose also that
\begin{equation}\label{eq:expansion M}
\min_{\substack{S\subseteq G\\0<|S|\le n/2}}\frac{E_G(S,V_G\setminus S)}{|S|}\ge \frac{1}{M}.
\end{equation}
Then
$$
\min_{\substack{S\subseteq G\\0<|S|\le n/2}}\frac{E_{L_G^t}(S,V_G\setminus S)}{|S|}\ge \frac{1}{4M}.
$$
\end{lemma}

\begin{proof}
Fix $S\subset V_G$ with $0<|S|\le |V_G|/2$. Write
$$
\mathscr{C}\eqdef \{C\in
\C_t(G):\ e_C\cap S\neq \emptyset\},
$$
 and for every $C\in \mathscr{C}$ choose
an arbitrary endpoint $u_C\in e_C\cap S$. Define
$$T\eqdef \{u_C\}_{C\in
\mathscr{C}}\subset V_G.
$$

If $|S|\ge 2M|T|$ then
$$
\frac{E_{L_G^t}(S,V_G\setminus S)}{|S|}\ge \frac{E_G(S,V_G\setminus S)-|T|}{|S|}\stackrel{\eqref{eq:expansion M}}{\ge} \frac{1}{2M}.
$$
So, we may assume from now on that $|S|<2M|T|$. Because every
distinct $x,y\in T$ satisfy $d_G(x,y)>8M$, the balls
$\{B_G(x,4M)\}_{x\in T}$ are disjoint. This implies that if we
define
$$
W\eqdef \left\{x\in T:\ \left|B_G\left(x,4M\right)\cap S\right|<4M\right\},
$$
then $|W|> |T|/2$, since otherwise $|S|\ge 4M|T\setminus W| \ge
2M|T|$.

For every $x\in W$ choose a maximal $k(x)\in \N\cup\{0\}$ for which
there exists $w_1(x),\ldots,w_k(x)\in S$ such that
$$
\{x,w_1(x)\},\{w_1(x),w_2(x)\},\ldots,\{w_{k-1}(x),w_k(x)\}\in
E_{L_G^t}=E_G\setminus I_G^t.
$$
In other words, we are considering a maximal path starting from $x$
in the graph $(S,E_{L_G^t}(S))$. Since
$\left|B_G\left(x,4M\right)\cap S\right|<4M$ it follows that
$k(x)\le 4M-1$. Since $w_{k(x)}(x)$ belongs to at most one edge in
$I_G^t$, and $w_{k(x)}(x)$ has $d\ge 3$ neighbors in $G$, it follows
that there exists $z(x)\in V_G\setminus \{w_{k(x)-1}(x)\}$ such that
$\{w_{k(x)},z(x)\}\in E_{L_G^t}$. By the maximality of $k(x)$ we
necessarily have $\{w_{k(x)},z(x)\}\in E_{L_G^t}(S,V_G\setminus S)$.
Moreover, for distinct $x,y\in W$ we have $\{w_{k(x)},z(x)\}\neq
\{w_{k(y)},z(y)\}$, since otherwise, because
$d_G(x,w_{k(x)}(x)),d_G(y,w_{k(y)}(x))\le 4M-1$ and
$d_G(x,z(x)),d_G(y,z(y))\le 4M$, it would follow that $d_G(x,y)\le
8M$.

We have proved that $E_{L_G^t}(S,V_G\setminus S)\ge
|W|>|T|/2>|S|/(4M)$, completing the proof of
Lemma~\ref{lem:expansion deleted}.
\end{proof}

\begin{proof}[Proof of Lemma~\ref{lem:random graph in
good family}] Fix a parameter $M\in (1,\infty)$ that will be
determined later. Let $\mathcal{E}^1_M$ denote the family of all
connected $n$-vertex $d$-regular simple graphs $G$ that satisfy the
following conditions
\begin{equation}\label{eq:diam M}
\diam(G)\le M\log_{d}n,
\end{equation}
and
\begin{equation*}
\min_{\substack{S\subseteq G\\0<|S|\le n/2}}\frac{E_G(S,V_G\setminus S)}{|S|}\ge \frac{1}{M}.
\end{equation*}
By the proofs in~\cite{BV} and~\cite{Bol88} we can fix $M$ to be a sufficiently large universal constant such that
\begin{equation}\label{eq:first event}
1-\mathcal{G}_{n,d}\left(\mathcal{E}^1_M\right)\lesssim_d \frac{1}{n}.
\end{equation}
(The papers~\cite{BV,Bol88} contain much more precise information on
the diameter and expansion of random regular graphs, respectively.)

By adjusting the value of the constant $C(d)$ in
Lemma~\ref{lem:random graph in good family}, it suffices to prove
the required statement under the assumption $n\ge d^{2000M}$. Define
$\d\in (0,1)$ and an integer $t\ge 3$ by
\begin{equation}\label{eq:choose t delta}
\d\eqdef \frac{21}{\log_d n},\qquad\mathrm{and}\qquad t\eqdef \left\lfloor \frac{\log_{d} n}{63}\right\rfloor.
\end{equation}
Also, set from now on $\e=1/3$. Let $\mathcal{E}^2$ denote the
family of graphs $\mathcal{S}_{\e,\d}$ (recall Definition~\ref{def:S
class}) and let $\mathcal{E}^3$ denote the family of all connected
$n$-vertex $d$-regular simple graphs $G$ that satisfy $
|\C_{t}(G)|\le \sqrt{n}$. By Lemma~\ref{lem:random graphs are
sparse} and Lemma~\ref{lem:number of short cycles} we have
\begin{equation}\label{eq:second and third events}
1-\mathcal{G}_{n,d}\left(\mathcal{E}^2\right)\lesssim_d \frac{1}{{n^{2/3}}}\qquad\mathrm{and}\qquad
1-\mathcal{G}_{n,d}\left(\mathcal{E}^3\right)\lesssim_d \frac{1}{\sqrt[3]{n}}
\end{equation}

By virtue of~\eqref{eq:first event} and~\eqref{eq:second and third
events}, if we denote $\mathscr{L}\eqdef
\mathcal{E}^1_M\cap\mathcal{E}^2\cap \mathscr{E}^3$ then
$$
1-\mathcal{G}_{n,d}\left(\mathscr{L}\right)\lesssim_d \frac{1}{\sqrt[3]{n}}.
$$
We will now show that for a large enough universal constant $K\in
(0,\infty)$ the graph family $\mathscr{L}$ can serve as the graph
family $\mathscr{L}_K^{n,d}$ of Lemma~\ref{lem:random graph in good
family}.

Take $G\in \mathscr{L}$ and write $I=I_G^t$ and $L=L_G^t$. By the
definition of $\mathcal{E}^3$ we have $|I|\le \sqrt{n}$, as
required. By the definition of $\mathcal{S}_{\e,\d}$, for every
$S\subseteq V_G$ with $|S|\le n^{2/3}$ we have $|E_G(S)|\le
(1+\d)|S|$. Since $n\ge d^{200}$, by the
definitions~\eqref{eq:choose t delta} we have $n^{2/3}\ge
d(d-1)^{t-1}\sqrt{n}$ and $\d\le 1/(3t-3)$. Since $|\C_t(G)|\le
\sqrt{n}$, this implies that assumption~\eqref{eq:sparsity
assumption edge deletion} of Lemma~\ref{lem:edge deletion distant
cycles} is satisfied, with $r=t$. It therefore follows from
Lemma~\ref{lem:edge deletion distant cycles} that $d_G(C_1,C_2)\ge
t$ for distinct $C_1,C_2\in \C_t(G)$ and
\begin{equation}\label{eq:diam L M}
\diam(L)\stackrel{\eqref{eq:L_G diameter}}{\le} \frac{2t-1}{t+1}\diam(G)+\frac{t(t-2)}{t+1}
\stackrel{\eqref{eq:diam M}\wedge\eqref{eq:choose t delta}}{\le} 3M\log_d n.
\end{equation}
This shows that assertion (b) of Definition~\ref{def:L class} holds
true provided $K\ge 3M$. Since by~\eqref{eq:choose t delta} and the
assumption $n\ge d^{2000M}$ we have $t>8M$, we may also use
Lemma~\ref{lem:expansion deleted} to deduce that
$$
\min_{\substack{S\subseteq G\\0<|S|\le n/2}}\frac{E_{L}(S,V_G\setminus S)}{|S|}\ge \frac{1}{4M}.
$$
Hence, assertion (c) of Definition~\ref{def:L class} holds true
provided $K\ge 4M$.

By the definition of $L$ we have
\begin{equation}\label{eq:grith L is at least t}
\girth(L)\ge t \stackrel{\eqref{eq:choose t delta}\wedge \eqref{eq:diam L M}}{\ge} \frac{\diam(G)}{378M}.
\end{equation}
So, assertion (1) of Definition~\ref{def:good family} holds true
provided $K\ge 378M$.

It remains to prove that if $K$ is large enough then the graph $L$
satisfies assertion (2) of Definition~\ref{def:good family}. Suppose
that $S_0\subseteq \Sigma(V_G)$ satisfies $|S_0|\le \sqrt{n}$. Let
$S\subseteq V_G$ denote the union of $S_0\cap V_G$ with the
endpoints of all the edges in $E_L$ whose corresponding unit
interval in the one-dimensional simplicial complex $\Sigma(L)$
contains a point from $S_0\setminus V_G$. Thus $|S|\le 2\sqrt{n}$.
Apply Lemma~\ref{lem:add geodesics} with $S=T$ and $r=\lfloor
(\log_d n)/36\rfloor$. We obtain $U\subseteq V_G$ with $U\supseteq
S$ such that if we let $H$ denote the graph $(U,E_L(U))$ then
\begin{equation}\label{eq: the diameter of H}
\diam(H)\le 6r+2\diam(L)\stackrel{\eqref{eq:diam L M}}{\le} 4M\log_d n,
\end{equation}
and for every $a,b\in S$
\begin{multline}\label{eq:undistorted subgraph-L}
d_L(a,b)\le d_H(a,b)\le 2\left(\frac{\diam(L)}{r}+1\right)
d_L(a,b)\\ \stackrel{\eqref{eq:diam L M}}{\le}2\left(\frac{3M\log_d
n}{\lfloor (\log_d n)/36\rfloor}+1\right)d_L(a,b) \lesssim M
d_L(a,b).
\end{multline}
Since $U\supseteq S$, and $S$ contains $S_0\cap V_G$ and the
endpoints of all the edges of $\Sigma(L)$ that contain points in
$S_0\setminus V_G$, it follows from~\eqref{eq:undistorted
subgraph-L} that
\begin{equation}\label{eq:on S_0}
\forall\, x,y\in S_0,\qquad d_{\Sigma(L)}(x,y)\le d_{\Sigma(H)}(x,y) \lesssim M d_{\Sigma(L)}(x,y).
\end{equation}

Also, by~\eqref{eq:U size}, and using the fact that $n\ge
d^{2000M}$, we have
$$
|U|\le 2\sqrt{n}\left(n^{1/12}+\diam(L)\right)\stackrel{\eqref{eq:diam L M}}{\le}
2n^{7/12}+2M\sqrt{n}\cdot\log_d n\le n^{2/3}.
$$
Since $G$ belongs to $\mathcal{S}_{\e,\d}$, it follows that $H$ is
$(1+\d)$-sparse. By Corollary~\ref{cor:no girth into L_1} we
conclude that
$$
c_1\left(\Sigma(H)\right)\lesssim 1+\d \diam(H) \stackrel{\eqref{eq:choose t delta}\wedge \eqref{eq: the diameter of H}}{\le} 1+84M.
$$
Due to~\eqref{eq:on S_0} we therefore have
$$
c_1\left(S_0,\frac{2\pi}{\girth(L)}\cdot d_{\Sigma(L)}\right)\lesssim M,
$$
and by Corollary~\ref{cor:not harder to embed cones},
$$
c_1\left(\cone\left(S_0,\frac{2\pi}{\girth(L)}\cdot d_{\Sigma(L)}\right)\right)\lesssim M.
$$
This concludes the proof of assertion (2) of
Definition~\ref{def:good family} provided $K$ is a sufficiently
large multiple of (the universal constant) $M$.
\end{proof}

\subsection{Proof of Proposition~\ref{prop:structure cone
random}}\label{sec:proof of structure cone prop} We shall continue
using here the notation and assumptions that were used in the proof
of Lemma~\ref{lem:random graph in good family}. Suppose that $G\in
\mathscr{L}$. We have already seen that if $G$ is distributed
according to $\mathcal{G}_{n,d}$ then this happens with probability
at least $1-a(d)/\sqrt[3]{n}$ for some $a(d)\in (0,\infty)$.

Recalling the definition of $t$ in~\eqref{eq:choose t delta}, define
$A_1,A_2\subset \Sigma(G)$ as follows.
$$
A_1\eqdef \bigcup_{e\in I_G^t} \left\{x\in \Sigma(G):\ d_{\Sigma(G)}(x,e)\le t\right\},
$$
and
$$
A_2\eqdef \Sigma(G)\setminus \bigcup_{e\in I_G^t} \left\{x\in \Sigma(G):\ d_{\Sigma(G)}(x,e)\le \frac{t}{2}\right\}.
$$
Then $A_1\cup A_2=\Sigma(G)$ and
\begin{equation}\label{eq:gap A1 A2}
d_{\Sigma(G)}\left(A_1\setminus A_2,A_2\setminus A_1\right)\ge \frac{t}{2}\stackrel{\eqref{eq:choose t delta}}{\gtrsim} \log_d n.
\end{equation}

As in the proof of Lemma~\ref{lem:random graph in good family},
denote  $I\eqdef I_{G}^t$ and $L\eqdef L_G^t$. We also let
$\Phi\subseteq \Sigma(G)$ be the union of all the unit intervals
corresponding to the edges in $I$. Thus $A_1$ is the
$t$-neighborhood of $\Phi$ in $\Sigma(G)$ and $A_2$ is the
complements of the $(t/2)$-neighborhood of $\Phi$ in $\Sigma(G)$.

We claim that
\begin{equation}\label{eq:on A_2}
\forall\, x,y\in A_2,\qquad d_{\Sigma(G)}(x,y)\le d_{\Sigma(L)}(x,y)\le 3d_{\Sigma(G)}(x,y).
\end{equation}
Since $E_L\subset E_G$, only the rightmost inequality
in~\eqref{eq:on A_2} requires proof. Fix $x,y\in A_2$ and let
$P_{x,y}:[0,d_{\Sigma(G)}(x,y)]\to \Sigma(G)$ be a geodesic joining
$x$ and $y$ in $\Sigma(G)$. If $P_{x,y}\subset \Sigma(L)$ then
$d_{\Sigma(G)}(x,y)= d_{\Sigma(L)}(x,y)$. The situation is therefore
only interesting when  $P_{x,y}\cap \Phi\neq \emptyset$. In this
case, since the distances of $x$ and $y$ from $\Phi$ are at least
$t/2$, the length of $P_{x,y}$ must be at least $t$. Let
$J_1,\ldots,J_k\subseteq [0,d_{\Sigma(G)}(x,y)]$ be a  maximal
collection of intervals such that $P_{x,y}(J_i)\in I$ for every
$i\in \{1,\ldots,k\}$. We have shown in the proof of
Lemma~\ref{lem:random graph in good family} that (using
Lemma~\ref{lem:edge deletion distant cycles}) we have
$d_{\Sigma(G)}(P_{x,y}(J_i), P_{x,y}(J_j))\ge t$ for every distinct
$i,j\in \{1,\ldots,k\}$. This means that the length of $P_{x,y}$ is
at least $(k-1)t$, or equivalently that $k\le
1+d_{\Sigma(G)}(x,y)/t$. Each edge $P_{x,y}(J_i)$ lies on a cycle
$C_i\in \C_t(G)$, and all the other edges of $C_i$ are in $E_L$.
Therefore, if we replace each edge $P_{x,y}(J_i)$ by the path
$C_i\setminus P_{x,y}(J_i)$ (whose length is at most $t-1$), we will
obtain a path joining $x$ and $y$ in $\Sigma(H)$ of length at most
\begin{multline*}
d_{\Sigma(G)}(x,y)+k(t-1)\le
d_{\Sigma(G)}(x,y)+\left(1+\frac{d_{\Sigma(G)}(x,y)}{t}\right)(t-1)\\\le
2d_{\Sigma(X)}(x,y)+t\le 3d_{\Sigma(X)}(x,y),
\end{multline*}
where we used the fact that $d_{\Sigma(G)}(x,y)\ge t$. This
completes the verification of~\eqref{eq:on A_2}.

The scaling factor $\sigma$ of Proposition~\ref{prop:structure cone
random} will be chosen to be
\begin{equation}\label{eq:choose sigma}
\sigma\eqdef \frac{2\pi}{\girth(L)}\stackrel{\eqref{eq:diam L M}\wedge \eqref{eq:grith L is at least t}}{\asymp} \frac{1}{\log_d n}.
\end{equation}
with this choice, due to~\eqref{eq:gap A1 A2} assertion $(II)$ of
Proposition~\ref{prop:structure cone random} holds true. We have
already proved in Lemma~\ref{lem:random graph in good family} that
$L\in \F_K$, and therefore $\cone(\Sigma(L),\sigma d_{\Sigma(L)})$
is isometric to a subset of $\X_K$. By~\eqref{eq:on A_2} and
Fact~\ref{fact:easy} we therefore have
$$
c_{\X_K}\left(\cone(A_2,\sigma
d_{\Sigma(G)})\right)\le c_{\cone(\Sigma(L),\sigma
d_{\Sigma(L)})}\left(\cone(A_2,\sigma
d_{\Sigma(G)})\right)\le 3,
$$
thus proving assertion $(IV)$ of Proposition~\ref{prop:structure
cone random}.

To prove assertion $(V)$ of Proposition~\ref{prop:structure cone
random} we define an embedding $f:\Sigma(G)\to
\cone\left(\Sigma(G)\right)$ by $f(x)=(1/\sqrt{2}, x)$. For every
$x,y\in \Sigma(G)$,
\begin{multline*}
d_{\cone\left(\Sigma(G),\sigma
d_{\Sigma(G)}\right)}(f(x),f(y))\stackrel{\eqref{eq:def cone}}{=}
\sqrt{1-\cos\left(\min\{\pi,\sigma d_{\Sigma(G)}(x,y)\}\right)}\\
\stackrel{\eqref{eq:cosine inequalities}\wedge \eqref{eq:choose
sigma}}{\asymp}\min\left\{\pi,\frac{d_{\Sigma(G)}(x,y)}{\log_d
n}\right\}\stackrel{\eqref{eq:diam M}}{\asymp} d_{\Sigma(G)}(x,y).
\end{multline*}

All that remains is to prove assertion $(III)$ of
Proposition~\ref{prop:structure cone random}. Define $S=A_1\cap V_G$
and $T=\Phi\cap V_G=\bigcup_{e\in I} e$. Thus $|T|\le 2|I|\le
2\sqrt{n}$. By the definition of $A_1$, condition~\eqref{eq:S in T
neighborhood} of Lemma~\ref{lem:add geodesics} holds true with
$r=t$. Consequently, by Lemma~\ref{lem:add geodesics} there exists
$U\subset V_G$ with $U\supseteq S$ such that if we let $H$ denote
the graph $(U,E_G(U))$ then  for every $a,b\in S$,

\begin{equation*}
d_G(a,b)\le d_H(a,b)\le 2\left(\frac{\diam(G)}{t}+1\right) d_G(a,b)\stackrel{\eqref{eq:diam M}
\wedge\eqref{eq:choose t delta}}{\lesssim} d_G(a,b).
\end{equation*}
Since $U\supseteq S$ and $S=A_1\cap V_G$, it follows that
\begin{equation}\label{eq:equivalence GH}
\forall\, x,y\in A_2,\qquad d_{\Sigma(G)}(x,y)\asymp d_{\Sigma(H)}(x,y).
\end{equation}
Moreover, using the assumption $n\ge d^{2000M}$, it follows
from~\eqref{eq:U size} that
$$
|U|\le |T|\left(d^{3t}+\diam(G)\right)\stackrel{\eqref{eq:diam M}
\wedge\eqref{eq:choose t delta}}{\le} 2\sqrt{n}\left(n^{1/21}+M\log_d n\right)\le n^{2/3}.
$$
recalling that $G$ belongs to $\mathcal{S}_{\e,\d}$, we conclude
that $H$ is $(1+\d)$-sparse, and therefore by Corollary~\ref{cor:no
girth into L_1},
$$
c_1\left(A_2,d_{\Sigma(G)}\right)\stackrel{\eqref{eq:equivalence GH}}{\lesssim} c_1\left(\Sigma(H),d_{\Sigma(H)}\right)\lesssim 1+\d \diam(H) \stackrel{\eqref{eq:choose t delta}
\wedge \eqref{eq: the diameter of H}}{\lesssim} 1.
$$
Now assertion $(III)$ of Proposition~\ref{prop:structure cone
random} follows by Corollary~\ref{cor:not harder to embed cones}.
\qed

\section{Are two stochastically independent random graphs expanders with respect to each other?}\label{sec:KF}

Here we present partial progress towards (a positive solution of)
Question~\ref{Q:kleinberg}. Proposition~\ref{thm:uri} below is based
on ideas of U. Feige.

\begin{proposition}\label{thm:uri}
For $n\in \N$ even let $G,H$ be two i.i.d. random graphs sampled from
$\mathcal{G}_{n,3}$. Then with probability that tends to $1$ as
$n\to \infty$, for every permutation $\pi \in S_n$ we have
\begin{equation}\label{eq:permutation poincare}
\frac{1}{n^2}\sum_{i=1}^n\sum_{j=1}^n d_H(\pi(i),\pi(j))^2\lesssim
\frac{1}{n}\sum_{\{i,j\}\in E_G}d_H(\pi(i),\pi(j))^2.
\end{equation}
\end{proposition}

A positive answer to Question~\ref{Q:kleinberg} would require
proving~\eqref{eq:permutation poincare} when $G,H$ are independent
random $3$-regular graphs of possibly different cardinalities, say
$G$ sampled from $\mathcal{G}_{n,3}$ and $H$ sampled from
$\mathcal{G}_{m,3}$, and with the permutation $\pi$ of
Proposition~\ref{thm:uri} replaced with a general mapping
$f:\{1,\ldots,n\}\to \{1,\ldots,m\}$. Note that for fixed $m\in \N$
we have $\gamma_+(G,d_H^2)\lesssim (\log m)^2$ asymptotically almost
surely. Indeed, $G$ is asymptotically almost surely an expander, and
by Bourgain's embedding theorem~\cite{Bourgain-embed}, the metric
space $(\{1,\ldots,m\},d_H)$ embeds into $\ell_2$ with bi-Lipschitz
distortion $O(\log m)$. When $m>n$ one can show that asymptotically
almost surely
\begin{equation}\label{eq:kleinberg problem for large m}
\gamma_+(G,d_H)^2\lesssim \max\left\{1,\left(\frac{\log n}{\log(m/n)}\right)^2\right\}.
\end{equation}
The validity of~\eqref{eq:kleinberg problem for large m} follows
from the fact that $G$ is asymptotically almost surely an expander,
combined with an application of Lemma~\ref{lem:random graphs are
sparse} and Corollary~\ref{cor:no girth into L_1} to the random
graph $H$.

Below we shall use the following version of Chernoff's bound (see
e.g~\cite[Thm.~A.1.12]{AS}). Suppose that $X_1,\ldots,X_n$ are
i.i.d. random variables taking values in $\{0,1\}$ and write
$\Pr[X_i=1]=p$. Then for every $\beta\in (1,\infty)$ we have
\begin{equation}\label{eq:chrnoff}
\Pr\left[\sum_{i=1}^n X_i>\beta p n\right]<\left(\frac{e^{\beta-1}}{\beta^\beta}\right)^{pn}.
\end{equation}

\begin{proof}[Proof of Proposition~\ref{thm:uri}]
By~\cite{BV} with probability that tends to $1$ as $n\to \infty$ we
have $\diam(H)\lesssim \log n$. The left hand side
of~\eqref{eq:permutation poincare} therefore asymptotically almost
surely satisfies
$$
\frac{1}{n^2}\sum_{i=1}^n\sum_{j=1}^n d_H(\pi(i),\pi(j))^2\lesssim (\log n)^2.
$$
Since there are $n!$ possible permutations $\pi\in S_n$, it suffices
to prove that there exists a universal constant $c\in (0,\infty)$
such that for every fixed permutation $\pi\in S_n$ we have
\begin{multline}\label{eq:perm version at least log}
\mathcal{G}_{n,3}\times \mathcal{G}_{n,3}\left(\left\{G,H\in \G_n:\
\frac{1}{n}\sum_{\{i,j\}\in E_G} d_H(\pi(i),\pi(j))^2\ge c(\log
n)^2\right\}\right)\\\ge 1-\frac{o(1)}{n!}.
\end{multline}

Because the distribution of the metric $d_H$ is invariant under
permutations, it suffices to prove~\eqref{eq:perm version at least
log} when $\pi$ is the identity permutation. To this end, given a
$3$-regular graph $H\in \G_n$, consider the following subset of the
unordered pairs of elements of $\{1,\ldots,n\}$.
$$
N_H\eqdef \left\{\{i,j\}\in \binom{\n}{2}:\ d_H(i,j)\le \frac{\log n}{16}\right\}.
$$
Since $H$ is $3$-regular we have $|N_H|\le\frac32 n^{17/16}$.

In order to prove~\eqref{eq:perm version at least log} (with
$c=1/16$ and $\pi$ the identity mapping) it suffice to prove that
for every fixed $3$-regular graph $H\in \G_n$ we have
\begin{equation}\label{eq:goal with N notation}
\mathcal{G}_{n,3}\left(\left\{G\in \G_n:\ |E_G\cap N_H|>\frac{4n}{3}\right\}\right)\le \frac{o(1)}{n!}.
\end{equation}
By~\eqref{eq:continguous} it suffices to prove~\eqref{eq:goal with N
notation} in the pairing model, i.e.,
\begin{equation}\label{eq:goal with N notation-pairing}
\mathcal{P}_{n,3}\left(\left\{G\in \G_n:\ |E_G\cap N_H|>\frac{4n}{3}\right\}\right)\le \frac{o(1)}{n!}.
\end{equation}

Assume from now on that $n$ is divisible by $4$. The proof for
general even $n$ follows mutatis mutandis from the argument below, the
only difference being that one needs to round certain numbers to
their nearest integers.

Recall that in the pairing model $\mathcal{P}_{n,3}$ we choose a
matching $M$ of $P=\n\times \{1,2,3\}$ uniformly at random, and
``project it" onto $\{1,\ldots,n\}$ so as to get a random
$3$-regular graph $G(M)$. It will be convenient to order the $\binom{3n}{2}$ pairs in $\binom{P}{2}$ arbitrarily. Once this is done, the uniformly random
matching $M$ can be obtained by choosing its edges sequentially,
where at each step we choose an additional edge uniformly at random
from the unused pairs.  Fix a $3$-regular graph $H\in \G_n$ and let
$Y_i^H$ be the $\{0,1\}$-valued random variable that takes the value
$1$ if and only if $i$th edge is in $N_H$. Note that $Y_i^H$ is
obtained by flipping a biased coin with success probability $p_i$
which is itself a random variable depending on the first $i-1$ edges
of $M$. However, for $i\le 5n/4$ we have
$$
p_i\le \frac{|N_H|}{\binom{3n-2(i-1)}{2}}\le \frac{2|N_H|}{(3n-2i)^2}\le \frac{8|N_H|}{n^2}\le \frac{12n^{17/16}}{n^2}=\frac{12}{n^{15/16}}\eqdef p.
$$

We proceed by a standard coupling argument. Having sampled
$Y_1^H,\ldots,Y_{5n/4}^H$, we sample a sequence
$X_1,\ldots,X_{5n/4}$ of $\{0,1\}$-valued random variables as
follows. If $Y_i^H=1$ then $X_i=1$, and if $Y_i^H=0$ then $X_i$ is
obtained by tossing an independent coin with success probability
$(p-p_i)/(1-p_i)$.
Then $\Pr[X_i=1]=p$, and because the coins that we used are
independent, $X_1,\ldots,X_{5n/4}$ are independent random variables.
Moreover, we have the point-wise inequality $X_i\ge Y_i^H$ for every
$i\le 5n/4$. Now,
\begin{multline*}
 \mathcal{P}_{n,3}
\left(\left\{G\in \G_n:\ |E_G\cap
N_H|>\frac{4n}{3}\right\}\right)=\Pr\left[\sum_{i=1}^{3n/2}
Y_i^H>\frac{4n}{3}\right]\\\le  \Pr\left[\sum_{i=1}^{5n/4}
Y_i^H>\frac{13n}{12}\right]\le \Pr\left[\sum_{i=1}^{5n/4}
X_i^H>\frac{13n}{12}\right].
\end{multline*}
Hence by~\eqref{eq:chrnoff} with $\beta=\frac{13n^{15/16}}{180}$, we
obtain the bound
\begin{equation*}
 \mathcal{P}_{n,3}
\left(\left\{G\in \G_n:\ |E_G\cap
N_H|>\frac{4n}{3}\right\}\right)\le
\left(\frac{180e}{13}\right)^{\frac{13n}{12}}
n^{-\frac{65}{64}n}=\frac{o(1)}{n!},
\end{equation*}
implying the desired estimate~\eqref{eq:goal with N
notation-pairing}.
\end{proof}

\medskip

\noindent{\bf Acknowledgements.}  We are grateful to Uriel Feige,
Konstantin Makarychev and Yuval Peres for helpful discussions, and to Jon Kleinberg
for formulating Question~\ref{Q:kleinberg}. We also thank Takefumi
Kondo for sending us a preliminary version of~\cite{Kondo}. M.~M.
was supported by ISF grants 221/07 and 93/11, BSF grant 2010021, and
NSF grants CCF-0832797 and DMS-0835373. Part of this work was
completed while M.~M. was visitor at Microsoft Research and a member of the Institute for Advanced
Study at Princeton. A.~N. was supported by NSF grant CCF-0832795,
BSF grant 2010021, the Packard Foundation and the Simons Foundation.
Part of this work was completed while A.~N. was a Visiting Fellow at
Princeton University.

\bibliographystyle{alphaabbrvprelim}
\bibliography{zigzag-expander}

 \end{document}